\documentclass[12pt]{amsart}

  \usepackage{fullpage}
    \usepackage{amssymb}
 \usepackage[mathscr]{euscript}
\usepackage[all]{xy}
   \usepackage{graphicx}
 \usepackage{epic}
 \usepackage{enumerate}

 \usepackage{graphicx}
 \usepackage{epstopdf}
     \usepackage{amssymb}
  \usepackage{color} 
 \usepackage{amsthm,graphpap}

    \newcommand\Id{\mathop{\rm Id}\nolimits}

 \let\subest\subset

\long\def\symbolfootnote[#1]#2{\begingroup%
\def\thefootnote{\fnsymbol{footnote}}\footnote[#1]{#2}\endgroup}

\newcommand\CD{\check{D}}

  \newcommand\Proj{\mathop{\rm Proj}\nolimits}
    
        \newcommand\reg{\mathop{\rm reg}\nolimits}

\newcommand\fbul{F^\bullet}

\newcommand\bsm{ \begin{smallmatrix}}
\newcommand\bspm{ \left(\begin{smallmatrix}}
\newcommand\esm{\end{smallmatrix} }
\newcommand\espm{\end{smallmatrix} \right)}
\newcommand\bbm{\left[\begin{smallmatrix}}
\newcommand\ebm{\end{smallmatrix}\right]}
\newcommand\bcs{\begin{cases}}
\newcommand\ecs{\end{cases}}
\newcommand\wbul{{W_\bullet}}
 \newcommand{\lra}[1]{\left\langle#1\right\rangle}
\newcommand{\lrc}[1]{\left\{ #1\right\}}

\newcommand{\lrp}[1]{\left(#1\right)}
 \newcommand\bp[1]{\bigl(#1\bigr)}
  \newcommand\Bp[1]{\Bigl(#1\Bigr)}

 \newcommand\wt[1]{\widetilde{#1}}

 \newcommand\V{\mathbb{V}}

  \newcommand\sideeq{\rotatebox{90}{$=$}}
    \newcommand\sidein{\rotatebox{90}{$\in$}}
 \newcommand\ssn[1]{\medbreak\noindent {\bf #1}}

\newcommand\hh{\mathfrak{h}}

\newcommand\mm{\mathfrak{m}}

\newcommand{\C}{{\mathbb{C}}}

\newcommand{\G}{\mathbb{G}}

\renewcommand{\P}{\mathbb{P}}
\newcommand{\Q}{{\mathbb{Q}}}

\newcommand{\R}{\mathbb{R}}

\newcommand{\Z}{\mathbb{Z}}

\newcommand{\cC}{{\mathscr{C}}}

\newcommand{\cE}{{\mathscr{E}}}
\newcommand{\cF}{{\mathscr{F}}}
\newcommand{\cG}{\mathfrak{g}}

\newcommand{\cH}{{\mathscr{H}}}
\newcommand{\cI}{{\mathscr{I}}}

\newcommand{\cL}{{\mathscr{L}}}

\newcommand{\cM}{{\mathscr{M}}}

\newcommand{\cO}{{\mathscr{O}}}

\newcommand{\cU}{{\mathscr{U}}}

\newcommand\wtcx{\widetilde{\cX}}

\newcommand{\cX}{{\mathscr{X}}}

\newcommand{\cZ}{{\mathscr{Z}}}

\newcommand\bpm{\begin{pmatrix}}
\newcommand\epm{\end{pmatrix}}

\newcommand\phs{polarized Hodge structure}
\newcommand\mhs{mixed Hodge structure}

\newcommand\lmhs{limiting mixed Hodge structure}

\newcommand\hn{Hodge number}
\newcommand\hb{Hodge bundle}

\newcommand\HR{Hodge-Riemann}

\newcommand\hs{Hodge struc\-ture}

\renewcommand\mod{\mathop{\rm mod}\nolimits}
\newcommand\Ad{\mathop{\rm Ad}\nolimits}

\newcommand\rank{\mathop{\rm rank}\nolimits}

\newcommand\rsl{\mathop{\rm sl}\nolimits}

 \newcommand{\Gr}{\mathop{\rm Gr}\nolimits}

 \newcommand\Aut{\mathop{\rm Aut}\nolimits}
\newcommand\End{\mathop{\rm End}\nolimits}

\newcommand\rim{\mathop{\rm Im}\nolimits}

\newcommand\sing{{\mathop{\rm sing}\nolimits}}
\newcommand\Sym{\mathop{\rm Sym}\nolimits}
\newcommand\id{\mathop{\rm id}\nolimits}
\newcommand\prim{{\rm prim}}

\newcommand\diag{\mathop{\rm diag}\nolimits}

\newcommand{\Tr}{\mathop{\rm Tr}\nolimits}
 \newcommand{\Hom}{\mathop{\rm Hom}\nolimits}

\newcommand\ad{\mathop{\rm ad}\nolimits}

\newcommand{\AJ}{{\rm AJ}}

 \renewcommand{\part}{\partial}
\newcommand{\la}{{\lambda}}
\newcommand\ga{{\gamma}}
\newcommand\Ga{{\Gamma}}
\newcommand{\La}{{\Lambda}}
\newcommand{\Om}{{\Omega}}
\newcommand{\om}{{\omega}}
\newcommand\Sig{{\Sigma}}

\newcommand\vp{\varphi}
\newcommand\sig{\sigma}

\newcommand{\chs}{\hskip-10pt}

\newcommand\bsl{\backslash}
\newcommand\eps{\epsilon}

\newcommand{\lab}{\label}
\newcommand\pref[1]{{\rm (\ref{#1})}}

\newcommand{\hensp}[1]{\enspace\hbox{#1}\enspace}

\newcommand{\opplus}{\mathop{\oplus}\limits}
\newcommand{\ottimes}{\mathop{\otimes}\limits}

  \renewcommand\thesection{\arabic{section}}

\newcounter{demo}[equation]

 \newtheoremstyle{mytheo}
 {3pt}
  {3pt}
  {\itshape}
  {}
  {\scshape}
  {:}
  {.5em}
  {}

\theoremstyle{mytheo}
\newtheorem*{cor}{Corollary}
\newtheorem{Cor}[equation]{Corollary}

\newtheorem{Thm}[equation]{Theorem}
  
  \newtheorem{Conj}[equation]{Conjecture}

\newtheorem{Lem}[equation]{Lemma}

\newtheorem{claim}[equation]{Claim}

 \newtheorem*{quest}{Question}

\newtheorem{Prop}[equation]{Proposition}

 \newtheorem{Concls}[equation]{Conclusions}
\newtheoremstyle{note}
  {3pt}
  {3pt}
  {}
  {}
  {\bfseries}
  {:}
  {.5em}
  {}
\theoremstyle{note}

\newtheorem{Defin}[equation]{Definition}

\newtheorem*{defin}{Definition}

 \newtheorem*{exam}{Example}

\newtheorem{Exam}[equation]{Example}
\newtheorem{example}{Example}

\newtheorem*{rem}{Remark}

\newtheorem{Rem}[equation]{Remark}

\theoremstyle{remark}

\newcommand\simto{\xrightarrow{\sim}}
\newcommand\xri[1]{\xrightarrow{#1}}

\newcommand\ol[1]{\overline{#1}}

\newcommand\olb{{\overline{\beta}}}

 \usepackage[mathscr]{euscript}
\input{xy}
\xyoption{all}
  \renewcommand\thesection{\arabic{section}}

 \numberwithin{equation}{subsection}
 \newcommand{\oltm}{\wt{\ol{M}}}
 \newcommand\vspth{\vspace*{-3pt}}
 \newcommand\cope{\cO_{\P E}}
 \newcommand\loc{{\rm loc}}
 \newcommand\met{{\rm met}}
  \newcommand\num{{\rm num}}
   \newcommand\gr{{\rm gr}}
\newcommand\mmax{\mathop{\rm max}\limits_m}
\newcommand\rk{\mathop{\rm rk}\nolimits}

\newcommand\hlb{Hodge line bundle}
\newcommand\lb{line bundle}
\newcommand\vb{vector bundle}
\newcommand\hvb{Hodge vector bundle}
\newcommand{\Var}{\mathop{\rm Var}\nolimits}

\newcommand\nproof{\subsubsection*{Proof}}

 \renewcommand\thesection{\Roman{section}}

  \newcommand\ssni[1]{\medbreak\noindent {\it #1\/}:}
  
  \newcommand\bmp[2]{\hbox{\begin{minipage}[t]{#1in}  #2 \end{minipage}}}
   \newcommand\bmpc[2]{\hbox{\begin{minipage}[c]{#1in}  #2 \end{minipage}}}

  \theoremstyle{mytheo}

\newtheorem{prop}[equation]{Proposition}

\begin{document}
 \  \title[Positivity of  vector bundles and Hodge theory]{Positivity of  vector bundles and Hodge theory} 
 \author{Mark Green and  Phillip Griffiths}
 \thanks{These are   notes prepared by the authors and  largely based on joint work in progress with  Radu Laza and Colleeen Robles (cf.\ \cite{GGLR17}).}
\begin{abstract}
It is well known that positivity properties of the curvature of a \vb\ have implications on the algebro-geometric properties of the bundle, such as numerical positivity, vanishing of higher cohomology leading to  existence of global sections etc.  It is also well known that bundles arising in Hodge theory tend to have positivity properties.  From these considerations several issues arise:
\begin{enumerate}[{\rm (i)}]
\item For a bundle  that  is semi-positive but not strictly positive; what further natural conditions lead to the existence of sections of its symmetric powers?
\item In Hodge theory the Hodge metrics generally have singularities; what can be said about these and their curvatures, Chern forms etc.?
\item What are some algebro-geometric applications of positivity of \hb s?
\end{enumerate} 

 The purpose of these partly expository notes is fourfold.  One is to summarize some of the general measures and types of positivity that have arisen in the literature.  A second is to introduce and give some applications of   \emph{norm positivity}.  This is a concept that implies the different notions of metric semi-positivity  that are present in many of  the standard examples and one that has an algebro-geometric interpretation in these examples.  A third purpose is to discuss and compare some of the types of metric singularities that arise in algebraic geometry and in Hodge theory.  Finally we shall present some applications of the theory from both the classical and recent literature.
\end{abstract}

 \maketitle
 \vspace*{-24pt}
  \section*{Outline}
  \begin{itemize}
  \item[I.] \emph{Introduction and notation and terminology}
  \begin{itemize}
  \item[A.] Introduction
  \item[B.] General notations and terminology
  \item[C.] Notations and terminology from Hodge theory
  \end{itemize}
  \item[II.] \emph{Measures and types of positivity}
  \begin{itemize}
  \item[A.] Kodaira-Iitaka dimension
  \item[B.] Metric positivity
  \item[C.] Interpretation of the curvature form
  \item[D.] Numerical positivity
  \item[E.] Numerical dimension
  \item[F.] Tangent bundle
  \item[G.] Standard  implications
  \item[H.] A further result
  \end{itemize}
  \item[III.] \emph{Norm positivity}
  \begin{itemize}
  \item[A.] Definition and first properties
  \item[B.] A result using norm positivity  
  \end{itemize}
  \item[IV.] \emph{Singularities}
  \begin{itemize}
  \item[A.] Analytic singularities
  \item[B.] Logarithmic and mild singularities
 
  \end{itemize}
  \item[V.]  {\em Proof of Theorem \ref{4b8}}
  \begin{itemize}
  \item[A.] Reformulation of the result
  \item[B.] Weight filtrations, representations of $\rsl_2$ and \lmhs s
  \item[C.] Calculation of the Chern forms $\Om$ and $\Om_I$\end{itemize}

  \item[VI.]  \emph{Applications, further results and some open questions}
  \begin{itemize}
  \item[A.] The  Satake-Baily-Borel  completion of period mappings
  \item[B.] Norm positivity and the cotangent bundle to the image of a period mapping
  \item[C.] The Iitaka conjecture 
  \item[D.] The Hodge vector bundle may detect extension data  
  \item[E.] The exterior differential system defined by a Chern form\end{itemize}   \end{itemize}
  
  \section{Introduction and notation and terminology}
  \subsection{Introduction}  \hfill
  
  The general purpose of these notes is to give an   account of some aspects and applications  of the concept of positivity of holomorphic vector bundles, especially   those that  appear in Hodge theory.  The  applications use the positivity of Hodge \lb\ (Theorem \ref{1a7}), the semi-positivity of the cotangent bundle to the image of a period mapping (Theorem \ref{6b1}), and the semi-positivity of the Hodge \vb\ (Theorem \ref{1a14}).  Also discussed is the numerical positivity of the Hodge bundles.
  
  The overall subject  of positivity is one   in which there is an extensive and rich literature and   in which there is currently  active research occurring with interesting results appearing regularly.  We shall only discuss a few particular topics  and shall use \cite{Dem12a} as our  reference for general background material concerning positivity in complex analytic geometry as well as an account of some recent work, and refer to \cite{Pau16} as a source for a summary of some  current research and as an overall  guide to the  more recent literature.  The recent paper \cite{Den18} also contains an extensive bibliography.
    For Hodge theory we shall use  \cite{PeSt08} and \cite{CM-SP} as our main general references.
  
  Following this introduction and the establishment of general notations and terminology in Sections I.B and I.C, in Section II we shall give a synopsis of some of the standard measures and types of positivity of holomorphic \vb s.  In Section II.H we present the further result \pref{1a5} below.  It is this result that enables us to largely replace the cohomological notion of weak positivity in the sense of Viehweg \cite{Vie83a}, \cite{Vie83b} with purely differential geometric considerations.
  
One   purpose of these notes is to give an informal account of   some of the   results in \cite{GGLR17} and to discuss   topics that are directly related to and/or grew out of that work.  One application to algebraic geometry of   Theorems 1.2.2 and   1.3.10 there is the following:\footnote{Significant amplification and complete proofs of the results in \cite{GGLR17} are currently being prepared.  Below we shall give a short algebro-geometric argument for the case when $\dim B=2$.} 
  \begin{equation}\lab{1a1}
  \bmp{5}{\em Let $\cM$ be the KSBA moduli space  for algebraic surfaces $X$ of general type and with given $p_g(X), q(X)$ and $K^2_X$.\footnotemark\/ Then the \hlb\ $\La_e\to \ol\cM$ is defined  over the canonical completion $\ol\cM$ of $\cM$.  Moreover $\Proj(\La_e)$ exists and defines the Satake-Baily-Borel completion $\ol M$ of the image of $M=\Phi(\cM)$ of the period mapping $\Phi:\cM\to\Ga\bsl D$.\footnotemark}\end{equation}
  \setcounter{footnote}{2}
  \footnotetext{We shall use \cite{Kol13} as our general reference for the topic of moduli.} 
  \setcounter{footnote}{3}
  \footnotetext{The definition of a Satake-Baily-Borel completion will be explained below. Here we are mainly considering that part of the period mapping that arises from the \phs\ on $H^2(X)$.  When $p_g(X)\geqq 2$ we are in a non-classical situation, and the new phenomena that arise in this case are the principal aspect of interest.}

\noindent The period mapping extends to
  \[
  \Phi_e :\ol\cM\to\ol M\]
  and set-theoretically the image of the boundary $\part\cM=\ol\cM\bsl \cM$ consists of the associated graded to the limiting mixed structures along the boundary strata of $\part \cM$.  As has been found in applying this result to the surfaces analyzed in \cite{FPR15a}, \cite{FPR15b}, \cite{FPR15c} and to similar surfaces studied in our work and in discussions with the authors of those papers, the extended period mapping   may serve as an  effective method for organizing and understanding the boundary structure of $\ol\cM$ and in suggesting how to desingularize it.\footnote{In contrast to the case of curves, $\ol\cM$ seems almost never to be smooth even when $\cM$ is.}
  
  The construction of $\ol M$ is general and in brief outline proceeds as follows:
  
  (i) One begins with a variation of Hodge structure over a smooth quasi-projective variety $B$ given by a period mapping
  \begin{equation}\lab{1a2n}
  \Phi:B\to\Ga\bsl D.\end{equation}
  Here $B$ has a smooth completion $\ol B$ where $Z= \ol B\bsl B$ is a reduced normal crossing divisor with irreducible component $Z_i$ around which the local monodromies $T_i$ are assumed to be unipotent.  We may also assume that $\Phi$ has been extended across any $Z_i$ for which $T_i$ are the identity; then $\Phi$ is proper and the image $\Phi(B)=M\subset \Ga\bsl D$ is a closed analytic subvariety (\cite{Som78}).
  
  (ii)  Along the smooth points $Z^\ast_I$ of the strata $Z_I:= \bigcap_{i\in I}Z_i$ of $Z$ there is a \lmhs\ (\cite{CKS86}), and passing to the associated graded gives a variation of Hodge structure
  \[
  \Phi_I:Z^\ast_I\to\Ga_I\bsl D_I.\]
  Extending $\Phi_I$ across the boundary components of $Z^\ast_I$ around which the monodromy is finite, from the image of the extension of $\Phi_I$ we obtain a complex analytic variety~$M_I$.  Then as a set
  \[
  \ol M=M\cup \Bp{\bigcup_I M_I},\]
  where on the right-hand side there are identifications made corresponding to strata $Z^\ast_I$, $Z^\ast_J$ where $Z_I\cap Z_J\ne \emptyset$ and where $\Phi_I,\Phi_{J}$ have both been  extended across intersection points.
  
  (iii) This defines $\ol M$ as a set, and it is not difficult to show that the resulting $\ol M$ has the structure of a compact Hausdorff topological space and that the extended period mapping
  \begin{equation}\lab{1a3n}
  \Phi_e:\ol B\to\ol M\end{equation}
  is proper with compact analytic subvarieties $F_x:= \Phi^{-1}_e(x)$, $x\in \ol M$, as fibres.  If we define
  \[
  \cO_{\ol M,x}= \left\{
  \bmpc{3}{ring of functions $f$ that are continuous in a neighborhood  $\cU$ of $x\in \ol M$ and which are holomorphic in $\Phi^{-1}_e(\cU)$ and constant on fibres of $\Phi_e$}\right\}\]
  then the issue is to show that \emph{there are enough functions in $\cO_{\ol M,x}$ to define the structure of a complex analytic variety on $\ol M$.}
  
  (iv) Using \cite{CKS86} the \emph{local} structure of $\Phi_e$ along $F_x\subset \Phi^{-1}_e(\cU)$ can be analyzed; this will be briefly recounted in Section V below, the main point being to use implications of the relative weight filtration property (RWFP) of \lmhs s.
  
  (v) This leaves the issue of the \emph{global} structure of $\Phi_e$ along $F_x$.  We note that there is a map
  \[
  \mm_x/ \mm^2_x\to H^0\lrp{N^\ast_{F_x/\ol B}}\]
  where $\mm_x\subset \cO_{\ol M,x}$ is the maximal ideal and $N^\ast_{F_x/\ol B} \to F_x$ is the co-normal bundle of $F_x$ in $\ol B$.  Thus one expects some ``positivity" of $N^\ast_{F_x/\ol B}$.  From $F_x\subset Z^\ast_{I,e}\subset \ol B$ we obtain
  \[
  0\to N^\ast_{Z^\ast_{I,e}/\ol B}\big|_{F_x}\to N^\ast_{F_x/\ol B}\to N^\ast_{F_x/Z^\ast_{I,e}} \to 0.\]
  Since $\Phi_I :Z^\ast_{I,e}\to \Ga_I\bsl D_I$ is defined and maps $F_x$ to a point, we obtain some positivity of $N^\ast_{F_x/Z^\ast_{I,e}}$.
    Denoting by $\mm_{I,x}$ the maximal ideal in $\cO_{\Ga_I\bsl D_{I,x}}$ at $x = \Phi_I(F_x)$, from the local knowledge of $\Phi_e$ along $F_x$ we are able to infer that the positivity of $N^\ast_{F_x/Z^\ast_{I,e}}$ that arises from $\mm_{I,x}/\mm^2_{I,x}$ lifts to sections in $H^0\lrp{N^\ast_{F_x/\ol B}}$; thus the main issue is that of the positivity of $N^\ast_{Z^\ast_{I,e}/\ol B} \big|_{F_x}$.
  
  (vi) Here, as will be explained in the revised and expanded version of  \cite{GGLR17}, some apparently new Hodge theoretic considerations arise.  In the literature, and in this work, there have been numerous consequences drawn from the positivity properties of the Hodge line bundle, and also from some of the positivity properties of the cotangent bundle of the smooth points of the  image $M=\Phi(\ol B)$ of a period mapping.  However we are not aware of similar applications of positivity properties of  the bundles constructed from the \emph{extension data}  $\cE$ associated to a \lmhs.  As will be explained in loc. cit., there is an ample line bundle $\cL\to \cE$ and an isomorphism
    \begin{equation}\lab{1a4n}
 \nu^\ast\cL\cong  N^\ast_{Z^\ast_{I,e/\ol B} \big|_{F_x}}\end{equation}
  where
  \begin{equation}\lab{1a5n}
  \nu:F_x\to\cE\end{equation}
  associates to each point of $F_x$ the extension data associated to the LMHS at that point.  Put another way, assuming $\Phi_\ast$ is generically injective, since the associated graded to the \lmhs\ is constant along $F_x$ one may expect non-trivial variation in the extension data to the LMHS's, i.e.\ in  $\nu$ in \pref{1a5n} (or possibly some jet of $\nu$) should be non-constant.  The map \pref{1a4n} then converts sections of the ample bundle $\cL\to\cE$ into sections of $N^\ast_{F_x/\ol B}.$\footnote{In the classical case of curves this extension data is given by the Jacobian variety of a smooth, but generally reducible, curve and $\cL\to\cE$ is the ``theta" line bundle given by the polarization.  It is interesting and we think noteworthy that constructions akin to classical theta functions can be given in non-classical situations.}

Returning to the discussion of \pref{1a2n}, a 
   central ingredient in its proof consists of the positivity properties of the Hodge line bundle $\La_e$.  Referring to \cite{GGLR17} and to Sections IV.B, V, and VI.A below for details, the Chern form $\om_e$ of $\La_e$  is a singular differential form on a desingularization $\ol B$ of $\ol \cM$, and the exterior differential system
 \begin{equation}\lab{1a2}
 \om_e=0\end{equation}
 defines a complex analytic fibration whose quotient captures the \phs\ on $H^2(X)$ when $X$ is smooth (or has canonical singularities), and when $X=\lim X_t$ is a specialization of smooth $X_t$'s it captures the associated graded to the \lmhs.\footnote{Of course one must make sense of \pref{1a2} where $\om_e$ has singularities; this is addressed in Section VI.E below.}
As discussed above it is   the extension data in the LMHS that is not detected  by the extended period mapping.  In this way the Satake-Baily-Borel completion is   a  minimal completion of the  image of the period mapping.
  
  The positivity of the \hlb\ raises naturally the issue of the degree  of positivity of the \hvb.  It is this question that provides some  of the background motivation for these notes.  A \vb\ $E\to X$ over a complex manifold will be said to be \emph{semi-positive} if it has a Hermitian metric whose curvature form
  \begin{equation}\lab{1a3}
  \Theta_E(e,\xi)\geqq 0\end{equation}
  (cf.\ Sections II.A and II.B for notations).    We shall usually write this condition as $\Theta_E\geqq 0$.  The \hvb\ is semi-positive, but the curvature form is generally not   positive,    even at a general $x\in X$ and  general $e\in E_x$.  An interesting algebro-geometric question is: \emph{How positive is it?}
  
  We shall say that a general bundle  $E\to X$ is \emph{strongly semi-positive} if   there is a metric where \pref{1a3} is satisfied and where
  \begin{equation}\lab{1a4}
  \Tr \Theta_E = \Theta_{\det E}>0\end{equation}
  on an open set.  For example, the \hvb\ is strongly semi-positive if the end piece $\Phi_{\ast,n}$ of the differential of the period mapping is injective at a general point, a condition that  is frequently satisfied in practice.\footnote{For families of algebraic curves and surfaces this condition is the same as the differential being injective.  In general, using the augmented Hodge bundle there is a similar interpretation using the full differential $\Phi_\ast$.}
    One of the main observations in these notes is
  \begin{equation}\lab{1a5}\bmp{5}{\em If  $X$ is compact and $E\to X$ is strongly semi-positive, then for some $r_0\leqq \rank E$ we have for $m\geqq r_0$
  \[
  \Theta_{\Sym^m E}>0\]
  on an open set.}\end{equation}
As a corollary,   $$ \Sym^m E \hensp{\emph{is big for}} m\geqq r_0.$$
  As the proof of \pref{1a5} will show, although $\Sym^mE$ ``gets more sections'' as $m$ increases, in general   $\Sym^m E$ will never  become entirely positive.  We will see in Section VI.E that  the $\Sym^m E$ have an intrinsic amount of ``flatness" no matter what metric we use for $E$.\footnote{This is not surprising as one can always add a trivial bundle to one that is strongly semi-positive   and it will remain strongly semi-positive.  Somewhat more subtly, a strongly semi-positive bundle may have ``twisted" sub-bundles on which the curvature form vanishes.}
  
  Among the algebro-geometric measures of positivity of a vector bundle $E\to X$ one may single out 
  \begin{enumerate}
  \item $E$ is nef (see Section II.D);
  \item some $\Sym^m E$ is big;
  \item some $\Sym^m E$ is free (i.e., $\Sym^mE$ is semi-ample).
  \end{enumerate}
  There are natural curvature conditions that imply both (i) and (ii), but other than the assumption of strict positivity which implies ampleness, we are not aware of natural curvature conditions that imply (iii).  For a specific question concerning this point, suppose that $L\to X$ is a line bundle and $Z\subset X$ is a reduced normal crossing divisor.  Assume that $h$ is a smooth metric in $L\to X$ whose Chern form has the properties
  \begin{itemize}
  \item[(a)] $\om\geqq 0$;
  \item[(b)] for $\xi\in T_xX$, $\om(\xi)=0\iff x\in Z$ and $\xi\in T_x Z\subset T_x X$.\end{itemize}
  Briefly, $\om\geqq 0$ and  the exterior differential system $\om=0$ defines $Z\subset X$.  Then by (a) we see that $L$ is nef, and as a consequence of (b) $L$ is big (cf.\ II.G).  Now $\om $ defines a K\"ahler metric $\om^\ast$ on $X^\ast:=X\bsl Z$, and one may pose the
  \begin{equation} 
 \lab{1a10n}
  \bmp{5}{{\bf Question:} Are there conditions on the curvature $R_{\om^\ast}$ of $\om^\ast$ that imply that $L$ is free?}\end{equation}
  
  We note that by (b) the Chern form $\om$ defines a norm in the normal bundle $N_{Z/X}$.  It seems reasonable to ask if the curvature form of this norm may be computed from $\lim_{x\to Z} R_{\om^\ast}(x)$, and if there are sign properties of this curvature form that imply freeness?\footnote{In a special case this question has been treated in the interesting paper \cite{Mok12}.  We refer to \pref{1a38} at the end of this introduction for similar questions concerning the Hodge vector bundles.}
  
  Remark that the construction of the completion $\ol M$ of the image of a period mapping \pref{1a2n} could have been accomplished   directly if one could show  that the canonically extended \hlb\ $\La_e\to \ol B$, which satisfies (i) and (ii), is free.\footnote{Of course the semi-ampleness of $\La_e\to \ol B$ is a~posteriori a consequence of \pref{1a2n}.}  Here one has the situation of the above question in the case where $h$ and $\om$ have singularities along $Z$.  However, as discussed in Section V these singularities are mild; in fact the singularities tend to ``increase" the positivity of the Chern form of~$\La_e$.

  In Section III we introduce the notation of \emph{norm positivity}.\footnote{The origin of the term is that a quantity (such as a Hermitian form) will be non-negative if it is expressed as a norm.}  This means that $E\to X$ has a metric whose curvature matrix is of the form
  \[
  \Theta_E=-{}^t \ol A\wedge A\]
  where $A$ is a matrix of $(1,0)$ forms arising from a holomorphic bundle map
  \begin{equation} \lab{n1a6}
  A:E\otimes TX\to G\end{equation}
  for $G$   a holomorphic vector bundle having a Hermitian metric.  The curvature form is then
  \begin{equation}\lab{n1a7}
  \Theta_E(e,\xi) = \| A(e\otimes \xi)\|^2_G\end{equation}
  where $\|\; \|_G$ is the norm in $G$.  Many of the bundles that arise naturally in algebraic geometry, such as the \hb\ and any globally generated bundle, have this property: The mapping $A$ in \pref{n1a6} generally has algebro-geometric meaning, e.g., as the differential of a map.  Such bundles are semi-positive,\footnote{Both in the sense of \pref{1a3} and in the sense of Nakano positivity (\cite{Dem12a}).} and their degree of positivity will have an algebro-geometric interpretation.
  
  \begin{rem}
  One may speculate as to the reasons for what might be called ``the unreasonable effectiveness of curvature in algebraic geometry."  After all, algebraic geometry is in some sense basically a $1^{\rm st}$ order subject---the Zariski  tangent space being   the primary infinitesimal invariant---whereas differential geometry is a $2^{\rm nd}$ order subject---the principle invariants are the  $2^{\rm nd}$ fundamental form  and the curvature.  One explanation is that for bundles that have the norm positivity property, the curvature form measures the \emph{size} of the $1^{\rm st}$ order quantity $A$ in \pref{n1a6}.\footnote{In this regard in the Hodge-theoretic situation we note that curvature $R_{\om^\ast}$ in \pref{1a10n} is a $2^{\rm nd}$, rather than a $3^{\rm rd}$, order invariant.}\end{rem}

  Another principal topic in these notes concerns the singularities   of metrics and their curvatures and   Chern forms.  This also is a very active area that continues to play a prominent role in complex algebraic geometry (cf.\ \cite{Dem12a} and \cite{Pau16}).  Reflecting the fact that in families of algebraic varieties there are generally interesting  singular members, the role of the singularities of Hodge metrics is central in Hodge theory.  Here there are two basic principles:
  \begin{enumerate}[{\rm (i)}]
  \item the singularities are mild, and
  \item singularities increase positivity.
  \end{enumerate}
  The first is explained in Section IV.B.  Intuitively it means that although the Chern polynomials are singular differential forms, they behave in their essential aspects as if they were smooth.  In particular, although they have distribution coefficients they may be multiplied and restricted to particular subvarieties as if they were smooth forms.  This last property is central to the proof of \pref{1a1} above.
  
    The first  major  steps in the general analysis of the singularities of the Chern forms of \hb s for several parameter variations of Hodge structure were taken by Cattani-Kaplan-Schmid (\cite{CKS86}), with subsequent refinements and amplifications by a number of people including Koll\'ar (\cite{Kol87}).  In the paper \cite{GGLR17}   further steps are taken, ones that may be thought of  as further refining the properties of    the wave front sets of the Chern forms.  This will be explained in Section IV.B and will be applied in Section V where an alternate proof of one of the two main ingredients in the proof of \pref{1a7} below will be given.
  
  The result \pref{1a1} is an application to a desingularization of  a KSBA moduli space  of a result that we now explain, referring to Sections I.B and I.C for explanations of notation and terminology.  Let $B$ be a smooth projective variety with smooth completion $\ol B$ such that $\ol B\bsl B=Z=\bigcup Z_i$ is a reduced normal crossing divisor.  We denote by $Z_I=\bigcap_{i\in I} Z_i$ the strata of $Z$ and by $Z^\ast_I = Z_{I,\reg}$ the smooth points of $Z_I$.  We consider a variation of \hs\ over $B$ given by a period mapping
  \begin{equation}\lab{1a6}
  \Phi:B\to\Ga\bsl D.\end{equation}
  We assume that the local monodromies $T_i$ around the $Z_i$ are unipotent with logarithms $N_i$, an assumption that may always be achieved by passing to a finite covering of $B$.\footnote{We will use \cite{PeSt08}, \cite{CM-SP} and \cite{CKS86} as general references for Hodge theory, including limiting \mhs s.}  We also suppose that the end piece
  \[
  \Phi_{\ast,n}:TB \to \Hom(F^n,F^{n-1}/F^n)\]
  of the differential of $\Phi$ is generically injective.\footnote{By using the \emph{augmented \hlb}\ $ \ottimes^{\lfloor n-1/2 \rfloor}_{p=0}  \det (F^{n-p}) $ 
  rather than just the \hlb\ $\La=\det F^n$, this assumption may be replaced by the injectivity of $\Phi_\ast$ (cf.\ \cite{GGLR17}).}  Finally assuming as we may that all $N_i\ne 0$, the image
  \[
  \Phi(B):= M\subset \Ga\bsl D\]
   of the period mapping \pref{1a6} is a closed analytic subvariety of $\Ga\bsl D$.  Beginning with \cite{Som78} there have been results stating that under certain conditions  $M$ is a quasi-projective variety and  the \hlb\  $\La\to M$ is at least big.  The following result from \cite{GGLR17} serves  to extend and clarify the previous work in the literature: 
   \begin{Thm}\lab{1a7} There exists a canonical completion $\ol M$ of $M$ as a  compact complex analytic space that has the properties
   \begin{enumerate}[{\rm (i)}]
   \item the \hb\ extends to $\La_e\to \ol M$ and there it is ample; and
   \item $\ol M$ is a Satake-Baily-Borel completion of $M$.\end{enumerate}
   \end{Thm}
   The second statement means the following:
   The period mapping \pref{1a6} extends to
   \begin{equation}\lab{1a8}
   \Phi_e :\ol B\to \ol M.\end{equation}
   Along the   non-singular strata $Z^\ast_I$ the exented period mapping $\Phi_e$  induces variations of graded polarized \lmhs s, and passing to the associated graded of these \mhs s gives period mappings
   \begin{equation}\lab{1a9}
   \Phi_I :Z^\ast_I\to \Ga_I\bsl D_I.\end{equation}
   Then \emph{the restriction to $Z^\ast_I$ of $\Phi_e$ in \pref{1a8} may be identified with} $\Phi_I$. 
   Setting $M_I=\Phi_I(Z^\ast_I)\subset \Ga_I \bsl D_I$, as a set $\ol M=M\amalg (\coprod_I M_I)$ 
    The precise meaning of this is explained in \cite{GGLR17}; among other things it means that $\Proj(\La_e\to \ol B)$ exists and on the boundary strata exactly detects the variation of the associated graded to the \lmhs s.\footnote{We remark that in the non-classical case when $\Ga\bsl D$ is not an algebraic variety, the construction of $\ol M$ is necessarily accomplished by gluing together local extensions of $M$ at the points of $\part M=\ol M\bsl M$. This requires both a \emph{local} analysis, based on \cite{CKS86},  of  local neighborhoods in $\olb$ along the fibres $F$ of the set-theoretically extended period map together with \emph{global} analysis of those fibres, and especially of the above mentioned positivity of the co-normal bundle $N^\ast_{F/\ol B}$.  As mentioned above, global issues necessitate  apparently new Hodge-theoretic constructions arising from the extension data in a limiting mixed Hodge structure.  Even in the classical case this   seems to give a new perspective on the Satake-Baily-Borel compactification.}
      The exterior differential system  \pref{1a2} defined by the singular differential form $\om_e$ may be made sense of on $\ol B$ (cf.\ Section VI.E), and there it defines a fibration by complex analytic subvarieties whose quotient is just $\ol M$.  The restriction property of $\om_e$ referred to above may be summarized as saying that
   \begin{equation}\lab{1a10}
   \om_e\big|_{Z^\ast_I}\hensp{\em is defined and is equal to} \om_I\end{equation}
   where $\om_I$ is the Chern form of the \hlb\ associated to \pref{1a9}.
   
   An implication of  the above is 
   \begin{equation}\lab{1a11}
   \om_e\hensp{\em is defined on $\ol M$ and there} \om_e>0.\footnotemark\end{equation}
   \footnotetext{In general, given a fibration $f\!:\!A\!\to\! B$ between manifolds $A,B$ and a smooth differential form $\Psi_A$ on $A$, the necessary and sufficient conditions that
   \[
   \Psi_A=f^\ast \Psi_B\]
   for a smooth differential form $\Psi_B$ on $B$ are that both $\Psi_A$ and its exterior derivative $d \Psi_A$ restrict to zero on
   \[
   \ker\{f_\ast :TA\to TB\}\subset TA.\]
   The above results extend this to a situation where $\Psi_A$ is a possibly singular differential form.}%
   
   One aspect of   this is that given $B$ and $\ol B$ as above the EDS
   \begin{equation}\lab{1a12}
   \bcs \om=0&\hensp{on} B\\
   \om_I = 0&\hensp{on the} Z^\ast_I\ecs\end{equation}
   on the smooth strata defines on each a   complex analytic fibration and a natural question is
   \begin{equation}\lab{1a13}
   \bmp{5}{\em For $I\subset J$ so that we have $Z^\ast_J\subset \ol Z^\ast_I$, does the closure of a fibre of $\om_I=0$ in $Z^\ast_I$ intersect $Z^\ast_J$ in a fibre of $\om_J=0$?  Are the limits of the fibres of $\Phi_I$ contained in the fibres of $\Phi_J$?}\end{equation}
   In other words, \emph{do the period mappings given by $\Phi$ on $B$ and $\Phi_I$ on $Z^\ast_I$ fit 
   together in an analytic way?}  That this is the case is proved in Section 3 of \cite{GGLR17} (cf.\ Step 1 in the proof of Theorem 3.15 there).  Although not directly related to the positivity of \hb s,   because of the subtle way in which the relative weight filtration property of a several parameter degeneration of \phs s enters in the argument  we think the result is of interest in its own right and in Section VI.A  we have  given a proof in the main  special case of a positive answer to \pref{1a13}.

   In Section VI we shall give applications of some of the bundles that naturally arise in Hodge theory.  As explained above, one such application is the use of the \hlb\ in proving that its ``Proj" defines the Satake-Baily-Borel completion of the image  of a period mapping.  The next application is to the cotangent bundle $\Om^1_{M^o}$ over the smooth points $M^o$ of $M$.  It is classical \cite{CM-SP} that the  induced metric on $M^o\subset \Ga\bsl D$ is K\"ahler and has   holomorphic sectional curvatures
   \[
   R(\xi)\leqq -c,\qquad \xi\in TM^o\]
   bounded above by a constant $-c$ where $c>0$.  On the other hand the more basic curvature form for $TM^o$ is given by the holomorphic bi-sectional curvatures $R(\eta,\xi)$, and we shall prove (cf.\ Theorem \ref{6b1}) that they satisfy
   \begin{equation} 
   \lab{n1a16} \begin{split}
   &R(\eta,\xi)\leqq 0\hensp{\em and} R(\eta,\xi)<0\hensp{\em is an open $\cU$ set in} TM^o\times_{M^o} TM^o.\\
   &
   \hbox{\em Moreover, $\cU$ projects onto each factor in $M^o\times M^o$.}\end{split}\end{equation}
 
   An interesting point here is that the induced curvature on the horizontal sub-bundle $(TD)_h\subset TD$ is a difference of non-negative terms, each of which has the norm positivity property \pref{n1a7}.  On \emph{integrable}  subspaces $I$ of $(TD)_h$ the positive term drops out leading to \pref{n1a16}.\footnote{In general the curvature matrices of \hb s are differences of non-negative curvature operators, and the paper \cite{Zuo} isolated the important point that on the kernels of Kodaira-Spencer maps one of the two terms drops out and on these subspaces the curvature forms have a sign.  Here the relevant bundles are constructed from the $\Hom(F^p/F^{p+1},F^{p-1}/F^p)$ bundles, and the relevant Kodaira-Spencer maps vanish in the integrable subspaces (cf.\ Proposition 2.1 in loc.\ cit.).  This issue is discussed below in some detail in Section VI.B.}
   
   A consequence of \pref{n1a16} is
   \begin{equation} \lab{n1a17}
   \bmp{5}{\em the curvature form $\Theta_{T^\ast M^o}\geqq 0$ and it  is positive on an open set.}\end{equation}
    To apply \pref{n1a17} two further steps are required.  One is  that curvatures decrease on holomorphic sub-bundles, which in the case at hand is used in the form 
       \begin{equation}\lab{n1a18}
       \Theta_I \leqq \Theta_{(TD)_h} \big|_I.\end{equation}
       The second involves the singularities that arise both at the singular point $M_{\sing} = M\bsl M^o$, and at the boundary $\part M=\ol M\bsl M$ of $M$.  Once these are dealt with we find the following, which are variants and very modest improvements of results in the literature due to Zuo in \cite{Zuo} and others (cf.\ \cite{Bru16} for a recent paper on this general topic and \cite{Den18} for related results and further bibliography).
       
       \begin{equation}\lab{n1a19}
  M\hensp{\em is of log general type,}\end{equation}
       and there exists an $m_0$ such that
       \begin{equation}\lab{n1a20}
       \Sym^m \Om^1_{\ol M}(\log)\hensp{\em is big for} m\geqq m_0.\end{equation}
       The precise meaning of these statements will be explained in Section VI.B below.\footnote{In fact, as will be discussed there $M$ is of \emph{stratified-log-general} type, a notion that is a refinement of log-general type.}

   The  theme of these notes is the positivity of vector bundles, especially those  arising from Hodge theory, and some  applications of  this   positivity to algebraic geometry.  Above we have discussed     applications of the  positivity of the Hodge \emph{line} bundle  and of the cotangent bundle at the smooth points to the image of a period mapping. We next turn to an application of positivity of the Hodge \emph{vector} bundle to the \emph{Iitaka conjecture}, of which a special case is
   \begin{equation}\lab{1a14}\bmp{5}{\em Let $f: X\to Y$ be a morphism between two smooth projective varieties and assume that
   \begin{enumerate}[{\rm (i)}]
   \item the general fibre  $X_y =f^{-1}(y)$ has Kodaira  dimension $\kappa(X_y)=\dim X_g$,
   \item $\Var f=\dim Y$.\footnotemark\end{enumerate}
   Then the Kodaira dimension is sub-additive}\end{equation} \footnotetext{This means that at a general point the Kodaira-Spencer map $\rho_y:T_yY\to H^1(TX_y)$ is injective.}
   \begin{equation}
   \lab{1a15} \kappa(X)\geqq \kappa(Y)+\kappa(X_y).\end{equation}
  This result was proved with one assumption, later seen to be unnecessary, by Viehweg (\cite{Vie83a}, \cite{Vie83b}).  His work built on \cite{Fuj78}, \cite{Uen75}, \cite{Uen77}, \cite{Kaw82}, \cite{Kaw83}, and \cite{Kaw85}, and was extended by \cite{Kol87}.  The sub-additivity of the Kodaira  dimension plays a central role in the classification of algebraic varieties, and over the years refinements of the result and alternative approaches to its proof have stimulated a very active and interesting literature (cf.\ \cite{Kaw02} and the more recent papers \cite{Sch15}, \cite{Pau16}).  
  It is not our purpose  here to survey this work but rather it is to focus on some of  the Hodge theoretic aspects of the result.
  
  One of these aspects is that the original proofs showed the necessity for understanding the singularities of the Chern form $\om$ of the \hlb. Here \cite{CKS86} is fundamental; as noted there by those authors    their work was in part motivated by establishing what was required for the proof of the Iitaka conjecture.

  In this particular discussion we shall mainly concentrate on the Hodge-theoretic aspect of the problem where everything is  smooth.  The issues that arise from the singularities may be dealt with using Theorem \ref{4b8}. 
  In Section VI.B following  general remarks on  what is needed to establish the result we first observe that if generic local Torelli\footnote{This means that the end piece $\Phi_{\ast,n}$ of the differential of the period mapping is generically injective.}  holds for $f:X\to Y$, then \pref{1a15} is a direct consequence of Theorem \ref{3b7} applied to the \hvb.  
  
  However, in general  the assumption $\kappa(X_y)=\dim X_y$ pertains to the $H^0(K^m_{X_y})$ for $m\gg 0$, not to  $H^0(K_{X_y})$ itself.  So the question becomes
  \begin{quote}
  \em How can Hodge theory be used to study the pluricanonical series $H^0(K^m_{X_y})$?\footnote{Here, as will be explained below, we are referring to both the  traditional use of the positivity of the Hodge line bundle, and also to the semi-positivity of the Hodge vector bundle which may be used in place of the concept of weak positivity introduced by Vieweg \cite{Vie83a} and \cite{Vie83b}.}\end{quote}
To address this an   idea in \cite{Kaw82} was  expanded and used in \cite{Vie83b} with further refinement and amplification in \cite{Kol87}.  It is this aspect that we shall briefly discuss here with further details given in Section VI.C.
  
  First a general comment.  Given a smooth variety $W$ of dimension $n$ and an ample \lb\ $L\to W$, there are two variations of Hodge structure  associated to the linear systems $|mL|$ for $m\gg 0$:
  \begin{itemize}
  \item[(a)] the smooth divisors $D\in |mL|$ carry \phs s;
  \item[(b)] for each $s\in H^0(L^m)$ with smooth divisor $(s)=D\in |mL|$ there is a cyclic covering $\wt W_s\xri{\pi} W$ branched over $D$ and with a distinguished section $\tilde s \in H^0(\wt W_s,\pi^\ast L)$ where $\tilde s^m =\pi^\ast s$.
  \end{itemize}
  We may informally think of (b) as the correspondence 
  \begin{equation}\lab{1a16}
(W,s^{1/m})\longleftrightarrow   (\wt W_s,\tilde s)  \end{equation}
obtained by extracting an $m^{\rm th}$ root of $s$.

  We   note that for  each $\la\in \C^\ast$ there is an isomorphism
  \begin{equation}\lab{1a18}
  \wt W_s \simto \wt W_{\la s},\end{equation}
  so that we should think of (b) as giving the smooth fibres in  a family
  \[
  \{ \wt W_{[s]} :[s] \in \P H^0(L^m)\}.\]
  A variant of  the construction (b) is then given by
  \begin{itemize}
  \item[(c)] $\xymatrix@C=.5pt@R=1.6pc{\wt W\ar[d]&\hspace*{-22pt}=\hbox{desingularization of }{\displaystyle\bigcup_{[s]\in \P H^0(L^m)} \wt W_{[s]}}\\
  \P H^0(L^m).&}$
  \end{itemize}

  It is classical that a suitable form of  local Torrelli for the end piece of the differential of the period mapping holds for the families of each of the  types (a) and (c) when $m\gg 0$ (cf.\ \cite{CM-SP}).  Moreover, the methods used to prove these results may be readily extended to show that 
  \begin{equation}\lab{1a17}\bmp{5}{\em local Torelli holds for both families of types (a) and  (c) when $W$ also varies.}\end{equation}

  In the case $L=K_W$  there is  the special feature
  \begin{equation}\lab{1a19}\bmp{5}{$H^0(K^m_W)$ \em is related to  the $H^{n,0}$-term of the \phs\ on $H^n(\wt W)$ arising in  the construction  (c).}\end{equation}
This means that, setting $\P = \P H^0(K^m_W)
  $
  \[
  \bmp{5}{\em $H^0(K^m_W)$ is a   summand of $F\otimes \cO_\P(-1)$ where $F\to \P$ is the \hvb\ associated to the family.}\]
As a consequence     
  \begin{equation}\lab{1a21}\bmp{5}{\em $H^0(K^m_W)\otimes H^0(\cO_\P(1))$ is a   summand of $H^0(F)$.}\end{equation}
   
There is also a metric Hodge theoretic interpretation of the pluricanonical series. For $\psi \in H^0(K^m_W)$ the Narashimhan-Simka \cite{NarSim68} Finsler-type norm
  \begin{equation}\lab{1a22}
  \|\psi\| = \int_W (\psi\wedge \ol\psi)^{1/m}\end{equation}
  has the Hodge-theoretic interpretation
  \begin{equation}\lab{1a23}
  \|\psi\|  = \int_{\wt W_\psi} \wt\psi\wedge\ol{\wt \psi}\end{equation}
  where  the RHS is a constant times the square of the Hodge length of $\wt\psi\in H^0(K_{\wt W_\psi})$. 
  The curvature properties of $\|\psi\|$  arising from its Hodge-theoretic interpretation were used by Kawamata \cite{Kaw82} in  in his proof of the Iitaka conjecture when $Y$ is  a curve.
  The metric properties of direct images of the pluricanonical systems have recently been an active and highly  interesting subject; cf.\ \cite{Pau16} for a summary of some of this work and references to the literature (cf.\ also \cite{Den18}).  It is possible that growing out of this work the Kawamata argument could be extended to give a proof of the full Iitaka conjecture.    We will comment further on this at the end of Section VI.C where a precise conjecture \pref{6c24} will be formulated.

 Returning to the discussion of \pref{1a15}, the essential point  is to show that
 \begin{equation}\lab{n1a25}
 f_\ast \om^m_{X/Y}\hensp{\em has lots of sections for} m\gg 0.\end{equation}
 The idea is that ``positivity $\implies$ sections'' and ``assumptions (i) and (ii) in \pref{1a14}$\implies f_\ast \om^m_{X/Y}$ has positivity.''  As noted above, when $m=1$ and local Torelli holds using Theorem \ref{3b7} there is sufficient positivity to achieve \pref{n1a25}.  The issue then becomes how to use the Hodge-theoretic interpretation of the pluricanonical series and semi-positivity of   \hvb s to also produce sections of $f_\ast \om^m_{X/Y}$.
 
 If we globalize the cyclic covering construction taking $W$ to be a typical $X_y$ we arrive at a commutative diagram 
  that  is essentially 
  \begin{equation*} 
 \begin{split}
 \xymatrix{\wtcx\ar[dr]^{\tilde f}\ar[d]_g\ar[r]^\pi&X\ar[d]^f
 \\ \P\ar[r]^p&Y}\end{split}\end{equation*}
 where the fibres are given by
  \begin{itemize}
 \item $\P_y = p^{-1}(y)=\P H^0(\om^m_{X_y})$,
 \item $\wtcx_y = \tilde f^{-1}(y)$, where  \item $\pi:\wtcx_y\to X_y$ is the family of cyclic coverings $\wt X_{y,[\psi]}\to X_y$ as  $[\psi]\in \P H^0(\om^m_{X_y})$ varies.\footnote{See \pref{6c12} for the complete definition and explanation of this diagram.}\end{itemize}
 If it were possible to  say that from \pref{1a19}      
 \begin{equation}\lab{1a24} f_\ast \om^m_X\hensp{\em is a direct factor of}  \tilde f_\ast\om_{\wtcx/Y}\end{equation}
 then     local Torelli-type statements   would imply positivity of $f_\ast \om^m_{X/Y}$.
   However, due to the twisting that occurs as $\la$ varies in the scaling identification \pref{1a18} essentially what happens is  that
 \begin{equation}\lab{1a25}
 f_\ast \om^m_{X/Y}\hensp{\em is a direct factor of} \tilde f_\ast \om_{\wtcx/\P}\otimes p_\ast \cO_\P(1).\end{equation}
 Now $\cO_\P(1)$ is positive along the fibres of $\P\to Y$, but the hoped for positivity of $f_\ast \om^m_{X/Y}$ means that $\cO_\P(1)$ tends to  be negative in the directions  normal  to $\P_y$ in $\P$.  Thus the issue has an additional subtlety that, as will be explained in Section VI.C, necessitates bringing in properties of the maps
 \[
 \Sym^k \tilde f_\ast \om^m_{\wtcx/\P} \to \tilde f_\ast \om^{k m}_{\wtcx/\P}.\]
   
   A final question that arises is this:
   \begin{equation} \lab{1a38}
   \bmp{5}{\em Let $f:X\to Y$ be a map between smooth, projective varieties, and assume that the locus of $y\in Y$ where $X_y = f^{-1}(y)$ is singular is a reduced normal crossing divisor.  Then
   \begin{enumerate}
   \item[{\rm (i)}] if $\det f_\ast \om^m_{X/Y}$ is non-zero, is it free?
   \item[{\rm (ii)}] if $f_\ast \om^m_{X/Y}$ is non-zero, if it free?\end{enumerate}
   }
   \end{equation}
   We note that (i) for $m=1$ follows from \pref{1a1}.  So far as we know, (ii) is not known for $m=1$.
   
 Turning to Section VI.D, from the description following \ref{1a7} of the Satake-Baily-Borel completion $\ol M$ of the image $M$ of a period mapping, we may say \emph{that the extended Hodge  line  bundle does  not detect extension data in \lmhs s.}  One may ask whether the same is or is not the case for the extended Hodge  vector  bundle.  In Section VI.C we give examples to show that in fact  this \vb\ may   detect both discrete and continuous extension data in \lmhs s.
 
 As noted above  the \hlb\ lives on the canonical completion $\ol\cM$ of the KSBA moduli space for surfaces of general type.  Interestingly, as will be seen by example in a sequel to \cite{GGLR17}  the \hvb\ does \emph{not} live on $\ol\cM$.  This is in contrast to the case of curves where the Hodge vector bundle is defined    on the moduli space $\ol\cM_g$.
 
  In the final section VI.E we shall revisit the exterior differential system \pref{1a2} defined by a Chern form, this time for  the Chern form $\om_E$ of the \lb\ $\cope(1)\to \P E$.  For a bundle $E\to X$ with $\Theta_E\geqq 0$ so that  $\om_E\geqq 0$ on $\P E$, the failure of strict positivity, or the degree of flatness, of the bundle $E$ is reflected by the foliation given by the integral varieties of the exterior differential system
  \[
  \om_E=0.\]
  The result here is Proposition \pref{6e2}, and it suggests a conjecture giving conditions under which equality might hold in the inequality
  \[
  \kappa(E)\leqq n(E)\]
  between the Kodaira-Iitaka dimension $\kappa(E)$ and numerical dimension $n(E)$ of the bundle.

  \subsection{General notations}
  \begin{itemize}
  \item $X,Y,W,\ldots$ will be compact, connected complex manifolds.
  \end{itemize}
  In practice they will      be smooth, projective varieties.
  \begin{itemize}
  \item $E\to X$ is a holomorphic \vb\ with fibres $E_x$, $x\in X$ and  rank $r=\dim E_x$;
  \item $A^{p,q}(X,E)$ denotes the global  smooth $E$-valued $(p,q)$ forms;
  \item we will not distinguish between a bundle and its sheaf of holomorphic sections; the context should make the meaning clear;
  \item $L\to X$ will   be a line bundle.
  \end{itemize}
  Associated to $L\to X$ are the standard notions
  \begin{enumerate}[{\rm (i)}]
  \item $\xymatrix{\vp_{L}: X\ar@{-->}[r]& \P H^0(X,L)^\ast}$ is the rational mapping given for $x\in X$ by
  \begin{equation}\lab{1b1}
  \vp_L(x)= [s_0(x),\dots,s_N(x)]\end{equation}
  where  $s_0,\dots, s_N$ is a basis for $H^0(X,L)$; in terms of a local holomorphic trivialization of $L\to X$  the $s_i(x)$ are given by holomorphic functions which are used to give the homogeneous coordinates on  the right-hand side of \pref{1b1};
  \item the \lb\ $L\to X$ is \emph{big}  if one of the equivalent conditions
  \begin{itemize}
  \item $h^0(X,L^m)=C m^d +\cdots$ where $C>0$, $\dim X=d$ and $\cdots$ are lower order terms;
  \item $\dim \vp_L(X)=\dim X$
  \end{itemize}
  is satisfied;
  \item $L\to X$ is \emph{free}\footnote{The term \emph{semi-ample} is also used.} if one of the equivalent conditions
  \begin{itemize}
  \item for some $m>0$,   the evaluation maps 
   \begin{equation}\lab{1b2}
  H^0(X,L^m)\to L^m_x\end{equation}
 are surjective for all $x\in X$;
  \item $\vp_{mL}(x)$ is a morphism; i.e., for all $x\in X$ some $s_i(x)\ne 0$;
  \item the linear system $|mL|:= \P H^0(X,L^m)^\ast$ is base point free for $m\gg 0$
  \end{itemize}
  is satisfied. 
  
  If the map \pref{1b2} is only  surjective for a general $x\in X$, we say that $L^m \to X$ is \emph{generically globally generated}.
   \item $L\to X$ is \emph{nef} if $\deg \lrp{L\big|_C} \geqq 0$ for all curves $C\subset X$; here $L\big|_C=L\otimes_{\cO_X}\cO_C$ is the restriction of $L$ to $C$;
  \item we will say that $L\to X$ is \emph{strictly nef} if 
 $  \deg \lrp{L\big|_C}>0$ for all curves $C\subset X$; \begin{itemize}
 \item for a vector bundle $E\to X$, we denote the $k^{\rm th}$ symmetric product
 by
 \[
 S^k E:=\Sym^k E;\]
  \item $\P E\xri{\pi} X$ is the projective bundle of 1-dimensional quotients of the fibres of $E\to X$; thus for $x\in X$
 \[
 (\P E)_x = \P E^\ast_x;\]
 \item $\cope(1)\to \P E$ is the tautological line bundle; then  
 \[
 \pi_\ast \cope(m) = S^m E \]
gives
 \[
 H^0(X,S^m E)\cong H^0(\P E,\cope(m))\]
 for all $m$.\footnote{The higher direct images $R^q_\pi\cope(m)=0$ for $q>0$, $m\geqq -r$; we shall note make use of this in these notes.}\end{itemize}
   \end{enumerate}

   \subsection{Notations from Hodge theory}
   We shall follow the generally standard notations and conventions as given in \cite{CM-SP} and are used in \cite{GGLR17}.  Further details concerning the  structure of \lmhs s (LMHS) will be given in Section IV.
   {\setbox0\hbox{(1)}\leftmargini=\wd0 \advance\leftmargini\labelsep
 \begin{itemize}
   \item $B$ will denote a smooth quasi-projective variety;
   \item a \emph{variation of Hodge structure} (VHS) parametrized by $B$ will be given by the equivalent data
   \begin{itemize}
   \item[(a)] a \emph{period mapping}
   \[
   \Phi:B\to \Ga\bsl D\]
   where $D$ is the period domain of weight $n$ \phs s $(V,Q,\fbul)$ with fixed \hn s $h^{p,q}$, and where the infinitesimal period relation (IPR)
   \[
   \Phi_\ast :TB\to I\subset T(\Ga\bsl D)\]
   is satisfied;
   \item[(b)] $(\V,\fbul,\nabla;B)$ where $\V\to B$ is a  local system with Gauss-Manin connection
   \[
   \nabla:\cO_B(\V)\to\Om^1_B(\V)\]
   and $\fbul=\{F^n\subset F^{n-1}\subset\cdots\subset F^0\}$ is a filtration of $\cO(\V)$ by holomorphic sub-bundles satisfying the IPR in the form  
   \[
   \nabla F^p\subset \Om^1_B(F^{p-1}),\]
   and where at each point $b\in B$ the data $(\V_b,\fbul_b)$ defines a \phs\ (PHS) of weight $n$.
   \end{itemize}
    In the background in both (a) and (b) is a bilinear form $Q$ that polarizes the \hs s; we shall   suppress the notation for it when it is not being explicitly used.
   \item The parameter space $B$ will have a smooth projective completion $\ol B$ with the properties
   \begin{itemize}
   \item[] $Z:= \ol B\bsl B$ is a reduced normal crossing divisor $Z=\bigcup Z_i$ having strata $Z_I:= \bigcap_{i\in I}Z_i$   with   $Z^\ast_I\subset Z_I$  denoting the non-singular points $Z^\ast_{I,\text{reg}}$ of $Z_I$, and where
    the local monodromies $T_i$ around the irreducible branches $Z_i$ of $Z$ are unipotent with logarithms $N_i$;
  \end{itemize}   \item the \emph{Hodge vector bundle} $F\to B$ has fibres $F_b:= F^n_b$;
   \item the \emph{Hodge line bundle} $\La:= \det F = \wedge^{h^{n,0}}F$;
   \item the polarizing forms induce Hermitian metrics in $F$ and $\La$;
   \item the differential of the period mapping is
   \[
   \Phi_\ast :TB\to \bigoplus_{p\geqq \left[\begin{smallmatrix} \frac{n}{2}\end{smallmatrix}\right]} \Hom(F^p,F^{p-1}/F^p)\]
   where $F^p\to B$   denotes the Hodge filtration bundles;
   \item setting $F=F^n$ and $G=F^{n-1}/F^n$ the \emph{end piece} of $\Phi_\ast$ is
   \[
   \Phi_{\ast,n}:TB\to \Hom(F,G); \]
     \item the Hodge filtration bundles have canonical extensions
   \[
   F^p_e\to\ol B;
   \]
 using the Hodge metrics on $B$, the holomorphic sections of $F^p_e\to \ol B$ are those whose Hodge norms have logarithmic growth along $Z$;
   \item the \emph{geometric case} is when the VHS arises from the cohomology along the fibres in a smooth projective family
   \[
   \cX\xri{f} B;
   \]
   \item such a family has a completion to
   \[
   \ol{\cX}\xri{\bar f} \ol B\]
   where this map has  the Abramovich-Karu (\cite{AK00}) form of semi-stable reduction (cf.\ Section 4 in \cite{GGLR17} for more details and for the notations to be used here); then the canonical extension of the \hvb\ is given by
   \[
   F_e = \bar f_{\ast} \om_{\ol{\cX}/\ol B}.\]
    \end{itemize}}
   
   We conclude this introduction with an observation and a question.  For line bundles $L\to X$ over a smooth projective variety $X$ of dimension $d$, there are three important properties:
  {\setbox0\hbox{(iii)}\leftmargini=\wd0 \advance\leftmargini\labelsep
  \begin{enumerate}
   \item $L$ if nef;
   \item $L$ is big;
   \item $L$ if free (and therefore it is semi-ample).
   \end{enumerate}}  \noindent 
   Clearly (iii)$\implies$(i) and (ii),  and (iii)$\implies |mL|$ gives a birational morphism for $m\gg 0$.

   In this paper we have considered the case where $L\to X$ has a metric $h$  that may be singular, with Chern  $\om$ that defines a $(1,1)$ current representing $c_1(L)$.  The singularities of $h$ are generally of the following types:
{\setbox0\hbox{()}\leftmargini=\wd0 \advance\leftmargini\labelsep
    \begin{itemize}
   \item $h$ vanishes along a proper subvariety of $X$;
   \item $h$ becomes infinite, either logarithmically or analytically (in the sense explained below) along a proper subvariety of $X$.
   \end{itemize}}  \noindent 
   We note that
   \[
   \om \geqq 0\hbox{ (as a currrent)}\implies \hbox{(i)},\]
   and that  in a Zariski open $X^0\subset X$ where $h$ is a smooth metric
   \[
   \om^d >0\implies \hbox{(ii).}\]
   The Kodaira theorem states that if $X^0=X$, then $h$ is ample.  As discussed above in relation to the question \pref{1a10n}, for a number of purposes, including applications to Hodge theory, it would be desirable to have conditions on $\om$ that imply (iii).

   \section{Measures of and types of positivity}
   In this section we will
   \begin{enumerate}[{\rm (i)}]
   \item define two measures of positivity---the Kodaira-Iitaka dimension and the numerical dimension;
   \item define two types of positivity---metric positivity and numerical positivity.
   \end{enumerate}
   Each of these will be done first for line bundles and then for   \vb s, and the   definitions will be related via the canonical association
   \[
   E\to X\leadsto \cope(1)\to \P E\]
   of the tautological \lb\ $\cope(1)\to \P E$ to the \vb\ $E\to X$.\footnote{We will use the   convention whereby $\P E$ is the bundle of 1-dimensional \emph{quotients} of the fibres of $E$; thus the fibre
   \[
   \P E_x = \P (E^\ast_x).\]}
   
   For each of the types of positivity there will be two notions,   \emph{strict positivity} denoted $>0$ and \emph{semi-positivity} denoted $\geqq 0$.  For metric positivity there will be a third, denoted   by $E^0_{\rm met}>0$ and $L^0_{\rm met}>0$ respectively  for  \vb s and \lb s,  and which means that we have $\geqq 0$ everywhere and $>0$ on an open set.  We also use the term   \emph{strong semi-positivity} to mean that $E_{\met}\geqq 0$ and $\det E^0_{\met}>0$.  
    Finally, since metric positivity will be the main concept used in these notes, in some of the later sections we will just write $E>0$ and $E^0>0$ to mean $E_{\met} >0$ and $E^0_{\met}>0$.
   
   \subsection{Kodaira-Iitaka dimension}
 In algebraic geometry positivity  traditionally suggests ``sections," and one standard measure of the amount of sections of a \lb\ $L\to X$ is given by its \emph{Kodaira-Iitaka dimension} $\kappa(L)$. This is defined by
   \begin{equation}\lab{2a1}
   \kappa(L)= \mmax \dim \vp_{mL}(X)
\end{equation}
where 
\[
\vp_{mL}:\xymatrix{X\ar@{-->}[r]&\P H^0(mL)^\ast}\]
is the rational mapping given by the linear system $|mL|$.  If $h^0(mL)=0$ for all $m$ we set $\kappa(L)=-\infty$.
From \cite{Dem12a} we have
\begin{equation}\lab{2a2}
h^0(mL)\leqq O(m^{\kappa(L)})\qquad \hbox{ for }m\geqq 1,\end{equation}
and  $\kappa(L)$ is the smallest exponent for which this estimate holds.\footnote{Including $\kappa(L)=-\infty$ where we set $m^{-\infty}=0$ for $m>0$.}  We  will sometimes write \pref{2a2} as
\[
h^0(mL)\sim Cm^{\kappa(L)} , \quad C>0.\]
   We note that 
\[
\kappa(L) = \dim X\iff L\to X\hensp{is big.}\]

\subsection{Metric positivity}  Given a Hermitian metric $h$ in the fibres of a holomorphic \vb\ $E\to X$ there is a canonically associated Chern connection
\[
D:A^0 (X,E)\to A^1 (X,E)\]
characterized by the properties (\cite{Dem12a})
\begin{equation}\lab{2b1}
\bcs
D'' =\ol\part\\
d(s,s') = (Ds,s')+(s,Ds')\ecs\end{equation}
where $s,s'\in A^0(X,E)$ and $(\;,\;)$ denotes the Hermitian inner product in $E$.  The curvature   
\[
\Theta_E:=D^2\]
is linear over  the functions; hence it is pointwise an algebraic operator.  Using \pref{2b1} it is given by a \emph{curvature operator}
\[
\Theta_E \in A^{1,1}(X,\End E)\]
which satisfies
\[
(\Theta_E e,e')+(e,\Theta_E e')=0\]
where $e,e'\in E_x$.  Relative to a local holomorphic frame $\{s_\alpha\},h=\|h_{\alpha\bar\beta}\|$ is a Hermitian matrix and the corresponding connection and curvature matrices are given by
\[
\bcs \theta= h^{-1}\part h\\
\Theta_E=\ol\part (h^{-1}\part h)=\Big\|{\displaystyle\sum_{\alpha,\beta,i,j}}\Theta^\alpha_{\bar\beta i \bar j} s_\alpha\otimes s^\ast_\beta  \otimes dz^i \wedge d\bar z^j\Big\| \ecs.\]
For line bundles the connection   and   curvature matrices   are respectively $\theta = \part\log h $  and  
$
\Theta_L =-\part\ol\part \log h$.  If $h=e^{-\vp}$, then
\[
\Theta_L\geqq 0\iff (i/2) \part\ol\part \vp \geqq 0\iff \vp\hensp{is plurisubharmonic.}\]

\begin{defin}
The \emph{curvature form} is given for $x\in X$, $e\in E_x$ and $\xi \in T_x X$ by
\begin{equation}\lab{2b2}
\Theta_E(e,\xi) = \lra{\lrp{\Theta_E(e),e},\xi\wedge \bar \xi}.\end{equation}
When written out in terms of the curvature matrix $\Theta_E(e,\xi)$ is the bi-quadratic form
\[
\sum_{\alpha,\beta,i,j} \Theta^\alpha_{\bar\beta i \bar j} e_\alpha \bar e_\beta \xi_i \bar\xi_j.\]
The bundle $E\to X$ is \emph{positive},\footnote{We should say \emph{metrically positive}, but since this is the main type of  positivity used in these notes we shall drop the ``metrically."} written $E_{\met}>0$, if there exists a metric such that $\Theta_E(e,\xi)>0$ for all non-zero $e,\xi$.  For simplicity we will write $\Theta_E>0$.   If we have just $\Theta_E(e,\xi)\geqq 0$, then we shall say that $E\to X$ is \emph{semi-positive} and write $E_\met \geqq 0$.  It is \emph{strongly semi-positive} if $E_{\met}\geqq 0$ and $(\det E)_{\met}>0$ on an open set.  

The bundle is \emph{Nakano positive} if there exists a metric such that for all non-zero $\psi\in E_x\otimes T_xX$ we have
\begin{equation}\lab{2b3}
\lrp{\Theta_E(\psi),\psi}>0.\end{equation}
The difference between positivity and Nakano positivity is that the former involves only the decomposable tensors in $E\otimes T X$ whereas the latter involves all tensors.  In \cite{Dem12a} there is the concept of $m$-positivity that involves the curvature acting on tensors of rank $m$ and which interpolates between the two notions defined above.
\end{defin}

Positivity and semi-positivity have functoriality properties (\cite{Dem12a}).  For our purposes  the two most important are
\begin{align} & \lab{2b4}\bmp{5}{the tensor product of positive bundles is positive, and similarly for semi-positive;}
\\
&\lab{2b5}\bmp{5}{the quotient of  a positive bundle  is positive, and similarly for semi-positive.}\end{align}
The second follows from an important formula that we now recall (cf.\   \cite{Dem12}).  If we have an exact sequence of holomorphic \vb s
\begin{equation}\lab{2b6}
0\to S\to E\to Q\to 0,\end{equation}
then a metric in $E$ induces metrics in $S,Q$ and there is a canonical \emph{second fundamental form} 
\[
\beta\in A^{1,0}(X,\Hom(S,Q))\]
that measures the deviation  from being holomorphic of the $C^\infty$ splitting of \pref{2b6} given by the metric.  For $j:Q\hookrightarrow E$ the inclusion given by the $C^\infty$ splitting and $q\in Q_x,\xi\in T_xX$ the formula is (loc.\ cit.) 
\begin{equation}\lab{2b7}
\Theta_Q(q\otimes \xi)=\Theta_E(j(q)\otimes \xi) +\|\beta^\ast(q)\otimes \xi\|^2 \end{equation}
where by definition the last term is  $ -\lra{ (\beta^\ast(q),\beta^\ast(q))_{S}, \xi\wedge \bar\xi}$ and  $(\enspace,\enspace)_S$ is the induced metric in $S$.  The minus sign is because the Hermitian adjoint $\beta^\ast$ is of type (0,1).

\subsection*{Examples} \hfill

(i) The universal quotient bundle $Q\to G(k,n)$ with fibres $Q_\La=\C^n/\La$ over the Grassmannian $G(k,n)$ of $k$-planes $\La\subset \C^n$ has a metric induced by that in $\C^n$, and with this metric
\[
\Theta_Q\geqq 0\hensp{and} \Theta_Q>0\iff k=n-1.\]
Similarly, the dual $S^\ast \to G(k,n)$ of the universal sub-bundle has $\Theta_{S^\ast}\geqq 0$ and $\Theta_{S^\ast}>0\iff k=1$.  

Geometrically, for a $k$-plane $\La\in\C^n$ we have the usual   identification
\[
T_\La  G(k,n)\cong \Hom(\La,\C^n/\La).\]
Then for $\xi\in \Hom(\La,\C^n/\La)$ and $v\in \La$
\[
\Theta_{S^\ast} (v,\xi) = 0\iff \xi(v)=0.\]
Here the RHS means that for the infinitesimal displacement $\La_\xi$ of $\La$ given by $\xi$ we have
\[
v\in \La\cap \La_\xi.\]
The picture for $G(2,4)$ viewed as the space of lines in $\P^3$ is
\[
\begin{picture}(60,80)
\put(10,0){\includegraphics{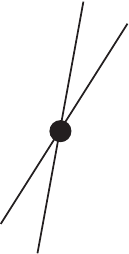}}
\put(0,0){$\La$}
\put(25,0){$\La_\xi$}
\put(35,30){$v$}
\end{picture}\]

There are similar  semi-positivity properties for any globally generated \vb, since such bundles are induced from holomorphic mappings  to a Grassmannian, and positivity and semi-positivity have the obvious functoriality properties.

(ii) The \hb\ $F\to B$ with the metric given by the \HR\ bilinear relation satisfies $\Theta_F\geqq 0$ \cite{Gri72}, but unless $h^{n,0}=1$ very seldom do we have $\Theta_F>0$.

These examples will be further discussed in Section III.A.

(iii) In the geometric case of a family $\cX\xri{\pi}B$ with smooth fibres, the Narashimhan-Simka \cite{NarSim68} Finsler type metrics \pref{1a22} in $\pi_\ast \om^m_{\cX/B}$ have the Hodge theoretic interpretation \pref{1a23}.  As a consequence
\begin{quote} \em There is a metric $h_m$ in $\cO_{\P f_\ast \om^m_{X/Y}} (1)$ whose Chern form $\om_m\geqq 0$.\end{quote}
Some care must both be taken here as although $h_m$ is continuous it is not smooth and so $\om_m = (i/2) \ol\part\part \log h_m$ and the inequality $\om_m\geqq 0$ must be taken in the sanse of currents (cf.\ \cite{Dem12} and \cite{Pau16}).  Metrics of this sort were used in \cite{Kaw82} and have been the subject of numerous recent works, including \cite{Ber09}, \cite{BerPau12}, \cite{PauTak14}, \cite{MorTak07}, \cite{MorTak08}, and also \cite{Pau16} where a survey of the literature and further references are given.

\subsection{Interpretation of the curvature form}
Given a holomorphic \vb\ $E\to X$ there is the associated projective bundle $\P E\xri{\pi} X$ of 1-dimensional quotients of the fibres of $E$; thus $(\P E)_x = \P E^\ast_x$.  Over $\P E$ there is the tautological \lb\ $\cope(1)$.  A metric in $E\to X$ induces one in $\cope(1)\to \P E$, and we denote by $\om_E$ the corresponding curvature form.  Then $\Om_E:=(i/2\pi)\om_E$ represents the Chern class  $c_1(\cope(1))$ in $H^2(\P E$).

Since $\om_E\big|_{(\P E)_x}$ is a positive $(1,1)$ form, the vertical  sub-bundle 
\[
V:= \ker \pi_\ast :T\P E\to T X\]
to the fibration $\P E\to X$ has a $C^\infty$ horizontal complement $H=\om^\bot_E$.  Thus as $C^\infty$ bundles    
\[
\bcs
T\P E&\chs\cong V\oplus H,\hbox{ and}\\
\pi_\ast :H&\chs\simto \pi^\ast T X.\ecs\]

In more detail, using the metric we have a complex conjugate linear identification $E^\ast_x \cong E_x$, and using this we shall write points in $\P E$ as $(x,[e])$ where $e\in E_x$ is a non-zero vector.
Then we have an isomorphism
\begin{equation}\lab{2c1}
\pi_\ast :H_{(x,[e])}\simto T_x X.\end{equation}
Using this identification and normalizing to have $\|e\|=1$, the interpretation of the curvature form is given by the equation 
\begin{equation}\lab{2c2}
\Theta_E(e,\xi) = \lra{\om_E,\xi\wedge \bar\xi}=: \om_E(\xi) \end{equation}
where $\xi\in T_xX\cong H_{(x,[e])}$ and  the RHS is evaluated at $(x,[e])$. 
Thus \begin{equation}\lab{2c3}
\Theta_E>0\iff \om_E >0,\end{equation}
and similarly for $\geqq 0$.  There are the evident extensions of \pref{2c3} to open sets in $\P E$ lying over open sets in $X$.  For semi-positive \vb s we summarize by saying that \emph{ the curvature form $\Theta_E$ measures the degree of positivity of $\om_E$ in the horizontal directions}.

For later use we conclude with the observation that using $\cO_X(E)\cong \pi_\ast \cope(1)$, given $s\in \cO_{X,x}(E)$ there is  the identification of (1,1) forms
\begin{equation}
\lab{2c4}
(-\part\ol\part \log \|s\|^2) (x)=\om_E(x,[s(x)]) \end{equation}
where the RHS is the (1,1) form $\om_E$ evaluated at the point $(x,[s(x)])\in \P E$ in the total tangent space (both vertical and horizontal directions).

\subsection{Numerical positivity}\hfill

In this section we shall discuss various measures of numerical positivity, one main point being that these will apply to bundles arising from Hodge theory.  The basic reference here is \cite{Laz04}.  A  conclusion will be that the \hvb\ $F$ is numerically semi-positive; i.e., $F_{\num}\geqq 0$ in the notation to be introduced below.

\subsubsection{Definition of numerical positivity}  We first recall the definition of the cone $\cC = \oplus \cC_d$ of positive polynomials $P(c_1,\dots,c_r)$ where $c_i$ has weighted degree $i$.  For this we consider partitions $\la=(\la_1,\dots,\la_n)$ with $0\leqq \la_n\leqq \la_{n-1}\leqq \cdots\leqq \la_1\leqq r$, $\Sig \la_i =n$, of $n=\dim X$.  For each such $\la$ the \emph{Schur polynomial} $s_\la$ is defined by the determinant 
\begin{equation}\lab{2d1}
s_\la = \left|
\begin{matrix}
c_{\la_1}& c_{\la_1+1}&\cdot&\cdot&c_{\la_1+n-1}\\
\cdot&&&\cdot\\
\cdot&&&\cdot\\
c_{\la_n-n+1}&\cdot&\cdot&\cdot& c_{\la_n}\end{matrix}\right|.\end{equation}
Then (\cite{Laz04}) $\cC$ is generated over $\Q^{>0}$ by the $s_\la$.  It contains the Chern monomials 
\[
c^{i_1}_1\cdots c^{i_r}_r,\qquad i_1+2i_2+\cdots+ri_r\leqq n\]
as well as some  combinations of these  with negative  coefficients, the first of which is  $c^2_1-c_2$.

For each $P\in \cC_d$ and $d$-dimensional subvariety $Y\subset X$ we consider
\begin{equation}
\lab{2d2} 
\int_Y P(c_1(E),\dots,c_r(E)) =P(c_1(E),\dots,c_r(E))[Y] \end{equation}
where the RHS is the value of  the cohomology class  $P(c_1(E),\dots,c_r(E))\in H^{2d}(X)$ on the fundamental class $[Y]\in H_{2d}(X)$.

 \begin{defin} $E\to X$ is \emph{numerically positive}, written $E_{\num}>0$, if \pref{2d2} is positive for all $P\in \cC_d$ and  subvarieties $Y\subset X$.
 \end{defin}
 
 We may similarly define $E_\num\geqq 0$. 
 
 A non-obvious result (\cite{Laz04}) is
 \begin{equation}\lab{2d4} 
 \bmp{5}{\em For line bundles $L\to X$, we have
 \[
 L_{\num}\geqq 0\iff L\hensp{is nef.}\]}
 \end{equation}
 The essential content of this statement is 
 {\setbox0\hbox{(1)}\leftmargini=\wd0 \advance\leftmargini\labelsep
 \begin{quote} \em $
 c_1(L)[C] \geqq 0\hensp{for all curves} C\implies c_1(L)^d [Y] \geqq 0$ for all $d$-dimensional subvarieties $Y\subset X.$\end{quote}}  \noindent 
 This is frequently formulated as saying that if $L$ is nef then it is in the closure of the ample cone.
 
 \subsubsection{Relation between $E_\num>0$ and $\cope(1)_\num>0$}\hfill

  For the  fibration  $\P E\xri{\pi}X$ there is a Gysin or  \emph{integration over the fibre} map
 \[
 \pi_\ast :H^{2(d-r+1)}(\P E)\to H^{2d}(X).\]
 It is defined by moving cohomology to homology via Poincar\'e duality, taking the induced map on homology and then again using  Poincar\'e duality.  In de~Rham cohomology the mapping is given by what the name suggests. 
 The $d^{\rm th}$ \emph{Segre polynomial} is defined by
 \begin{equation}\lab{2d5}
 S_d(E)=\pi_\ast \lrp{c_1\lrp{\cope(1)}^{d-r+1}}.\end{equation}
Then (\cite{Laz04}):  (i) $S_d(E)$ is a polynomial in  the Chern classes $c_1(E),\dots,c_r(E)$, and (ii)   $S_d(c_1,\dots,c_r)\in \cC_d$. That it is a polynomial in the Chern classes  is a consequence  of the \emph{Grothendieck relation}
 \begin{equation}\lab{2d6} c_1\lrp{\cope(1)}^r - c_1\lrp{\cope(1)}^{r-1} \pi^\ast e_1(E)+\cdots +(-1)^r \pi^\ast c_r(E)=0.\end{equation}
 The first few Segre polynomials are
 \[
 \bcs
 S_1=c_1\\
 S_2=c_1^2-c_2\\
 S_3 = c_1 c_2\\
 S_4= c^4_1 -2c^2_1c_2+c_1c_3-c_4.\ecs\]
 An important implication is 
 \begin{equation}
 \lab{2d7}
 \cope(1)_{\num} >0\implies E_\num>0.\end{equation}
 
  \begin{proof}
 By Nakai-Moishezon (cf.\ \eqref{2f4} below), $\cope(1)$ and hence $E$ are ample.  Then $E_\num>0$ by the theorem of Bloch-Gieseker (\cite{Laz04}).\end{proof}
 
 As will be noted below, the converse implication is not valid. 
 \subsection{Numerical dimension}
 Let $L\to X$ be a line bundle with $L_\num\geqq 0$; i.e., $L$ is nef.
 
 \begin{defin}[cf. \cite{Dem12a}]
 The \emph{numerical dimension} is the largest integer $n(L)$ such that
 \[
 c_1 (L)^{n(L)}\ne 0.\]
 In practice in these notes there will be a semi-positive (1,1) form $\om$ such that $[(i/2\pi)\om]=c_1(L)$, and then $n(L)$ is the largest integer such that $\om^{n(L)+1}\equiv 0$ but
 \[
 \om^{n(L)}\ne 0\]
 on an open set.\end{defin}

 Relating the Kodaira-Iitaka and numerical dimensions, from \cite{Dem12a} we have
 \begin{equation}
 \kappa(L)\leqq n(L)\end{equation}
 with equality if $n(L)=\dim X$, but where equality may not hold if $n(L)<\dim X$.
 
\begin{exam} [loc.\ cit.]
Let $C$ be an elliptic curve and $p,q\in C$ points such that $p-q$ is not a torsion point.  Then the line bundle $[p-q]$ has a flat unitary metric, but $h^0(C,m[p-q])=0$ for all $m>0$.  For any nef line bundle $L'\to X'$ we set
\[
X=C\times X',\qquad L=[p-q]\boxtimes L'.\]
Then $\kappa(X)=-\infty$ while $n(L)$ may be any integer with $n(L)\leqq \dim X-1$.\end{exam}

In Section VI.E we  will   discuss that  understanding the reason  we may have the strict inequality $\kappa(L)<n(L)$  seems to involve some sort of flatness as in the above example.\footnote{Cf.\ \cite{CD17a} and \cite{CD17b} for a related discussion involving certain Hodge bundles.} 

For a \vb\ $E\to X$ with $E_\num\geqq 0$ we  have the
\begin{defin} The \emph{numerical dimension} $n(E)$ of the \vb\ $E\to X$ is given by $n(\cope(1))$.  
\end{defin}

Since $\cope(1)$ is positive on the fibers of $\P E\to X$ we have
\[
r-1\leqq n(E)\leqq \dim \P E= \dim X + r-1.\]
Conjecture \ref{6e6}  below suggests  conditions under which equality will hold. 

 From \pref{2d5} we have  
\begin{equation}
\lab{2e2}
n(E)\hensp{\em is the largest integer with} S_{n(E)-r+1}(E)\ne 0,\end{equation}
which serves to define the numerical dimension of a semi-positive \vb\  in terms of its Segre classes.
 We remark that as noted above
 \begin{equation}\lab{2e3}
 \cope(1)_\num \geqq 0\implies S_q(E)\geqq 0 \hensp{for any} q\geqq 0;\end{equation}
 we are not aware of results concerning the converse implication.
 
  \subsection{Tangent bundle}
 When $E=TX$ and the metric on $E\to X$ is given by a K\"ahler metric on $X$ the curvature form has
 the interpretation
 \begin{equation}\lab{2g1}
 \Theta_{TX}(\xi,\eta) = \lrc{\begin{matrix}\hbox{holomorphic bi-sectional curvature}\\
 \hbox{ in the complex 2-plane $\xi\wedge \eta$ }\\
 \hbox{spanned by $\xi,\eta\in T_xX$}\end{matrix}}.\end{equation}
 When $\xi=\eta$ we have
 \begin{equation}\lab{2g2}
  \Theta_{TX}(\xi,\xi)=\lrc{\begin{matrix}\hbox{holomorphic sectional curvature in}\\
 \hbox{the complex line spanned by $\xi$}\end{matrix}}.\end{equation}
 Of particular interest and importance in Hodge theory and in other aspects of algebraic geometry is the case when $TX$ has some form of negative curvature.   
 \begin{prop}[\cite{BKT14}] \lab{2g3}
 Assume there is $c>0$ such that {\rm (i)} $\Theta_{TX}(\xi,\xi)\leqq -c$ for all $\xi$,   and  {\rm (ii)} $ \Theta_{TX}(\la,\eta)\leqq 0$ for all $\la,\eta$.  Then there exists $\xi,\eta$ such that $\Theta_{TX}(\xi,\eta)\leqq -c/2$ for all $\eta$.\end{prop}
 
 In other words, if the holomorphic sectional curvatures are negative and the holomorphic bi-sectional curvatures are non-positive, then they are negative on an open set in $G(2,TX)$, the Grassmann bundle of 2-planes in $TX$, and this open set maps onto $X$.  Noting that $\Gr(2,TX)$ maps to an open subset of the horizontal sub-bundle in the fibration $\P TX\to X$, from \eqref{2f2} below we have the 
 
 \begin{Cor}\lab{2g4}
 If the assumptions in \ref{2g3} are satisfied, then $T^\ast X$ is big.\end{Cor}
 
 \subsection{First implications}
 In this section for easy reference we will summarize the first implications of the two types of positivity on the Kodaira-Iitaka dimension and numerical dimension.  These are either well known or easily inferred from what is known.
 
 \ssn{Case of a \lb\ $L\to X$}
 \begin{align}\label{2f1}
 L_\met>0  &\implies L\hbox{ ample}.\end{align}
 This is the  Kodaira theorem which initiated the relation between metric positivity and sections.  
 
 The next is   the  Grauert-Riemenschneider conjecture, established by Siu and Demailly (cf.\ \cite{Dem12a} and the references there):
 \begin{equation}
 \lab{2f2} L^0_\met>0 \implies \kappa(L)=\dim X.\end{equation}
 Here we recall that $L^0_\met>0$ means that there is a metric in $L\to X$ whose curvature form $\om\geqq 0$ and where $\om>0$ on an open set. 
   The result may be  phrased as $$L^0_\met >0\implies L \hbox{ is big.}$$
 For these notes this variant of the Kodaira theorem will play a central role as bundles constructed from the extended \hvb\ tend to be big and perhaps free,\footnote{The issue of additional conditions that will imply that an $L$ which is nef and big is also free is a  central one in birational geometry (cf.\ \cite{Ko-Mo}).  The results there   seem to involve assumptions on $mL-K_X$.  In Hodge theory $K_X$ is frequently one of the things that one wishes to establish properties of.} but just exactly what their ``Proj" is seems to be an interesting issue. Because of this for later use, in the case when $X$ is projective we now give a 
 
 \begin{proof}[Proof of \pref{2f2}]
 Let $H\to X$ be a very ample line bundle chosen so that $H-K_X$ is ample.  Setting $F= L+H$ we have $F_{\met}>0$.  For $D\in |F|$ smooth using the Kodaira vanishing theorem  we have
 \[
 \bcs h^q (X,mF)=0,&q>0,\\
 h^q(D,mF\big|_D)=0,&q>0. \ecs\]
 We note that the  vanishing theorems    will remain true if we replace $L$ by a positive
  multiple.
 
 Let $D,\dots,D_m\in |H|$ be distinct smooth divisors.  From the exact sequence $0\to m(F-H)\to mF\to \opplus^m_{j=1} mF\big|_{D_j}$ we have
 \[
 0\to H^0(X,m(F-H))\to H^0(X,mF)\to \opplus^m_{j=1} H^0(D_j,mF).\]
 This gives
 \[
 h^0(X,mL) = h^0(X,m(F-H))\geqq h^0(X,mF)-mh^0(D,mF).\]
 Using the above vanishing results  
 \[
 h^0(X,mL)\geqq \chi(X,mF)-m\chi(D,mF).\]
 For $d=\dim X$ and letting $\sim$ denote modulo lower order terms, from the Riemann-Roch theorem we have
 \begin{align*}
 \chi(X,mF)&\sim \frac{m^d}{d!} F^d\\
 m\chi(D,mF)&\sim \frac{m^d}{d!} (dF^{d-1}\cdot H).\end{align*}
 For $m\gg 0$ this gives
 \[
 h^0(X,mL)\geqq \frac{m^d}{d!} (F^d-dF^{d-1}\cdot H)+o(m^d).\]
From
 \begin{align*}
 F^d-dF^{d-1}\cdot H& = (L+H)^d -d(L+H)^{d-1}\cdot H\\
 &= (L+H)^{d-1} \cdot (L-(d-1)H) \end{align*}
 replacing $L$ by a multiple we may make this expression positive.\footnote{This argument is due to Catanese; cf.\   \cite{Dem12a}.}
 \end{proof}

    The next result is  
 \begin{equation}\lab{2f3}
\quad L_\met >0\implies L_\num>0, \hbox{ and similarly for }\geqq 0\hspace*{1in} \hbox{(obvious)}.\end{equation}
  The  inequality in  \pref{2f3}  is sometimes phrased as $$L_\met \geqq 0\implies L\hbox{ is nef.}$$
 The theorem of Nakai-Moishezon is 
 \begin{align}\lab{2f4}
 L_\num&>0\iff  L\hensp{is ample.} \end{align}
  The next inequality was noted above:
 \begin{align}
 \lab{2f5}
 \kappa(L)&\leqq n(L),\hbox{ with equality if }n(L)=\dim X;\end{align}
 \ssn{Case of  a \vb\ $E\to X$}
 
 \begin{equation}\lab{2f6}  \bmp{5}{$E_\met>0\implies E$ ample, and $E^0_\met >0\implies \kappa(E)=\dim X+r-1$.}\end{equation}
In words,  $E^0_\met>0$ implies that $E$ is big.

We next have
 \begin{equation}
 \lab{2f7}  E_\met>0\implies E_\num>0. \end{equation}
 This result may be found in \cite{Laz04}; the proof is not obvious from the definition.
 We are not aware of any  implication along the lines of 
   $$E_\met\geqq 0\implies E_\num\geqq 0.$$
 
  Next we have 
 \begin{equation}
 \lab{2f9}  \bmp{5}{$\kappa(E)\leqq n(E)$ with equality if $n(E)=\dim X+r-1$    (this follows from \pref{2f5}).}\end{equation}
This leads to 
  \begin{equation}
 \lab{2f8}    E_\met \geqq 0 \hensp{and} E_\num>0\implies \kappa(E)= \dim X+r-1.\end{equation}
 This follows from \pref{2e2} and \pref{2f9}.  
 
 It is \emph{not} the case that
 \[
 E_\num >0\implies E \hbox{ ample;}\]
   there is an example due to Fulton  of a numerically positive vector bundle over a curve that is not ample (cf.\ \cite{Laz04}).

 We will conclude this section with a sampling of well-known results whose proofs illustrate some of the traditional uses of positivity.
 
 \begin{Prop}\lab{2g10}
 If $E\to X$ is a Hermitian \vb\ with $\Theta_E>0$, then $H^0(X,E^\ast)=0$.\end{Prop}
 
 \begin{proof} 
Using \pref{2c4} applied to $E^\ast\to X$, if $s\in H^0(X,E^\ast)$ then when we evaluate $\part\ol\part\log\|s\|^2$ at a strict maximum point  where the Hessian is definite we obtain a contradiction.
 If the maximum is not strict then the usual perturbation argument may be used.
 \end{proof}
\begin{Prop}\lab{2g11}
If $E\to X$ is a Hermitian \vb\ of rank $r\leqq \dim X$ and with $\Theta_E>0$,  then every seciton $s\in H^0(X,E)$ has a zero.\end{Prop}

\begin{proof} The argument is similar to the preceding proposition, only this time we assume that $s$ has no zero and evaluate $\part\ol\part\log\|s\|^2$ at a minimum.  This (1,1) form is positive in the pullback to $X$ of the vertical tangent space, and it is negative in the pullback of the horizontal tangent space.  The assumption $r\leqq \dim X$ then guarantees that it has at least one negative eigenvalue.\end{proof}

\begin{Prop}\lab{2g12}
If $E\to X$ is a Hermitian \vb\ with $\Theta_E\geqq 0$ and $s\in H^0(X,E)$ satisfies $\Theta_E(s)=0$, then $Ds=0$.\end{Prop}
\begin{proof} Let $\om$ be a K\"ahler form on $X$.  Then
\[
\part\ol\part (s,s) = (Ds,Ds)+(s,\Theta_E(s))=(Ds,Ds)\geqq 0,\]
and if $\dim X=d$ using Stokes theorem we have
\[
0=\int_X \om^{d-1} \wedge (i/2) \part\ol\part \|s\|^2 = \int_X \om^{d-1}\wedge (i/2) (Ds,Ds)\]
which gives the result.\end{proof}

 \subsection{A further result}\hfill

 As discussed above, for many purposes  positivity is too strong (see the   examples in Section II.B) and semi-positivity is too weak  (adding the trivial bundle to a semi-positive bundle gives one that  is semi-positive).  One desires a more subtle notion than just $\Theta_E\geqq 0$.  With this in mind, a specific guiding question for these notes has been the 
\begin{quest}  
Suppose that one has $\Theta_E\geqq 0$ and for $\wedge^r E=\det E$ we have $\Theta_{\det E}>0$ on an open set; that is, $E$ is strongly semi-positive. Does this enable one to produce lots of  sections of $\Sym^mE$ for $m\gg0$?\end{quest} 

The following is a response to this question:
\begin{Thm}\lab{2h1}
Suppose that $E\to X$ is a Hermitian \vb\ of rank $r$ that is strongly semi-positive. Then $\Sym^m E\to X$ is big for any $m\geqq r$.\end{Thm}

\subsubsection*{Proof}
Setting $S^r E=\Sym^r E$, we have $\Theta_{S^\ast E}\geqq 0$.  Let $\om_r$ be the curvature form for
$\cO_{\P S^r(E)}(1)$.  Then $\om_r \geqq 0$, and we will show that
\begin{equation}\lab{2h2}
\hbox{$\om_r>0$ on an open set.}\end{equation}
For this it will suffice to find one point $p=(x,[e_1\cdots e_r])\in \P S^r (E)_x$ where \pref{2h2} holds.  Let $x\in X$ be a point where $(\Tr \Theta_E)(x)>0$ and  let $e_1,\dots,e_r$ be a unitary basis for $E_x$.  Then some
\[
\lra{\lrp{\Theta_E(e_i),e_i} ,\xi\wedge \ol\xi}>0.\]
We may assume that the Hermitian matrix
\[
\lra{\lrp{\Theta_E(e_i),e_j},\xi\wedge \ol\xi} = \delta_{ij}\la_i,\quad \la_i\geqq 0\]
is diagonalized.  Then
\begin{align*}
\lra{\lrp{\Theta_{S^r E}(e_1\cdots e_r),e_1\cdots e_r}, \xi\wedge\ol\xi} &= 
\sum_i \Big\langle\Big(\Theta_{S^r E}(e_i) e_1\; \raise8pt\hbox{$\begin{matrix} i\\[-5pt] \hat{\cdots}\end{matrix}$} \; e_r,e_1  \cdots e_r\Big) ,\xi\wedge \ol\xi\Big\rangle\\
&=  \sum_i \la_i>0. \end{align*}
The same argument works for any $m\geqq r$. \hfill\qed \medbreak

Using this we are reduced to proving the
\begin{Lem}\lab{n2h3}
Let $F\to X$ be a  Hermitian vector bundle with $\Theta_F\geqq 0$, and where this is a point $x\in X$ and $f\in F_x$ such that the $(1,1)$ form $\Theta_F(f,\cdot)$ is positive.  Then for  the curvature form $\om_F$ of the \lb\ $\cO_{\P F}(1)\to \P F$, we have $\om_F>0$ at the point $(x,[f])\in \P F$.
\end{Lem}

\begin{proof}
Let $f^\ast_1,\dots,f^\ast_r$ be a local holomorphic frame for $F^\ast \to X$ and
\[
\sig = u([a_1,\dots,a_r];x)\sum_i a_if^\ast_i\]
a local holomorphic section of $\cO_{\P E}(1)\to \P F$, where   $a_1,\dots,a_r$ are   variables defined modulo the scaling action $(a_1,\dots,a_r)\to \la(a_1,\dots,a_r)$ and   $u([a_1,\dots,a_r],x)$ is holomorphic.  We have
\[
\| \sig\| = |u|^2 \sum_{i,j} h_{i\bar j}\cdot a_i \bar a_j\]
where $h_{i\bar j}=(f^\ast_i,f^\ast_j)$ is the metric in $F\to X$; up  to a constant
\[
\om_F = \part\ol\part \log \|\sig\|^2.\]
We may choose our frame and scaling parameter so that at the point $(x,[f])$
\begin{equation}\lab{n2h4}
h_{i\bar j}(x)=\delta_{ij},\qquad dh_{ij}(x)=0\hensp{and} \|\sig(x,[f])\|=1.\end{equation}
Computing $\part\ol\part\log \|\sig\|^2$ and evaluating at the point where \pref{n2h4} holds any cross-terms involving $dh_{\ol{ij}}(x)$ drop out and we obtain
\[
\om_F=\sum_{ij}(\part\ol\part h_{i\bar j})(x) a_i\bar a_j+\lrp{
\sum_{i,j} da_i\wedge\bar da_i-\lrp{\sum a_i\cdot da_i} \wedge \lrp{\ol{\sum_j a_j da_j}}}.\]
When we take the scaling action into account and use Cauchy-Schwarz it follows that  that $\om_F>0$ in $T_{(x,[f])}\P F$.
\end{proof}

Remark that the point $(x,[e_1\cdots e_r])$ corresponding to a decomposable tensor in $S^r E_x$ is very special.  Easy examples show that we do not expect to have $\om_r>0$ everywhere.  In fact, the exterior differential system (EDS)
\[
\om_r =0\]
is of   interest and will be discussed in Section VI.E in the situation when the curvature has the norm positivity property to be introduced in Section III.A.

\begin{Exam}  \lab{2h3}We will illustrate the mechanism of how passing to $S^r E$ increases the Kodaira-Iitaka dimension of the bundle.  Let $E\to G(2,4)$ denote the dual of the universal sub-bundle.  As above, points of $G(2,4)$ will be denoted by $\La$ and thought of as lines in $\P^3$.  For $v\in \La$ we denote by $[v]$ the corresponding line in $\C^4$.  Points of $\P E$ will be $(\La,v)$
\[
\begin{picture}(150,45)
\put(0,0){\line(4,3){50}}
\put(26,20){\circle*{4}}
\put(32,14){$v$}
\put(52,35){$\La$}
\end{picture}\]
and then the fibre of $\cope(1)$ at $(\La,v)$ is $[v]$.
The fibre $E_\La\cong \La^\ast$, and we have
\[
\xymatrix@R=1pt{ H^0(G(2,4),E)\ar[r]&E_\La\ar[r] &0\\
\sideeq &\sideeq\\
C^{4\ast} \ar[r]&\La^\ast\ar[r] & 0.}\]
The tangent space
\[
T_\La G(2,4)\cong \Hom(\La,\C^4/\La)\]
is isomorphic to the horizontal space $H_{(\La,v)}\subset T_{(\La,v)} \P E$.  As previously noted,  for $\xi\in T_\La G(2,4)$ we have
\begin{align}\lab{2h4}
\Theta_E(v,\xi)= \om(\xi)=0\iff \xi(v)=0\\ \notag
\begin{picture}(150,50)
\put(0,0){\line(4,3){50}}
\put(26,20){\circle*{4}}
\put(32,14){$v$}
\put(52,35){$\La$}
\dashline[+30]{3}(13,-4)(40,43)
\put(19,-4){$\La_\xi$}
\end{picture}\end{align}
Here $\La_\xi$ is the infinitesimal displacement of $\La$ in the direction $\xi$.

We   observe  that
\begin{equation}\lab{2h5}
\vp_{\cope(1)} :\P E\to \P^3\end{equation}
is the tautological map $(\La,v)\to [v]$, and consequently the fibre of \pref{2h5} through $v$ is the $\P^2$ of lines in $\P^3$ through $v$.  The tangent space to this fibre are the $\xi$'s as pictured above.  We note that $\dim \P E=5$ while $\kappa(\cope(1))=n(\cope(1))=3$.

Points of $\P S^2 E$ are   $(\La,v,v')$
\[
\begin{picture}(150,45)
\put(0,0){\line(4,3){50}}
\put(16,12.5){\circle*{4}}
\put(22,7){$v$}
\put(34,25.5){\circle*{4}}
\put(39,21){$v'$}
\put(52,35){$\La$}\end{picture}\]
and unless $v=v'$ we have 
\[
\Theta_E(v\cdot v',\xi)\ne 0\]
for any non-zero $\xi\in T_\La G(2,4)$.  Thus for $\om_2$ the curvature form of $\cO_{\P S^2 E}(1)$ we have
\[
\om_{2}>0\hensp{at} (\La,v\cdot v')\]
unless $v=v'$; consequently $S^2 E$ is big.
\end{Exam}

We shall give some further observations and remarks concerning the question posed at the beginning of this section.

\begin{Prop}\lab{2h6}
If $E\to X$ is generically globally generated and $\det E$ is big, then $S^k E$ is big for some $k>0$.\end{Prop}

\begin{proof} By standard arguments passing to a blowup of $X$ and pulling $E$ back, we may reduce to the case where $E$ is  globally generated.  Let $N=h^0(X,E)$ and denote  by $Q\to G(N-r,N)$ the universal quotient bundle over the Grassmannian. We then  have a diagram
\begin{equation}\lab{2h7}
\xymatrix{\P E\ar[d]\ar[r]^\alpha&\P Q\ar[d]\ar[r]^\beta&\P^{N-1}\\
X\ar[r]^(.3)f& G(N-r,N)&}\end{equation}
where
\begin{itemize}
\item $f^\ast Q=E$;
\item $\beta\circ \alpha=\vp_{\cope}(1)$ \end{itemize}
where $\vp_{\P E}(1)$ is the map induced by $H^0(\cope(1))$.
As a metric in $E\to X$ we use the one induced by the standard metric on $Q\to G(N-r,N)$.  Then we claim that
\begin{itemize}
\item $\Theta_E\geqq 0$;
\item $\Tr  \Theta_E>0$ on an open set.\end{itemize}
The first of these is clear.  For the second, $\det Q$ is an ample line bundle over $G(N-r,r)$ and $\det E=f^\ast \det Q$.  Thus
\[
\Theta_{\det E}=\Tr  \Theta_E \geqq 0\]
and for $d=\dim X$
\[
(\Tr \Theta_E)^d\equiv 0\]
contradicts the assumption that $\det E$ is big.  The proposition now follows from Theorem \ref{2h1}.
\end{proof}

This proposition will be used in connection with the following well known
\begin{Prop}\lab{2h8}
If $S^k E$ is big for some $k>0$, then there exist arbitrarily large $\ell$ such that $S^{k  \ell}E$ is generically globally generated.\end{Prop}

\begin{proof}    As a general comment, for any holomorphic \vb\ $F\to X$
\[
\lrc{\begin{array}{c}\hbox{$F$ is generically}\\ \hbox{globally generated} \end{array}} \iff  \lrc{\begin{array}{c}\hbox{$\cO_{\P F}(1)\to \P F$ is generically}\\ \hbox{globally generated}\end{array}}.\]
Since sufficiently high powers of a big \lb\ are generically globally generated, and since by definition $F\to X$ is big if the \lb\ $\cO_{\P F}(1)$ is big, we have
\[
F\hbox{ big}\implies S^n F\to X\hbox{ is generically globally generated.}\]
Taking $F=S^k E$,
our assumption then  implies that there are arbitrarily large $\ell$ such that $S^\ell(S^kE)$ is generically globally generated, consequently  the direct summand $S^{k \ell}E$ of $S^\ell (S^kE)$ is also generically globally generated.\end{proof}

\begin{Cor}\lab{2h9}
If $E\geqq 0$ and $\det E>0$, then $S^mE$ is big for arbitrarily large~$m$.\end{Cor}
\begin{proof}
By Theorem \ref{2h1}, $S^m E$ is generically globally generated and $\det S^m E>0$ so that \pref{2h6} applies.\end{proof}

\begin{Rem}\lab{2h12}
For a line bundle $L\to X$ we consider the properties
{\setbox0\hbox{(iii)}\leftmargini=\wd0 \advance\leftmargini\labelsep
 \begin{enumerate}
\item $L$ is \emph{nef};
\item $L$ is \emph{big};
\item $L$ is \emph{free} (frequently expressed by saying that $L$ is \emph{semi-ample}).
\end{enumerate}} 

For vector bundles $E\to X$ one has the corresponding  properties using $\cO(1)$.  In (ii) and (iii) one generally uses $\Sym^{m\P E} E$ rather than just $E$ itself. 

Although all these properties are important, in some sense (iii) is the most stable (note that (iii)$\implies$(i)).  Repeating a discussion from the introduction, 
suppose that we have in $L\to X$ a Hermitian metric with Chern form $\om$ such that
\begin{equation}\lab{2h13}
\bmp{5}{$\om\geqq 0$ and $\om>0$ on $X\bsl Z$ where $Z\subset X$ is a normal crossing divisor.}\end{equation}
Then $\om$ defines a K\"ahler metric $\om^\ast$ on $X^\ast := X\bsl Z$, and then as in \pref{1a10n} one may ask the
\begin{equation}\lab{2h14}
\bmp{4.75}{{\bf Question:} \emph{Are there properties of the metric $\om^\ast$, especially those involving its curvature $R_{\om^\ast}$, that imply that $L\to X$ is free?}}\end{equation}
\end{Rem}
 
 \section{Norm positivity}
 \subsection{Definition and first properties} \hfill
 
 As   previously noted, for many purposes including those arising from Hodge theory  strict positivity of a holomorphic vector bundle in the sense of \pref{2b2} and ~\pref{2b3} is too strong, whereas semi-positivity  is too weak.  The main observation of this section is that for many bundles that arise naturally  in algebraic geometry the curvature has a special form, one that implies semi-positivity in both of the above senses and where in examples the special form  has  Hodge-theoretic and algebro-geometric interpretations.
 
 \begin{Defin} \lab{3a1}
 Let $E\to X$ be a Hermitian vector bundle with curvature $\Theta_E$.  Then $\Theta_E$ has the \emph{norm positivity} property if there is a Hermitian vector bundle $G\to X$ and a holomorphic bundle mapping
 \begin{equation}\lab{3a2}
 A:TX\otimes E\to G\end{equation}
 such that for $x\in X$ and $e\in E_x$, $\xi\in T_xX$
 \begin{equation}\lab{3a3}
 \Theta_E(e,\xi)=\|A(\xi\otimes e)\|^2_G.\end{equation}
 
 Here we are identifying $E_x$ with $E^\ast_x$ using the metric, and  $\|\enspace\|^2_G$ denotes the square norm in $G$.
 In matrix terms, relative to unitary frames in $E$
 and $G$ there will be a matrix $A$ of $(1,0)$ forms such that the curvature matrix is given by 
 \begin{equation}\lab{3a4}
 \Theta_E=-{}^t\ol{A}\wedge A.\end{equation}
 We note that \pref{3a3} will hold for any tensors in $TX\otimes E$, not just decomposable ones.  As a consequence $E\to X$ is semi-positive in both senses \pref{2b2} and \pref{2b3}.\end{Defin}
 
The main implications of norm positivity will use the following observation:
\begin{equation}\lab{3a5}
\bmp{5}{\em If the curvatures of Hermitian bundles $E,E'\to X$ have the norm positivity property, then the same is true for $E\oplus E'\to X$ and $E\otimes E'\to X$, as well as Hermitian direct summands of these bundles.}\end{equation}

 \begin{proof}
 If we have $A:TX\otimes E\to G$ and $A':TX\otimes E'\to G'$, then $\Theta_{E\otimes E'}=( \Theta_E\otimes \Id_{E'})\oplus (\Id_E\otimes \Theta_{E'})$, and
 \[
 (A\otimes \Id_{E'})\oplus (\Id_E\otimes A'):TX\otimes E\otimes E'\to (G\otimes E')\oplus (E\otimes G')\]
leads to norm positivity for $\Theta_{E\otimes E'}$.  The argument for $\oplus$ is evident.\end{proof}

 \subsection{A result using norm positivity}\hfill
 
 The idea is this: For this discussion we abbreviate
 \[
 T=T_x X, \quad E=E_x,\quad G=G_x \]
 and identify $E\cong E^\ast$ using the metric.  
We have a linear mapping
 \begin{equation}\lab{3b1}
 A:T\otimes E\to G,\end{equation}
 and using \pref{3a3} non-degeneracy properties of this mapping will imply positivity properties of $\Theta_E$.  Moreover, in examples the mapping $A$ will have algebro-geometric meaning so that  algebro-geometric assumptions will lead to positivity properties of the curvature.
 
 The simplest non-degeneracy property of \pref{3b1} is that $A$ is injective; this seems to not so frequently happen in practice.  The next simplest is that $A$ has injectivity properties in each factor separately.  Specifically  we consider the two conditions
 \begin{equation}\lab{3b2}
  A:T\to \Hom(E,G)\hensp{is injective;}\end{equation}
  \begin{equation}
\lab{3b3}
 \bmp{5}{for general $e\in E$, the mapping $A(e):T\to G$ given by 
\[
A(e)(\xi)=A(\xi\otimes e),\quad \xi\in T\]
is injective.}\end{equation}
The geometric meanings of these are:
\begin{equation}\lab{3b4}
\bmp{5}{\pref{3b2} \emph{is equivalent to having
\[
\Theta_{\det E}=\Tr\,\Theta_E>0\]
at $x$}; and}\end{equation}
\begin{equation}\lab{3b5}
\bmp{5}{\pref{3b3} \emph{is equivalent to having
\[
\om>0\]
at $(x,[e])\in (\P E)_x$.}}\end{equation}
This gives the 
\begin{prop}\lab{3b6} If $E\to X$ has a metric whose curvature has the norm  positivity property, then
\begin{enumerate}[{\rm (i)}]
\item \pref{2b2}$\implies \det E$  {is big};
\item \pref{3b3} $\implies E$ {is big}.\end{enumerate}
\end{prop}
A bit more subtle is the following result, which although it is a consequence of Theorem \ref{2h1} and \pref{3b3}, for later use we shall give another  proof.
\begin{Thm}  \lab{3b7}If the rank $r$ bundle  $E\to X$ has a metric whose curvature has the norm  positivity property, then
\[
\hbox{\pref{2b2} $\implies  S^r E$ {is big}.}\]
\end{Thm}
 
  \begin{Cor} \lab{3b8}
 With the assumptions in \pref{3b7} the map
 \[
 H^0(X,S^m(S^r E)) \to S^m(S^rE)\]
 is generically surjective for $m\gg 0$.\footnote{Then this also implies that the map
 \[
 H^0(X,S^m(S^rE))\to S^{mr}\]
 is generically surjective.}\end{Cor}

 \begin{proof}[Proof of Theorem \ref{3b7}] Keeping the above notations and working at a general point  in $\P E$ over $x\in X$, given $\xi\in T_xX$ and a basis $e_1,\dots, e_r$ of $E_x$ from \pref{3b3} we have
 \begin{equation}\lab{3b10}
 \sum^r_{i=1}\| A(\xi\otimes e_i)\|^2_G\ne 0.\end{equation}
 Then using \pref{3a5} for the induced map
 \[
 A:T_xX\otimes S^r E_x\to S^{r-1} E_x\otimes G_x\]
 from \pref{3b10}  for $\om_r$ the canonical (1,1) form on $\P S^r E$  at the point $(x,[e_1\cdots e_r])$  \[
 \lra{\om_r ,\xi\wedge \bar\xi}>0.\qedhere\]
 \end{proof}
 
 \begin{rem} Viehweg (\cite{Vie83a}) introduced the notion of \emph{weak positivity} for a coherent sheaf.  For \vb s this means that for  any ample \lb\ $L\to X$ there is a $k>0$ such that the evaluation mapping 
 \[
 H^0(X,S^\ell(S^k E\otimes L))\to S^\ell (S^k E\otimes L)\] 
 is generically surjective for $\ell\gg 0$.  He then shows   that for the particular bundles that arise in the proof of the Iitaka conjecture if one has $\det E>0$ on an open set, and then an intricate cohomological argument gives that $E$ is weakly positive.  As will be explained in Section V.A of these notes,  \pref{3b8} may be used to circumvent the need for weak positivity in this case.
  \end{rem}
  
  We note that the ample line bundle $L\to X$ is not needed in \pref{3b8}.  We also point out that
  \[
  E_\met \geqq 0 \implies E\hbox{ is weakly positive (cf.\ \cite{Pau16}).}\]
  This is plausible since $S^k E \geqq 0$ and $L>0\implies S^kE\otimes L>0$.  In loc.\ cit.\  this result  is extended to important situations where the metrics have certain types of singularities.

  We conclude this section with a discussion of the Chern forms of bundles having the norm positivity property, including the \hvb.
  \begin{Prop}\lab{3b10}
  The linear mapping $A$ induces 
  \[
  \wedge^q A:\wedge^q T\to \wedge^q G\otimes S^q E,\]
  and up to a universal constant
  \[
  c_q(\Theta)=\|\wedge^q A\|^2.\]
  \end{Prop}
  
  \begin{proof}
  The notation means
  \[
  \| \wedge^q A\|^2 = (\wedge^q A,\wedge^q A)\]
  where in the inner product we use the Hermitian metrics in $G$ and $E$, and we identify
  \[
  \wedge^q T^\ast\otimes \ol{\wedge^q T^\ast}\cong (q,q)\hbox{-part of } \wedge^{2q}(T^\ast\otimes \ol T^\ast).\]
  Then letting $A^\ast$ denote the adjoint of $A$ we have
  \[
  \wedge^q \Theta= \wedge^q A\otimes \wedge^q A^\ast= \wedge^q A\otimes (\wedge^q A)^\ast\]
  and
  \[
  c_q(\Theta)=\Tr \wedge^q (\Theta)=(\wedge^q A,\wedge^q A).\qedhere\]
  \end{proof}
  
  In matrix terms, if 
  \[
  A= \dim G\times \dim T \hbox{ matrix with entries in }E\]
  then
  \[
  \wedge^q A=\lrc{ \begin{matrix} \hbox{matrix whose entries are the $q\times q$ minors of $A$,}\\
  \hbox{where the terms of $E$ are multiplied as polynomials.}\end{matrix}}\]
  It follows that up to a universal constant for the \hvb\ $F$
  \[
  c_q(\Theta_F)=\sum_a \Psi_\alpha\wedge\ol\Psi_\alpha\]
  where the $\Psi_\alpha$ are $(q,0)$ forms.  In particular, any monomial $c_I(\Theta_F)\geqq 0$.
  
  The vanishing of the matrix $\wedge^q \Phi_{\ast,n}$ is {\em not} the same as $\rank\Phi_{\ast,n}<q$.  In fact, 
  \begin{equation}\lab{3b11}
  \rank \Phi_{\ast,n}<q \iff c_1(\Theta_F)^q=0.\end{equation} 
  
  In general we have the 
  
  \begin{Prop}\lab{3b12}
  If $E\to X$ has the norm positivity property, then $ P(\Theta_E)\geqq 0$ for any $P\in \cC$.\end{Prop}
  
  A proof of this appears in \cite{Gri69}.
  
  In the geometric case when we have a VHS arising from a family of smooth varieties  we have the period mapping $\Phi$ with the end piece of differential  being
  \[
  \Phi_{\ast,n}:T_b B\to \Hom\lrp{H^0(\Om^n_{X_b}),H^1(\Om^{n-1}_{X_b})}\]
 and   the algebro-geometric interpretation of \pref{3b11} is standard; e.g., $\Phi_{\ast,n}$ injective is equivalent to local Torelli holding for the $H^{n,0}$-part of the Hodge structure.

  We conclude this subsection with a result that pertains to a question that was raised above.
  
  \begin{Prop}\lab{3b13} If $\Phi:B\to\Ga\bsl D$ has no trivial factors, and if $h^{n,0}\leqq \dim B$ and $H^0(\ol B,F_e)\ne 0$, then
  \[
  c_{h^{n,0}}(F)\ne 0.\]
  \end{Prop}
  \begin{proof} We will first prove the result when $B=\ol B$.  We let $s\in H^0(F,B)$ and assume that $c_{h^{n,0}}(F)=0$.  Then $s$ is everywhere non-zero and we may go to a minimum of $\| s\|^2$.  From Proposition \ref{2g12} we have $D\sigma=0$, which implies that the norm $\|s\|$ is constant and
  \[
  \nabla s=0\]
  where $\nabla$ is the Gauss-Manin connection.  Using the arguments \cite{Gri72} we may conclude that the variation of Hodge structure has a trivial factor.
  
  If $B\ne \ol B$, the arguments given in Section IV below may be adapted to show that the proof still goes through.  The point is the equality of the distributional and formal derivatives that arise in integrating by parts.\end{proof}
 \section{Singularities}
 In recent years the use of singular metrics and their curvatures  in algebraic geometry has become widespread and important.  Here we shall discuss some aspects of this development; the main objective is  to define mild singularities and show that the singularities that  arise in Hodge theory have this property.
 
 One may roughly divide   singularities into three classes:
   \begin{enumerate}[{\rm (i)}]
 \item metrics with \emph{analytic singularities} as defined in \cite{Dem12a} and \cite{Pau16}; these arise in various extensions of the Kodaira vanishing theorem. We shall only briefly discuss these in part to draw a contrast with the next type of singularities that will play a central role in these notes;
 \item metrics with \emph{logarithmic singularities}; these arise in Hodge theory especially in \cite{CKS86}, and also in  \cite{Kol87}, \cite{Zuo}, \cite{Bru16} and \cite{GGLR17}, and for those that do arise in Hodge theory we shall show that they are \emph{mild} as defined below;
 
 \item metrics with \emph{PDE singularities}; these arise in several places including the applications to moduli where important classes of varieties have canonical special metrics. We shall not discuss these here but refer to the survey papers \cite{Don1} and \cite{Don2} and the Bourbaki talk \cite{Dem16} for summaries of results and guides to the literature.\end{enumerate}  
 
 \subsection{Analytic singularities}
 For $E\to X$ a holomorphic \vb\ over a compact complex manifold, these singularities arise from metrics of the form
 \begin{equation}\lab{4a1}
  h=e^{-\vp}h_0\end{equation}
  where $h_0$ is an ordinary smooth metric in the bundle and $\vp$ is a \emph{weight function}, which in practice and   will almost always locally  be of the form
  \begin{equation}\lab{4a2}
  \vp=\vp' + \vp'' \end{equation}
  where $\vp'$ is  plurisubharmonic (psh) and $\vp''$ is smooth.
  
  \begin{defin}
  The metric \pref{4a1} has \emph{analytic singularities} if it locally has the form \pref{4a2} with 
  \[
  \vp' = \alpha\lrp{\log \sum^m_{j=1}|f_j(z)|^{\alpha_j}},\qquad\alpha,\alpha_j>0\]
  and where the $f_j(z)$ are holomorphic functions.\end{defin}
  
  \begin{defin}  Given a weight function $\vp$ of the form \pref{4a2}, we define the subsheaf $I(\vp) \subset \cO_X$ by
  \[
  I(\vp) = \lrc{ f\in \cO_X:\int e^{-\vp} |f|^2<\infty}.\]
  Here the integral is taken over a relatively compact open set in the domain of definition of $f$ using any smooth volume form on $X$.
  \end{defin}
  
  \begin{Prop}[Nadel; cf.\ \cite{Dem12a}]  \lab{4a3}$I(\vp)$ is a coherent sheaf of ideals.\end{Prop}
  \begin{Exam}[\cite{Dem12a}] \lab{4a4}
  For
  \[
  \vp= \log\lrp{|z_1|^{\alpha_1}+\cdots + |z_n|^{\alpha_n}} ,\qquad\alpha_j >0\]
  the corresponding ideal is
  \[
  I(\vp)=\lrc{ z^{\beta_1}_1 \cdots z^{\beta_n}_n:\sum_j (\beta_j+1)/\alpha_j>1}\]
  where the RHS is the ideal generated by the monomials appearing there.    \end{Exam}
  
  Thus for $\beta_j=1$ and $\sum_j \alpha_j=2+\eps$
  \[
  I(\vp)=\mm_0\]
  is the maximal ideal at the origin in $\C^n$.  The blow up using this ideal $I(\vp)$ is the usual blow up $\wt \C^n_0 \xri{\pi}\C^n$ of $\C^n$ at the origin.

   Variants of this construction give powers $\mm^k_0$ of the maximal ideal, and using $z_1,\dots,z_m$ for $m\leqq n$ for suitable choices of the $\alpha_j$ and $\beta_j$ leads to weighted blowups of $\C^n$ along the coordinate subspaces $\C^{n-m}\subset \C^n$.  The use of weight functions with analytic singularities provides a very flexible analytic alternative to the traditional technique of blowing up along subvarieties.
   
   \begin{Exam} \lab{4a5}  Quite different behavior occurs for the weight function 
   \[
   \vp=\log\bp{(-\log |z_1|) \cdots (-\log|z_k|)}.\]
   In this case $I(\vp)=\cO_X$.  We will encounter weight functions of this type in Hodge theory.
   \end{Exam}
   \begin{Exam}\lab{4a6}
   If $L\to X$ is a \lb\ and $s\in H^0(X,L)$ has divisor $(s)=D$, we may define a metric $\|\enspace\|$ in the \lb\ by writing any local section $s'\in\cO_X(L) $ as 
   \[
   s'=fs\]
   where $f$ is a meromorphic function and  then setting
   \[
   \|s'\|=|f|=e^{\log |f|}.\]
   By the Poincar\'e-Lelong formula (\cite{Dem12a}), formally $\part\ol\part \log \|s'\|=0$ but as currents  the Chern form is the (1,1) current
   \[
   \Om_L=[D]\]
   given by integration over the effective divisor $D$.\end{Exam}
   
   We will be considering the case when $E=L$ is a \lb\ with singular metrics of the form \pref{4a1}, \pref{4a2}.  In this case the \emph{curvature form} is given by 
   \begin{equation}\lab{4a7}
   \om_h =(i/2)\ol\part\part\log h=(i/2)\part\ol\part \vp-(i/2)\part\ol\part \log h_0 .\end{equation}
   For $\vp$ given by \pref{4a2} where $\vp'$ is psh,  the singular part of the curvature form is the (1,1) current
   \[
   (i/2)\part\ol\part\vp' \geqq 0\]
   where the inequality is taken in the sense of currents as explained in \cite{Dem12a} and \cite{Pau16}.
   
   \begin{Thm}[Nadel vanishing theorem]
   \lab{4a8}  If $\om_h>0$ in the sense of currents, then
   \[
   H^q(X,(K_X+L)\otimes I(\vp))=0,\qquad q>0.\]
   \end{Thm}
  
  \ssn{Application:}  Assuming that $\om_L>0$ and using a cutoff function to globalize the function $\vp$ in \pref{4a4}, we may replace $L$ by $mL$ to make $(i/2)\part\ol\part\vp +\om_{K_X}+m\om_L>0$.  Then the Nadel vanishing theorem gives $H^1(X,mL\otimes I_x)=0$ from which  we infer that 
  \[
  H^0(X,mL)\to mL_x\to 0.\]
  Thus $mL$ is globally generated, and similar arguments show that for $m\gg 0$ the map
  \[
  \vp_{mL}:X\to \P^{N_m}\]
  is an embedding. 
  
  This proof of the Kodaira embedding theorem illustrates the advantage   of the flexibility provided by the choice of the weight function $\vp$.  Instead of blowing up as in the original Kodaira proof, the use of weight functions achieves the same effect with greater flexibility. 
  
  Another use of singular metrics is given by the
  \begin{Thm}[Kawamata-Viehweg vanishing theorem]
  \lab{4a9}
  If $L\to X$ is big,
  then
  \[
  H^q(X,K_X+L)=0,\qquad q>0.\]\end{Thm} 
   \begin{proof}[Proof {\rm (cf.\ \cite{Dem12a})}]
Let $H\to X$  be a very ample line bundle and $D\in |H|$ a smooth divisor.  From the cohomology sequence of
\[
0\to mL-H\to mL\to mL\big|_D\to 0\]
and
\[
h^0(X,mL)\sim m^d,\; h^0(D,mL\big|_D)\sim m^{d-1}\]
where $d=\dim X$ we have $h^0(mL-H)\ne 0$ for $m\gg 0$.  If $E\in |mL-H|$ from  Example \ref{4a6}   there exists a singular metric in $mL-H$ with Chern form
\[
\Om_{mL-H}=[E]\geqq 0\]
where the RHS is the (1,1) current defined by $E$.  If $\om_H>0$ is the curvature form for a positively curved metric in $H\to X$, then
\[
\om_L = \frac{1}{m}\lrp{[E]+\om_H} \geqq \lrp{\frac{1}{m}} \om_H>0,\]
and Nadel vanishing gives the result.
 \end{proof}
 
 \begin{rem}
 A standard algebro-geometric proof of Kawamata-Viehweg vanishing uses  the branched covering method to reduce it to Kodaira vanishing.    The above argument again illustrates the flexibility gained by the use of singular weight functions (we note that $D$ plays a role similar to that of the branch divisor in the branched covering method).
   \end{rem}
   \subsection{Logarithmic and mild singularities}\label{4b}\hfill
   
As our main applications will be to Hodge theory, 
 in this section we will use the notations from Section I.C.
We recall  from that section that 
{\setbox0\hbox{()}\leftmargini=\wd0 \advance\leftmargini\labelsep
    \begin{itemize}
   \item $B$ is a smooth quasi-projective variety;
   \item $\ol B$ is a smooth projective completion of $B$;
   \item $Z=\ol B\bsl B$ is a divisor with normal crossings 
   \[
   Z=\cup Z_i\]
   where $Z_I:= \bigcap_{i\in I}Z_i$ is  a stratum of $Z$ and   $Z^\ast_I=Z_{I,\text{reg}}$  are  the smooth points of $Z_I$;
   \item $E\to \ol B$ is a holomorphic vector bundle.\end{itemize}}  \noindent 
   A neighborhood $\cU$ in $\ol B$ of a point $p\in Z$ will be
   \[
   \cU\cong \Delta^{\ast k}\times \Delta^\ell\]
   with coordinates $(t,w)=(t_1,\dots,t_k;w_1,\dots,w_\ell)$.
   
   We   now introduce the co-frame in terms of which we shall express the curvature  forms in $\cU$.  The Poincar\'e metric in $\Delta^\ast = \{0<|t|<1\}$ is given by the (1,1) form
   \[
   \om_{\rm PM} = (i/2) \frac{dt\wedge d\bar t}{|t|^2(-\log|t|)^2}.\]
   We are writing $-\log|t|$ instead of just $\log|t|$ because we will want to have positive quantities in the computations below.  As a check on signs and constants we note the formula
   \begin{equation}\lab{4b1}
   (i/2)\part\ol \part \bp{-\log(-\log|t|)} = (1/4)\om_{\rm PM}.\end{equation}
   The inner minus sign is to have $-\log |t|>0$ so that $\log(-\log|t|)$ is defined.  The outer one is to have the expression in parentheses  equal to $-\infty$ at $t=0$ so that we have a psh function.  
For
   \[
   \vp =  \log(-\log|t|)\]
   the curvature form in the trivial bundle over $\Delta$ with the singular metric given by $e^{-\vp}$ has   curvature form
   \begin{equation}\lab{4b2}
   (i/2) \ol\part\part \log(e^{-\vp})=(1/4)\om_{\rm PM}.\end{equation}
   
   \begin{rem}
   The functions that appear as coefficients in formally computing \pref{4b1} using the rules of calculus are all in $L^1_{\rm loc}$ and therefore define distributions.  We may then compute $\part$ and $\ol\part$ either in the sense of currents or formally using the rules of calculus. An important observation is 
   \begin{equation} \lab{4b3}
   \bmp{4}{\em these two methods of computing $\ol\part\part\vp$ give the same result.}\end{equation}
   This is in contrast with the situation when we take
   \[
   \vp=\log|t|\]
   in which case we have in the sense of currents the  Poincar\'e-Lelong formula  
   \begin{equation}\lab{4b4}
   (i/\pi)\part\ol\part\log |t|=\delta_0 \end{equation}
   where $\delta_0$ 
   is the Dirac $\delta$-function at the origin.  Anticipating the discussion below, a charateristic feature of the metrics that arise in Hodge theory will be that the principle \pref{4b3} will   hold.
   \end{rem}
   \begin{defin} The \emph{Poincar\'e coframe} has as basis the (1,0) forms
   \[
   \frac{dt_i}{t_i(-\log |t_i|)},\; dw_\alpha\]
   and their conjugates.\end{defin}
   
   \begin{defin} A metric in the holomorphic \vb\ $E\to B$ is said to have \emph{logarithmic singularities} along the divisor $Z=\ol B\bsl B$   if  locally in an open set $\cU$ as above and  in terms of a holomorphic frame  for the bundles and the Poincar\'e coframe  the metric~$h$, the connection matrix $\theta=h^{-1}\part h$, and the curvature matrix $\Theta_E=\ol\part(h^{-1}\part h)$ have entries that are Laurent polynomials in the $\log|t_i|$ with coefficients that are real analytic functions in $\cU$.\end{defin}
   
   \begin{prop}\lab{4b5} The Hodge metrics in the Hodge bundles $F^p\to B$ have logarithmic singularities relative to the canonically extended Hodge bundles $F^p_e \to \ol B$.\end{prop}
   In the geometric case this result may be inferred from the theorem on regular singular points of the Gauss-Manin connection (\cite{Del}) and \pref{6.1} above.  In the general case it is a consequence of the several variable nilpotent orbit theorem (\cite{CKS86}).  More subtle is the behavior of the coefficients of the various quantities, especially the Chern polynomials $P(\Theta_{F^p})$, when they are expressed in terms of the Poincar\'e frame, a topic analyzed in \cite{CKS86} and  where the analysis is refined in  \cite{Kol87}, and  to which we now turn.
   
   We recall that a distribution $\Psi$ on a manifold $M$ has a \emph{singular support} $\Psi_{\sing}\subset M$  defined by the property that on any open set $W\subset M\bsl \Psi_\sing$ in the complement the restriction $\Psi\big|_W$ is given by a smooth volume form.  A  finer invariant of the singularities of $\Psi$ is given by its \emph{wave front set}\footnote{A good discussion of wave front sets and references to the literature is given in Wikipedia.  We will not use them in a technical sense but rather as a suggestion of an important aspect  to be analyzed for the Chern polynomials of the Hodge bundles.}
   \[
   WF(\Psi)\subset T^\ast M.\]
   Among other things the wave front set  was introduced to help deal with two classical problems concerning distributions:
\begin{equation}\lab{4b6} \bmp{5}{\begin{itemize}
\item[(a)] distributions cannot in general be multiplied;
\item[(b)] in general distributions cannot be restricted to submanifolds $N\subset M$.\end{itemize}}
 \end{equation}
    For (a) the wave front sets should be transverse, and for (b) to define $\Psi\big|_N$ it suffices to have $TN\subset WF(\Psi)^\bot$.
    
   In the case of currents represented as differential forms with distribution coefficients, multiplication should be expressed in terms of the usual wedge product of forms.  For restriction, if $N$ is locally given by $f_1=\cdots = f_m=0$, then for a current $\Psi$ we first set $df_i = 0$; i.e., we cross out  any terms with a $df_i$.  Then the issue is to restrict the distribution coefficients of the remaining terms to $N$.  Thus the notion of  the wave front set for a current $\Psi$   involves both the differential form  terms appearing in $\Psi$ as well as the distribution coefficients of those terms.
   
   \begin{defin} The holomorphic bundle $E\to B$ has \emph{mild logarithmic singularities} in case it has logarithmic singularities and the following conditions are satisfied:
     \begin{enumerate}[{\rm (i)}] \item
   the Chern polynomials $P(\Theta_E)$ are closed currents given by differential forms with $L^1_{\rm loc}$ coefficients and which represent $P(c_1(E),\dots,c_r(E))$ in $H^\ast_{\rm DR}(\ol B)$;
   \item the products $P(\Theta_E)\cdot P'(\Theta_E)$ may be defined by formally multiplying them as $L^1_{\rm loc}$-valued differential forms, and when this is done we obtain a representative in cohomology of the products of the   polynomials in the Chern classes;
   \item the restrictions $P(\Theta_E)\big|_{Z^\ast_I}$ are defined and represent $P\lrp{c_1\lrp{E\big|_{Z^\ast_I}},\dots,c_r\lrp{E_{Z^\ast_I}}}$.\end{enumerate} 
   \end{defin}
   
   We note the opposite aspects of analytic singularities and mild logarithmic singularities: In the former one wants the singularities to create behavior different from that of smooth metrics, either with regard to the functions that are in $L^2$ with respect to the singular metric, or to create non-zero Lelong numbers in the currents that arise from their curvatures.  In the case of mild logarithmic singularities, basically one may work with them as if there were no singularities at all.  An important  additional point  to be explained in more detail below is that the presence of singularities \emph{increases} the positivity of the Chern forms of the Hodge bundles, so that in this sense one  uses singularities to positive effect. 
   
    The main result, stated below and which will be discussed in the next section, is that the Hodge bundles have mild logarithmic singularities.  This would follow if one could show that
   \begin{equation}\lab{4b7} \bmp{5}{\em When expressed in terms of the Poincar\'e frame the polynomials $P(\Theta_E)$ have bounded coefficients.}\end{equation}
   This is true when $Z$ is a smooth divisor, but when  $Z$ is not a smooth divisor  this is not the case and the issue is more subtle. 
     
     \subsection*{Main result}
 
     \begin{Thm}[\cite{CKS86}, with amplifications in  \cite{Kol87}, \cite{GGLR17}]\lab{4b8}
     The Hodge bundles have mild singularities.\end{Thm}
  The general issues  \pref{4b6}(a), \pref{4b6}(b)  
     concerning distributions were raised above.  Since currents are differential forms with distribution coefficients, these issues are also present for currents, where as noted above the restriction issue \pref{4b6}(b) involves both the differential form aspect and the distribution aspect of currents.  This is part (iii) of the definition  and is the   property of the Chern polynomials that appears in \cite{GGLR17}.
     
     The proof of  (i) and part of (ii) in Theorem \ref{4b8} is based on the fundamental results in \cite{CKS86}, with  refinements in \cite{Kol87} concerning a particular multiplicative property  \pref{4b6}(a)  in the definition  of mild logarithmic
      singularities,\footnote{The specific result in \cite{Kol87} is that the integral
     \[
     \int_{\ol B} c_1 \lrp{\Theta_{\det E}}^d<\infty,\quad \dim \ol B=d\]
     of the top power of the Chern form of the \hlb\ is  finite.  This result also follows from the analysis in Section V of \cite{GGLR17}.} and  in \cite{GGLR17} the general multiplicative property and  the restriction property  
     of the Chern polynomials is addressed.   Both of these involve estimates  in  the $\Delta^\ast$-factors in neighborhoods $\cU\cong \Delta^{\ast k}\times \Delta^\ell$    in $\ol B$.  In effect these estimates may be intrinsically thought of as occurring in sectors in the co-normal bundle of the singular support of the Chern forms, and in this sense may be thought of as dealing with the wave front sets of the these forms.
     
     A complete proof of Theorem \ref{4b8} is given in Section 5 of \cite{GGLR17}. In the next section we shall give the argument for the Chern form $\Om = c_1\lrp{\Theta_{\det F}}$ of the \hlb\ and in the special case when  the localized VHS is a nilpotent orbit.  The computation will be  explicit; the intent is to provide a perspective on some of the background  subtleties in the general  argument, one of which we now explain.  
     
     We  restrict to the case when $\cU\cong \Delta^{\ast k}=\{(t_1,\dots,t_k):0<|t_j|<1\}$, and setting
$
    \ell(t_j)= \log t_j/2\pi i$ and $
     x_j=-\log |t_j| $ 
     consider a nilpotent orbit
     \[
     \Phi(t)=\exp \lrp{\sum^k_{j=1}\ell(t_j)N_j} \cdot F_0.\]
     Following explicit computations of the Chern form $\Om$ and of the Chern form $\Om_I$ for the restriction of the \hlb\ to $Z^\ast_I$, the desired result comes down to showing that a limit
     \begin{equation}
     \lab{4b9}
     \lim_{x_j\to\infty} \frac{Q(x)}{P(x)}\end{equation}
     exists where $Q(x), P(x)$ are particular homogeneous polynomials of the same degree with $P(x)>0$ if $x_j>0$.  Limits such as \pref{4b9} certainly do not exist in general, and the issue to be understood is how  in the case at hand  the very special properties of several parameter \lmhs s  imply the existence of the limit.

     As an application of Theorems \ref{3b7} and \ref{4b8}, using the notations from Section I.C we consider a VHS given by a period mapping
     \[
     \Phi:B\to\Ga\bsl D.\]
     Denoting by
     \[
     F_e\to \ol B\]
     the canonically extended \hvb\ we have
     
     \begin{Thm}\label{4b10}
     \leavevmode
     \begin{enumerate}[{\rm (i)}]
     \item The Kodaira-Iitaka dimension
     \[
     \kappa(F_e)\leqq 2h^{n,0}-1.\]
     \item Assuming the injectivity of the end piece $\Phi_{\ast,n}$ of the differential of $\Phi$,
     \[
     \kappa(S^{h^{n,0}}F_e)=\dim \P S^{h^{n,0}} F_e;\]
     i.e., $S^{h^{n,0}}F_e\to \ol B$ is big.\end{enumerate}
     \end{Thm}
     
     \begin{proof}
     It is well known \cite{Gri72}, \cite{CM-SP} that the curvature of the \hvb\ has the norm positivity property.  In fact, the curvature form is given by
     \begin{equation}\lab{4b11}
     \Theta_F(e,\xi) = \| \Phi_{\ast,n}(\xi)(e)\|^2.\end{equation}
Concerning the singularities that arise along $\ol B\bsl B$, it follows from Theorem \ref{4b8} that we may treat the Chern form $\om$ of $\cO_{\P F_e}(1)\to \ol B$ as if the singularities were not present. 
     
     The linear algebra situation is
     \begin{equation}\lab{4b12}
     T\otimes F\to G\end{equation}
     where $\dim T=\dim B$, $\dim F=h^{n,0}$ and $\dim G=h^{n-1,1}$.  By \pref{4b11} condition \pref{3b2} is equivalent to the injectivity of $\Phi_{\ast,n}$, and Theorem \ref{4b10} is then a consequence of Theorem \ref{3b7}.   \end{proof}
     
     This result gives one answer to the question
   {\setbox0\hbox{(1)}\leftmargini=\wd0 \advance\leftmargini\labelsep
   \begin{quote}
     \em The \hvb\ is somewhat positive.  Just how positive is it?\end{quote}
}  \noindent
Since in the geometric case the linear algebra underlying the map \pref{4b12} is expressed cohomologically, in particular cases the result (i) in \pref{4b10} can be considerably sharpened.  For example, in the weight $n=1$ case the method of proof of the theorem gives the
\begin{Prop}[\cite{Bru16b}] \lab{4b13}
In weight $n=1$, {\rm (i)} $\kappa(F_e)\leqq 2g-1$, and {\rm (ii)} $S^2 F_e\to \ol B$ is big.\end{Prop}

\begin{proof} In this case $D\subset \mathcal{H}_g$ where $g=h^{1,0}$  and $\mathcal{H}_g$ is the Siegel generalized upper-half-plane.
 We then have
\begin{itemize}
\item $T\subset S^2V^\ast$;
\item $G=V^\ast$;
\item $T\otimes V\to G$ is induced by the natural contraction map $S^2 V^\ast\otimes V\to V^\ast$.\end{itemize}
For any $v\in V$ the last map has image of dimension $\leqq g$, and therefore the kernel has dimension $\geqq \dim T-g$.  This gives (i) in the proposition.

For (ii) we have
\[\xymatrix@R=1pt{T\otimes S^2V\ar[r]&V^\ast\otimes V\\
\cap&\\
S^2 V^\ast \otimes S^2 V.\ar[uur]_\rfloor}\]
For a general $q\in S^2 V$ the contraction mapping $\rfloor$ is injective, and this implies (ii).
\end{proof}

At the other extreme we have the
\begin{Prop} \lab{4b14}
Let $\cM_{d,n}$ denote the moduli space of smooth hypersurfaces $Y\subset \P^{n+1}$ of degree $d=2n+4$, $n\geqq 3$.  Then the \hvb\ $F\to \cM_{d,n}$ is big.\end{Prop}

\begin{proof}
Set $V=\C^{n+2\ast}$ and let $P\in V^{(d)}$ be a homogeneous form of degree $d$ that defines $Y$.  Denote by $J_P \subset \opplus_{k\geqq d-1} V^{(k)}$ the Jacobian ideal.  Then (cf.\ Section 5 in \cite{CM-SP})
\begin{itemize}
\item $T_Y \cM_{d,n}\cong V^{(d)}/J_P^{(d)}$;
\item $F_Y = H^{n,0}(Y)\cong V^{(d-n-2)}$;
\item $G_Y = H^{n-1,0}(Y)\cong V^{(2d-n-2)}/J^{(2d-n-2)}_P$.
\end{itemize}
It will suffice to show
\begin{equation} \lab{4b15}
\bmp{5}{\em For general $P$ and general $Q\in V^{(d-n-2)}$ the mapping 
\[
 V^{(d)}/J^{(d)}_P \xri{Q} V^{(2d-n-2)}/J^{(2d-n-2)}_P\]
 is injective.}\end{equation}
 Noting that $d-n-2=n+2$ and that it will suffice to prove the statement for one $P$ and $Q$, we take
 \begin{align*}
 Q&= x_0\cdots x_{n+1},\\
 P&= x^d_0+\cdots + x^d_{n+1}.\end{align*}
 Then $J_P = \{ x^{2n+3}_0 ,\dots, x^{2n+3}_{n+1}\}$ and a combinatorial argument gives \pref{4b15}.\end{proof}

 \begin{Rem}\lab{4b16} The general principle that Proposition \ref{4b14} illustrates is this: Let $L\to $ be an ample line bundle.  Then both for general smooth sections $Y\in |mL|$ and for cyclic coverings $\wt X_Y\to X$ branched over a smooth  $Y$, as $m$ increases the \hvb\ $F\to |mL|^0$ over the open set of smooth $Y$'s becomes increasingly positive in the sense that the $k$ such that $S^k F$ is big decreases, and for $m\gg 0$ $F$ itself is big.\end{Rem}
 
 \section{Proof of Theorem \ref{4b8}}
 
 \subsection{Reformulation of the result}
 
 We consider a variation of Hodge structure given by a period mapping
 \[
 \Phi:\Delta^{\ast k}\to \Ga_\loc\bsl D.\]
 Here we assume that the monodromy generators $T_i\in \Aut_Q (V)$ are unipotent with logarithms $N_i \in \End_Q(V)$; $\Ga_\loc$ is the local monodromy group generated by the $T_i$.
 
 For $I\subest \{1,\cdots,k\}$ with complement $I^c = \{1,\dots,k\}\bsl I$ we set
 \[
 \Delta^\ast_I = \{(t_1,\dots,t_k):t_i=0\hensp{for} i\in I\hensp{and} t_j\ne 0\hensp{for} j\in I^c\}.\]
 From the work of Cattani-Kaplan-Schmid \cite{CKS86} the limit $\lim_{t\to \Delta^\ast_I}\Phi(t)$ is defined as a polarized variation of \lmhs s on $\Delta^\ast_I$.  Passing to the primitive parts of the associated graded \phs s gives a period mapping 
 \[
 \Phi_I:\Delta^\ast_I\to \Ga_{\loc,I}\bsl D_I\]
 where $D_I$ is a product of period domains and $\Ga_{\loc,I}$ is generated by the $T_j$ for $j\in I^c$.  This may be suggestively expressed by writting \[
 \lim_{t\to t_I} \Phi(t)=\Phi_I(t_I).\]
 However caution must be taken in interpreting the limit, as the ``rate of convergence" is not uniform but depends on the sector in which the limit is taken in the manner explained in \cite{CKS86}.

 We denote by $\La\to\Delta^{\ast k}$ and $\La_I\to \Delta^\ast_I$ the \hlb s.  The \HR\ bilinear relations give metrics in these bundles and we denote by $\Om$ and $\Om_I$ the respective Chern forms.  The result to be proved is
 \begin{equation}\lab{5a1}
 \lim_{t\to \Delta_I}\Om=\Om_I,\end{equation}
 where again care must be taken in interpreting this equation.
 In more detail, this means: In $\Om$ set $dt_i = d\bar t_i=0$ for $i\in I$.  Then the limit, in the usual sense, as $t\to \Delta_I$ of the remaining terms exists and is equal to $\Om_I$.  We will write \pref{5a1} as
 \begin{equation}\lab{5a2}
 \Om\big|_{\Delta^\ast_I} = \Om_I.\end{equation}
 
 The proof of \pref{5a1} that we shall give can easily be adapted to the case when the period mapping depends on parameters.   
 
The limit  can also be   reduced to the case when $\Phi$ is a \emph{nilpotent orbit}.  This means that
 \begin{equation}\lab{5a3}
 \Phi(t)=\exp\lrp{\sum^k_{i=1} \ell(t_i)N_i}F\end{equation}
 where $F\in\CD$ and the conditions
 \begin{enumerate}[{\rm (i)}]
 \item $N:F^p\to F^{p-1}$,
 \item $\Phi(t)\in D$ for $0<|t|<\eps$
 \end{enumerate}
 are satisfied.  This reduction is non-trivial and  is given in Section 5 of \cite{GGLR17}.
 
 The main points in the proof of Theorem IV.B.8 in the nilpotent orbit case are as follows:
 {\setbox0\hbox{(1)}\leftmargini=\wd0 \advance\leftmargini\labelsep
 \begin{itemize}
 \item[(a)] without changing the associated graded's to $\Phi$ and $\Phi_I$ we may replace the $F$ in \pref{5a3} by an $F_0$ such that the \lmhs\ is $\R$-split;
 \item[(b)] in this case $N_I$ can be  completed to an $\rsl_2$ which we denote by $\{N^+_I,Y_I,N_I\}$;\footnote{There are two ways of doing this---one is the method in \cite{CKS86} and the other one, which is purely linear algebra, is due to Deligne.}
 \item[(c)] the $Y_I$-weight decomposition of $N_{I^c}$ is
 \[
 N_{I^c}=N_{I^c,0}+ N_{I^c,-1}+N_{I^c,-2}+\cdots\]
 where $N_{I^c,-m}$ has $Y_I$-weight $-m$, $m\geqq 0$;
 \item[(d)] if all the $N_{I^c,-m}=0$ for $m>0$, then there is an $\rsl^c_2=\{ N^+_{I^c}, Y_{I^c},N_{I^c}\}$ that commutes with the  previous $\rsl_2$, and the result \pref{5a1} is immediate;
 \item[(e)] in general, by direct computation we have
 \[
 \Om_I \equiv \Om +R\ \mod dt_i,d\bar t_i\hensp{for} i\in I\]
 where the remainder term $R$ consists of expressions $Q(x)/P(x)$ as in \pref{4b9}, and then direct computation  using the relative filtration property and the fact that for $m>0$ the $N_{I^c,-m}$ have negative $Y_I$-weights gives the result.\end{itemize}}  \noindent 
 
 \subsection{Weight filtrations, representations of $\rsl_2$ and \lmhs s}\lab{5b}\hfill
 
 The proof of \pref{5a1} will be computational, using only that $\Phi(t)$ is a nilpotent orbit \pref{5a3} and that the commuting $N_i\in\End_Q(V)$ have the \emph{relative weight filtration property} (RWFP), which will be reviewed below.\footnote{The proof in \cite{GGLR17} uses the detailed analysis of \lmhs s from \cite{CKS86}, of which the RWFP is one consequence.  Part of the point for the argument given here is to isolate the central role played by that property.}   The computation will be facilitated by using the representation theory of $\rsl_2$ adapted to the Hodge theoretic situation at hand.  The non-standard but hopefully suggestive notations for doing this will now be explained.
 
 (i)  Given a nilpotent transformation $N\in \End_Q(V)$ with $N^{n+1}=0$ there is a unique increasing weight filtration $W(N)$ given by subspaces
 \begin{equation}\lab{5b1}
 V^{W(N)}_k:= W_k(N)V\end{equation}
 satisfying the conditions
 \begin{itemize}
 \item $N:V^{W(N)}_k\to V^{W(N)}_{k-2}$,
 \item $N^k:V^{W(N)}_{n+k}\simto V^{W(N)}_{n-k}$ \qquad\qquad (Hard Lefschetz property).\end{itemize}
 Remark that the two standard choices for the ranges of indices in \pref{5b1} are
 \[
 \bcs 0\leqq k\leqq 2n &\qquad\hbox{(Hodge theoretic)}\\
 -n\leqq k\leqq n&\qquad\hbox{(representation theoretic)}.\ecs\]
 We will use the first of these.
 
 The weight filtration is   self-dual in the sense that using the bilinear form $Q$
 \begin{equation}\lab{5b2}
 V^{W(N)\bot}_k=V^{W(N)}_{2n-k-1}\end{equation}
 which gives
 \[
 V^{W(N)\ast}_k\cong V/V^{W(N)}_{2n-k-1}.\]
 The \emph{associated graded} to the weight filtation is the direct sum of the
 \[
 \Gr^{W(N)}_\ell V:= V^{W(N)}_\ell / V^{W(N)}_{\ell-1},\]
 and the \emph{primitive subspaces} are defined for $\ell\geqq n$ by
 \[
 \Gr^{W(N)}_{n+k,\prim}V=\ker\lrc{ N^{k+1}:\Gr^{W(N)}_{n+k}V\to \Gr^{W(N)}_{n-k-2}V}.\]
 
 (ii) A \emph{grading element} for $W(N)$ is given by a semi-simple $Y\in \End_Q(V)$ with integral eigenvalues $0,1,\dots,2n$, weight spaces  $V_k\subset V^{W(N)}_k$ for the eigenvalue $k$,  and where the induced maps
 \[
 V_k\simto \Gr^{W(N)}_k V\]
are isomorphisms.  Thus
 \[
 V^{W(N)}_k = \opplus^k_{\ell=0} V_\ell.\]
 Grading elements always exists, and for any one such $Y$  we have
 \begin{itemize}
 \item $[Y,N]=-2N$;
 \item there is a unique $N^+\in\End_Q(V)$ such that $\{N^+,Y,N\}$ is an $\rsl_2$-triple.
 \end{itemize}
 The proof of the second of these uses the first together with the Hard Lefschetz property of $W(N)$.  
 
 We denote by $\cU$ the standard representation of $\rsl_2$ with weights $0,1,2$.  Thinking of $\cU$ as degree 2 homogeneous polynomials in $x,y$ we have
 \begin{itemize}
 \item weight $x^a y^b = 2a$, $a+b=2$;
 \item $N=\part_x$ and $N^+=\part_y$.\end{itemize}
 We denote by
 \[
 \cU_i = \Sym^i \cU\cong \lrc{ \begin{matrix}
 \hbox{homogeneous polynomials}\\
 \hbox{ in $x,y$ of degree $i+1$}\end{matrix}}\]
 the standard $(i+1)$-dimensional irreducible representation  of $\rsl_2$.  The $N$-\emph{string} associated to $\cU_i$ is
 \[
 \{x^{i+1}\} \to \{x^i y\} \to\cdots\to \{y^{i+1}\}\]
 where $N=\part_x$.  The top of the $N$-string is the primitive space.
 
 Given $(V,Q,N)$ as above and a choice of a grading element $Y$, for the $\rsl_2$-module
 \[
 V_{\gr}= \opplus^{2n}_{k=0}V_k\]
 we have a unique identification
 \begin{equation}\lab{5b3}
 V_\gr \cong \opplus^n_{i=0} H^{n-i}\otimes \cU_i\end{equation}
 for vector spaces $H^{n-i}$.  The notation is chosen for Hodge-theoretic purposes.  The $N$-string associated to $H^{n-i}\otimes \cU_i$ will be denoted by
 \begin{equation}\lab{5b4}
 H^{n-i} (-i)\xri{N}H^{n-i}(-(i-1))\xri{N}\cdots \xri{N} H^{n-i} \end{equation}
 and we define
 \begin{equation}\lab{5b5}
 \hbox{\em the Hodge-theoretic weight of $H^{n-i}(-j)$ is $n-i+2j$.}\end{equation}
 The representation-theoretic weight of $H^{n-i}(-j)$ is $2j$.  It follows that $H^{n-i}(-i)$ is the primitive part of the $\cU_i$-component of $V_{\gr}$.
 
 Relative to $Q$ the decomposition \pref{5b3} is orthogonal and the pairing
 \[
 Q_i:H^{n-i}(-i) \otimes H^{n-i}(-i)\to \Q\]
 given by
 \begin{equation}\lab{5b6}
 Q_i(u,v)=Q(N^i u,v)\end{equation}
 is non-degenerate.
 
 (iii) We recall the 
 \begin{defin} A \emph{\lmhs}\ (LMHS) is a \mhs\ $(V,Q,W(N),F)$ with weight filtration $W(N)$ defined by a nilpotent $N\in \End_Q(V)$ and Hodge filtration $F$ which satisfies the conditions
   \begin{enumerate}[{\rm (a)}]
\item $N:F^p\to F^{p-1}$;
\item  the form $Q_i$ in \pref{5b6} polarizes $\Gr^{W(N)}_{n+k,\prim}V\cong H^{n-k}(-k)$.\end{enumerate}
 \end{defin}

The MHS on $V$ induces one on $\End_Q(V)$, and (a) is equivalent to
\[
N\in F^{-1}\End_Q(V).\]

We denote by
\[
V_{\C} = \opplus_{p,q} I^{p,q}\]
the unique Deligne decomposition of $V_\C$ that satisfies
\begin{itemize}
\item $W_k(N)V=\opplus_{p+q\leqq k}I^{p,q}$;
\item $F^pV = \opplus_{{p'\geqq p\atop q}} I^{p',q}$;
\item $\ol I^{q,p} \equiv I^{p,q}\mod W_{p+q-2}(N)V$.\end{itemize}
The LMHS is $\R$-\emph{split} if $\ol I^{p,q}=I^{q,p}$.  Canonically associated to a LMHS is an $\R$-split one $(V,Q,W(N),F_0)$ where
\[
F_0 = e^{-\delta}F\]
for a canonical  $\delta\in I^{-1,-1}\End_Q(V_R)$.  For this $\R$-split LMHS there is an evident grading element $Y\in I^{0,0}(\End_Q(V_\R))$.

Given a LMHS $(V,Q,W(N),F)$ there is an associated nilpotent orbit
\[
\xymatrix@R=1pt{\Delta^\ast \ar[r]& \Ga_T\bsl D\\
\sidein&\sidein\\
t\ar[r]& \exp(\ell(t)N)F}\]
where $\ell(t)=\log t/2\pi i$ and $\Ga_T=\{T^\Z\}$.  Conversely, given a 1-variable nilpotent orbit as described above there is a LMHS.  We shall use consistently the  bijective correspondence
\[
\hbox{LMHS's} \iff \hbox{1-parameter nilpotent orbits}.\]
Since
\[
\det F_0 =\det F\]
without loss of generality for the purposes of this paper  we will assume that our LMHS's are $\R$-split and therefore have canonical grading elements.

(iv) Let $N_1,N_2\in \End_Q(V)$ be commuting nilpotent transformations and set $N=N_1+N_2$.  Then there are two generally different  filtrations defined on the vector space $\Gr^{W(N_1)}V$:
\begin{itemize}
\item[(A)] the weight filtration $W(N)V$ induces a filtration on any sub-quotient space of $V$, and hence induces a filtration on $\Gr^{W(N_1)}_\bullet  V$;
\item[(B)] $N$ induces a nilpotent map $\ol N: \Gr^{W(N_1)}_\bullet  V\to \Gr^{W(N_1)}_\bullet  V$, and consequently there is   an associated weight filtration $W(\ol N)\Gr^{W(N_1)}_\bullet V$ on $\Gr^{W(N_1)}_\bullet V$.\end{itemize}

\begin{defin} The \emph{relative weight filtration property} (RWFP) is that these two filtrations coincide:
\begin{equation}\lab{5b7}
W(N)\cap \Gr^{W(N_1)}_\bullet V=W  (\ol N)\Gr^{W(N_1)}_\bullet V.\footnotemark\end{equation}
\footnotetext{There is a shift in indices that will not be needed here (cf.\ \pref{6a13} below).}%
We note that $\ol N$ is the same as the map induced by $N_2$ on $\Gr^{W(N_1)}_\bullet V$, so that \pref{5b7} may  be perhaps more suggestively written as
\begin{equation}
\lab{5b8}
W(N)\cap \Gr^{W(N_1)}_\bullet V=W  (N_2) \Gr^{W(N_1)}_\bullet V.\end{equation}
The RWFP is a highly non-generic condition on a pair of commuting nilpotent transformation, one that will be satisfied in our Hodge-theoretic context.
\end{defin}

(v) Suppose now that $Y_1$ is a grading element for $N_1$ so that the corresponding $\rsl_2 = \{N^+_1 , Y_1,N_1\}$ acts on $V$ and hence on $\End_Q(V)$.  We observe that the $Y_1$-eigenspace deomposition of $N_2$ is of the form
\begin{equation}\lab{5b9}
N_2=N_{2,0}+N_{2,-1}+\cdots +N_{2,-m},\qquad m>0\end{equation}
where $[Y,N_{2,-m}]=-mN_{2,-m}$.  The reason for this is that
\[
[N_1,N_2]=0\implies \bigg\{\raise6pt\hbox{\bmp{2.75}{$N_2$ is at the bottom of the $N_1$-strings for $N_1$ acting on $\End_Q(V)$}}\bigg\}.\]
It can be shown that there is an $\rsl'_2=\{N^+_2 ,Y_2,N_{2,0}\}$ that commutes with the $\rsl_2$ above.  Thus
\begin{equation}\lab{5b10}
\bmp{5}{
Given $N_1,N_2$ as above, there are commuting $\rsl_2$'s with $N_1$ and $N_{2,0}$ as nil-negative elements.  Moreover, $N_2= N_{2,0}+$(terms of \emph{strictly negative} weights) relative to $\{N^+_1,Y_1,N_1\}$. }
\end{equation}
It is the ``strictly negative" that will be an  essential ingredient needed to establish that the limit exists in the main result.

\subsection{Calculation of the Chern forms $\Om$ and $\Om_I$}\label{5c}
\hfill

\ssn{Step 1:} For a nilpotent orbit \pref{5a3} holomorphic sections of the canonically extended VHS over $\Delta$ are given by
\[
\exp\lrp{\sum^k_{j=1} \ell(t_j)N_j}v,\qquad v  \in V_\C.\]
Up to non-zero constants the Hodge metric is
\begin{align*}(u,v)&= Q\lrp{\exp\lrp{\sum_j\ell(t_j)N_j}u, \exp\lrp{ {\sum_j \ell(t_j)N }}\bar v}\\
&=Q\lrp{\exp\lrp{\sum_j \log |t_j|^2  N_j}u,\bar v}.\end{align*}

Using the notation \pref{5b3} the associated  graded to the LMHS as $t\to 0$ will be written as 
\begin{equation} \lab{5c1} V_{\gr} = \opplus^n_{i=0} H^{n-i} \otimes \cU_i\end{equation}
and
\[
F^n = \opplus^n_{i=0} H^{n-i,0}.\]
For $u\in H^{n-i,0}(-i)\hensp{and} v\in H^{n-i,0}$
\[ 
Q\lrp{\exp\lrp{\sum_j \log |t_j|^2N_j}u,\bar v} 
 = \lrp{\frac{1}{i!}} Q\lrp{\lrp{\sum_j  \log |t_j|^2 N_j}^i u,\bar v } .\]
 Setting 
 \[
 x_j = -\log |t_j|\]
 the metric on the canonically extended \lb\ is a non-zero constant times
 \begin{equation}\lab{5c2}
 P(x) = \prod^n_{i=0} \det \lrp{\lrp{\sum_j x_j N_j\big|_{H^{n-i,0}}}^i}.\end{equation}
 Here to define ``det" we  set $N=\sum N_j$ and  are identifying $H^{n-i}(-i)$ with $H^{n-i}$ using $N^i$.  Note that the homogenous polynomial $P(x)$  is positive in the quadrant $x_j>0$.
 The Chern form is 
 \begin{equation}\lab{5c3}
 \Om=\part\ol\part \log P(x).\end{equation}
 
 \ssn{Step 2:} Define
 \[
 \bcs \displaystyle N_I=\sum_{i\in I }x_i N_i,\; N_{I^c}= \sum_{j\not\in I}x_j N_j\\
  \displaystyle  N=\sum^k_{i=1} x_i N_i = N_I+N_{I^c} \ecs\]
  and set
   \[
 P=\prod^n_{i=0} \det \lrp{N\big|_{H^{n-i,0}(-i)}}^i.\]
 Denoting by
 \[
 V_{\gr ,I}=\opplus^n_{i=0} H^{n-i,i}_I\otimes \cU_i\]
 the associated graded to the LMHS as $t\to\Delta^\ast_I$, we define
 \[
 P_I = \prod^n_{i=0} \det\lrp{ N_I\big|_{H^{n-i,0}_I (-i)}}^i.\]
 Taking $N_I=N_1$ and $N_{I^c}=N_2$ in (iv) in Section \ref{5b}, we have
 \[
 N_{I^c,0} = \hbox{ weight zero component of $N_{I^c}$}\]
 where weights are relative to the grading element $Y_I$ for $N_I$.  Decomposing the RHS of \pref{5c1} using the $\rsl_2\times \rsl'_2$ corresponding to $N_I$ and $N_{I^c,0}$ by  \pref{4b10}  we obtain
 \begin{equation}\lab{4c4}
 V_{\gr}\cong \opplus_{i,j} H^{n-i-j}_{i,j}\otimes \cU_i \otimes \cU_j\end{equation}
 where $H^{n-i-j}_{i,j}$ is a \phs\ of weight $n-i-j$. Note that this decomposition depends on $I$.   On $H^{n-i-j}_{i,j}\otimes \cU_i \otimes \cU_j$ we have a commutative square
 \[
 \xymatrix{ H^{n-i-j}(-i-j)\ar[r]^{N^i_I}\ar[d]_{N^j_{I^c,0}}& H^{n-i-j}_{i,j}(-j)\ar[d]^{N^j_{I^c,0}}\\
 H^{n-i-j}_{i,j}(-i)\ar[r]^{N^i_I}& H^{n-i-j}_{i,j}.}\]
 Using \pref{4c4} this gives 
 \[
 P=\prod_{i,j}\det\lrp{ N^i_I N^j_{I^c,0}\big|_{H^{n-i-j}_{i,j}(-i-j)}}+R\]
 where the remainder term $R$ involves the $N_{I^c,-m}$'s for $m>0$.  We may factor the RHS to have
 \[
 P=\prod_{i,j} \det \lrp{ N^i_I\big|_{H^{n-i-j}_{i,j}(-i-j)}}\prod_{i,j} \det\lrp{N^j_{I^c,0}\big|_{H^{n-i-j}_{i,j}(-i-j)}}+R\]
 which we write as
 \begin{equation}\lab{5c5}
 P=P_I\cdot P_{I^c} + R\end{equation}
 where $P_I$ and $P_{I^c}$ are the two $\prod_{i,j}$ factors.  We note that
 \begin{equation}\lab{5c6}
 \hbox{\em the remainder term $R=0$ if we have commuting $\rsl_2$'s.}\footnotemark\end{equation}
 \footnotetext{To have commuting $\rsl_2$'s means that $N_{2,-m}=0$ for $m>0$.}
 
 We next have the important observation 
 \begin{Lem}\lab{5c7} $P_{I^c}$ is the Hodge metric in the \lb\ $\La_I\to Z^\ast_I$.\end{Lem}
 \begin{proof} This is a consequence of the RWFP \pref{5b7} applied to the situation at hand when we take $N_1 = N_I$ and $N_2=N_{I^c}$.
 \end{proof}
 
 By \pref{5c6}, if we have commuting $\rsl_2$'s, then $R=0$   
 \begin{align*}
 \Om= -\part\ol\part \log P&= -\part\ol\part \log P_I - \part\ol\part R_{I^c}\\
 &\equiv -\part\ol\part \log P_{I^c} \hensp{modulo} dt_i,d\bar t_i\hensp{for} i\in I\\
 &\equiv \Om_I\end{align*}
 and we are done.
 
 In general, we have
 \begin{equation}\lab{5c8}
 \Om\equiv \Om_I+S_1+S_2\end{equation}
 where
 \begin{equation}\lab{5c9}
 \bcs S_1 =\displaystyle \frac{\part P_{I^c}\wedge \ol\part R+\part R\wedge \ol\part P_{I^c}-P_{I^c}\part\ol\part R}{P_I P_{I^c}}\\
 S_2 =\displaystyle \frac{\part R\wedge \ol\part R}{P^2_I P^2_{I^c}}.\ecs\end{equation}
 
 \ssn{Step 3:} We will now use specific calculations to analyze the correction terms $S_1,S_2$.  The key point will be to use that  
 \vspth \[
 N=N_I +N_{I^c}  = \underbrace{N_I+N_{I^c,0}} +\underbrace{\sum_{m\geqq 1}N_{I^c,-m}} \vspth\]
 where the terms over the first brackets may be thought of as ``the commuting $\rsl_2$-part of $N_I,N_{I^c}$'' and the correction term over the second bracket has negative $Y_I$-weights.
 
 We set
 \vspth \[
 h^{n-i,0}_I=\dim H^{n-i,0}_I \vspth\]
 and for a monomial $M=x^{\ell_1}_1\cdots x^{\ell_k}_k$ we define
 \vspth\[
 \deg_I M = \sum_{i\in I} \ell_i.\]
 
 \begin{Lem}\lab{5c10}  \leavevmode
  \begin{enumerate}[{\rm (i)}]
 \item For any monomial $M$ appearing in $P$
 \[
 \deg_I M\leqq nh^0_I + (n-1) h^{1,0}_I+\cdots + h^{n-i,0}_I = \sum^n_{i=1} ih^{n-i,0}_I.\]
 \item If $\pi$ is any permutation of $1,\dots,k$ and
 \vspth\[
 \ell_{\pi,i} = \sum^n_{j=1} j\lrp{h^{n-j,0}_{\{\pi(1),\dots,\pi(i)\}} - h^{n-j,0}_{\{\pi(1),\dots,\pi(i-1)\}}} \vspth\]
 then
  \vspth\[
 M_\pi:= x^{\ell_{\pi,1}}_{\pi(1)} \cdot  x^{\ell_{\pi,2}}_{\pi(2)} \cdots x^{\ell_{\pi,k}}_{\pi(k)} = x^{\ell_{\pi,\pi^{-1}(k)}}_1 \cdots x^{\ell_{\pi,\pi^{-1}(k)}}_k \vspth\]
 appears with a non-zero coefficient in $P$.\end{enumerate}     \end{Lem}

 \begin{cor} The monomials appearing in $P$ are in the convex hull of the monomials $M_\pi$.\end{cor}
 
 \begin{proof} For $V_{\gr,I}=\Gr^{W(N_I)}V$ we have as $\{N^+_I,Y_I,N_I\}$-modules
 \[
 V_{\gr,I}\cong \opplus^n_{i=0} H^{n-i}_I\otimes \cU_i.\]
 Decomposing the RHS as $\rsl'_2$-modules we have
 \[
 V_{\gr,I}\cong \opplus^n_{i=0} H^{n-i}_{a,i-a}\otimes \cU_a \otimes \cU'_{i-a}\]
 where the $H^{n-i}_{a,i-a}$ depend on $I$.  The map
 \[
 N^i :H^{n-i,0}(-i) \to H^{n-i}\]
 gives  
 \[
 \opplus^i_{a=0} H^{n-i,0}_{a,i-a}(-i)\to \opplus^i_{a=0} H^{n-i,0}_{a,i-a}.\]
 The $Y_I$-weights of vectors in $H^{n-i,0}_{a,i-a}$ are equal to $a$, and thus $\wedge^{h^{n-i,0}} \lrp{\opplus^i_{a=0} H^{n-i,0}_{a,i-a}(-i)}$ has weight $\sum a h^{n-i,0}_{a,i-a}$ and $\wedge^{h^{n-i,0}} \lrp{\opplus^i_{a=0} H^{n-i,0}_{a,i-a}}$ has weight $-\sum a h^{n-i,0}_{a,i-a}$. As a consequence
 \begin{quote} \em any monomial in $\det \lrp{N^i\big|_{H^{n-i,0}(-i)}}$ drops weights by $2\sum h^{n-i,0}_{a,i-a}$.\end{quote}
 We have
 \begin{equation}\lab{5c11}
 \det\lrp{ \lrp{N\big|_{H^{n-i,0}(-i)}}^i}= \det\lrp{\lrp{\lrp{N_I+N_{I^c,0}}\big|_{H^{n-i,0}(-i)}}^i}+T\end{equation}
 where $T=$ terms involving $N_{I^c,\mathrm{neg}}$.  For any monomial $M$ in a minor involving the $N_{I^c,\mathrm{neg}}$ of total weight $-d$,  
 \[
 2\deg_I M+d=2\sum^i_{a=0} ih^{n-i,0}_{a,i-a},\qquad d>0\]
 and so
 \[
 \deg_I M <\sum^i_{a=0}ah^{n-i,0}_{a,i-a}.\]
 Putting everything together, we have \pref{5c10} where
 \begin{equation}\lab{5c12}
 T=\lrc{ \begin{matrix} \hbox{linear combinations of monomials $M$}\\
 \hbox{satisfying } \deg_IM<\sum^n_{i=0} \sum^i_{a=0} ah^{n-i,0}_{a,i-a}\end{matrix}}.\end{equation}
 Using the bookkeeping formula $h^{n-i,0}_I=\sum^n_{i=a} h^{n-i,0}_{a,i-a}$ we obtain
 \[
 \sum^n_{i=0} \sum^i_{a=0} ah^{n-i,0}_{a,i-a}=\sum^i_{a=0} \sum^n_{i=a} h^{n-i,0}_{a,i-a}=\sum^n_{a=0} ah^{n-a,0}_I \]
 which gives
 \[
 P=P_I P_{I^c}+\lrp{\hbox{correction term with }\deg_I<\sum^n_{a=0} ah^{n-a,0}_I}\]
 where $\deg_I P_I=\sum^n_{a=0} ah^{n-a,0}_I$ and $\deg_I P_{I^c}=0$, giving (i) in \pref{5c10}.
 
 A parallel argument shows that for $I\cap J=\emptyset$  
 \begin{align*}
 D_{I\cup J} &:= \prod^n_{i=0}\det\lrp{\lrp{N_I+N_{I,0}\big|_{H^{n-i,0}_{I\cup J}(-i)}}}^i
 \\ &\quad{}+ \hbox{ a correction with }\deg_I <\sum^m_{a=0} a h^{n-a,0}_I. \end{align*}
 By the definition of $H^{n-i}_{I\cup J}$,
 \[
 \det \lrp{\lrp{N_I+N_{J,0}\big|_{H^{n-i,0}_{I\cup J}(-i)}}^i}  \ne 0\]
 and
 \[ \deg_I \lrp{\det \lrp{\lrp{N_I+N_{J,0}\big|_{H^{n-i,0}_{I\cup J}(-i)}}^i}}= \sum^n_{a=0} a h^{n-a,0}_I  \]
 while automatically 
 \[
 \deg_{I\cup J} (\hbox{all terms of }D_{I\cup J}) = \sum^n_{a=0} a h^{n-a,o}_{I\cup J}.\]
 Thus 
 \begin{align*}
  \deg_{J} \det \lrp{\lrp{ N_I+N_{J,0} \big|_{H^{n-i,0}_{I\cup J}(-i)}}^i}
 & =\deg_{I\cup J}\det \lrp{\lrp{ N_I+N_{J,0} \big|_{H^{n-i,0}_{I\cup J}(-i)}}^i} 
\\&\quad{} - \deg_I \det\lrp{\lrp{N_I+N_{J,0}\big|_{H^{n-i,0}_{I\cup J}(-i)}}^i}\\
 &= \sum^n_{a=0} a \lrp{h^{n-a,0}_{I\cup J}-h^{n-a,0}_I}.\end{align*}
 Proceeding inductively on $\{\pi(1)\}\subset \{\pi(1),\pi(2)\}\subset\cdots \subset \{\pi(1),\dots,\pi(k)\}$
 we obtain, if
 $
 N_{\{\pi(1),\dots,\pi(\ell)\},0}=\hbox{weight 0 piece of } N_{\{\pi(1),\dots,\pi(\ell)\}}$  with respect to $\Gr^{W(N)_{\{\pi(1),\dots,\pi(k)\}}}$ then
 \[
 \prod^n_{i=0} \det\lrp{\lrp{ N_{\{\pi(1)\}} + N_{\{ \pi(1),\pi(2)\},0} + \cdots + N_{\{\pi(1),\dots,\pi(k)\},0} \big|_{H^{n-i,0}}}^i}\]
 is a \emph{non-zero} multiple of $
 x^{\ell_1}_{\pi(1)} x^{\ell_2}_{\pi(2)}\cdot\;\cdot\; x^{\ell_k}_{\pi(k)}$.  This is our $M_\pi$.  Tracking the correction  terms we have 
 \[
 P= \sum_\pi C_\pi M_\pi + \hbox{ terms strictly in the convex hull of the $M_\pi$ }\]
 where $C_\pi\ne 0$ for all $\pi$.
 This proves (ii) in \ref{5c10}.\end{proof}
 
 \ssn{Step 4:} Referring to \pref{5c8} and \pref{5c9}, from Lemma \ref{5c10} we have:
 \begin{itemize}
 \item[(a)] $R_1$ has $\deg_IR_1<\deg_I P_I$, and all monomials satisfy (i) in \ref{5c10}.
 \item[(b)] $R_2$ is a sum of products of monomials $M_1M_2$ where each $\deg_IM < \deg_I P_I$ and $M_i$ satisfies (i) in \ref{5c10}.\end{itemize}
 To complete the proof we have
 \begin{Lem}\lab{5c13}
 Given a monomial $M$ in the $I$-variables satisfying $\deg_I M<\deg_I P$ and {\rm (ii)} in Lemma \ref{5c10},
 \[
 \lim_{t\to\Delta^\ast_I} M/P_I = 0.\]\end{Lem}
 
 \begin{proof}
 Implicit in the lemma is that the limit exists.  We note that $t\to \Delta^\ast_I$ is the same as $x_i\to \infty$ for $i\in I$.  We also observe that the assumptions in the lemma imply that there is a  positive degree monomial $M'$ with $\deg_I(M'M)= \deg_IP_I$ and where $M'M$ lies in the convex hull of the $M_\pi$'s for $P_I$.  Using this convex hull property we will show that
 \begin{equation}\lab{5c14}
 M'M/P_I\hensp{is bounded as $x_i\to \infty$ for} i\in I.\end{equation}
 Since $\lim_{x_i\to\infty} M'(x)=\infty$, this will establish the lemma.
 
 We now turn to the proof of \pref{5c14}.  Because the numerator and denominator are homogeneous of the same degree, the ratio is the same for $(x_1,\dots,x_k)$ and $(\la x_1,\dots,\la x_k)$, $\la>0$.
 
 For simplicity, reindex so that $I=\{1,\dots,d\}$.  Suppose that  $x_\nu = (x_{\nu1},\dots,x_{\nu d})$ is a sequence of points in $( x_i>0 , i\in I)$  such that
 \[
 \lim_{\nu\to\infty} \frac{M'M(x_\nu)}{D_I(x_\nu)} = \infty.\]
 Consider a successive set of subsequences such that for all $i,j$, we have one of three possibilities:
 \begin{enumerate}[{\rm (i)}]
 \item $\lim_{\nu\to\infty} x_{\nu i}/x_{\nu j} = \infty$;
 \item $x_{\nu i}/x_{\nu j}$ is bounded above and below, which we write as $x_{\nu i}\equiv x_{\nu j}$;
 \item $\lim_{\nu\to\infty} x_{\nu i}/x_{\nu j}=0$.\footnote{In effect we are doing a sectoral analysis in the co-normal bundle to the stratum $\Delta^\ast-I$, which explains the wave front set analogy mentioned above.}
 \end{enumerate}
 Now replace our sequence by this subsequence.
 Let $I_{i},\dots,I_r$ be the partition of $I$ such that
 \[
 i \equiv j \iff   \hensp{(ii) holds for} i,j\]
 and order them so that
 (i) holds for $i,j\iff i\in I_{m_1},j\in I_{m_2}$ and $m_1<m_2$.  We may thus find a $C>0$ such that 
$
 \frac{1}{C}\leqq x_{\nu i}/x_{\nu j}\leqq C$ if $i,j$ in same $I_m$, and for any $B>0$ 
 \[
 x_{\nu i}/x_{\nu j}>B^{m_2-m_1}\hensp{if} i\in I_{m_1},j\in I_{m_2}, \nu \hbox{ sufficiently large.}\]
 By compactness, we may pick a subsequence so that $\lim_{\nu\to\infty}(x_{\nu i}/x_{\nu j})=C_{ij}$ if $i,j\in$ same $I_m$.  
 
 Now introduce variables $y_1,\dots, y_d$ and let
 \[
 x_i = a_i y_m\hensp{if} i\in I_m,\quad a_i/a_j=C_{ij}, a_i>0.\]
 We may restrict our cone by taking
 \[
 \wt N_m = \sum_{i\in I_m} a_i N_i.\]
 This reduces us to the case $|I_m|=1$ for all $m$, i.e., 
 \[
 \lim_{\nu \to \infty} x_{\nu i}/x_{\nu j}=\infty \hbox{ if $i<j$.}\footnotemark\]
 \footnotetext{In effect we are making a generalized base change $\Delta^{\ast d}\to \Delta^{\ast k}$ such that for the pullback to $\Delta^{\ast d}$ the   coordinates $y_m$ go to infinity at different rates.}
  Thus for any $B$,
 \[
 x_{\nu i}/x_{\nu j}>B^{j-i} \hensp{for} v\gg 0.\]
 Now
 \[
\frac{x^{m_1}_{\nu_1} x^{m_2}_{\nu_2}\cdot\;\cdot\; x^{m_d}_{\nu_d}}{x^{\ell_1}_{\nu_1} x^{\ell_2}_{\nu_2}\cdot\;\cdot\; x^{\ell_d}_{\nu_d}} \to 0
\]
if $m_2+m_2+ \cdot\;\cdot\; + m_d = \ell_2+\ell_2 +\cdot\;\cdot\;+ \ell_d$ and $m_1<\ell_1$, or $m_1=\ell_1$ and $m_2<\ell_2,\cdot\;\cdot$.
 Thus
 \[
 P_I = cM_{\{1,2,\dots,d\}} + \hbox{terms of slower growth as $\nu\to \infty, c>0$,} \]
 i.e., $$(M_{\{1,2,\dots,d\}}/\hbox{other terms)}(x_\nu)>B.$$ Since $M'M$ belongs to the convex hull of the $M_\pi$, $(M'M/M_{\{1,2,\dots,d\}})(x_\nu)$ is bounded as $\nu\to\infty$.  This proves the claim.
 \end{proof}
 
 \begin{Exam}  \lab{5c15} An example that illustrates most of the essential points in the argument is provided by a neighborhood $\Delta^3$ of the dollar bill curve
 \[ \hspace*{-1.5in}\begin{picture}(140,45)
 \put(14,15){$\$\; \longleftrightarrow$}
 \put(60,0){ \includegraphics{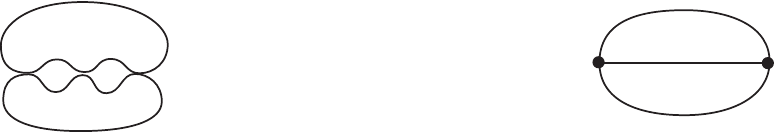}}
 \put(120,17){with the dual graph}\end{picture}\] 
 in $\ol\cM_2$.  The family may be pictured as follows: \pagebreak
 \begin{figure}[h]
\[ \begin{picture}(350,180)
 \put(0,0){\includegraphics{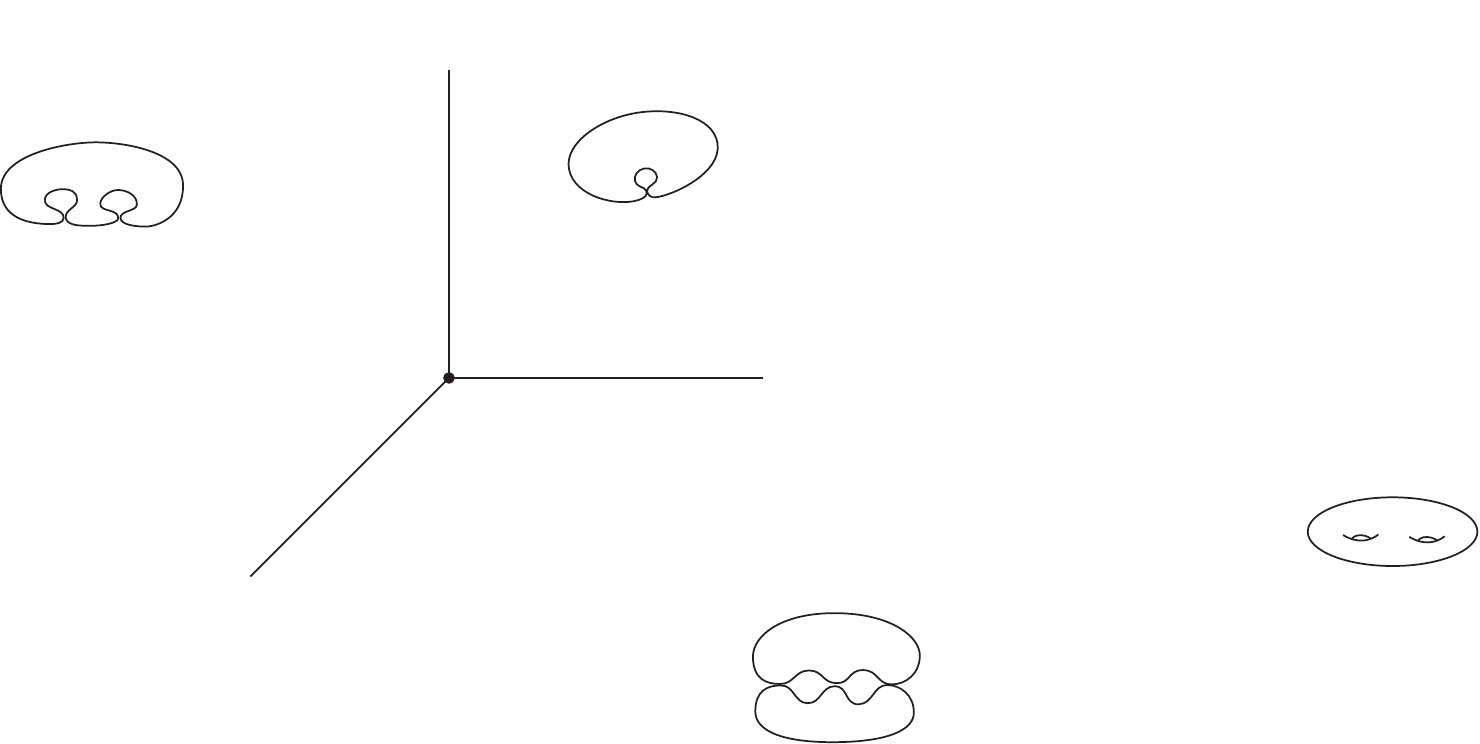}}
 \put(220,180){\bmp{1.25}{\small one nodal curve\hfill\break on the coordinate\hfill\break plane}}
 \put(-20,130){\bmp{1.25}{\small two nodal curve on the coordinate axis}}
 \put(130,20){\bmp{1.25}{\small dollar bill curve\hfill\break at the origin}}
 \put(75,130){\vector(1,0){40}}
 \put(164,38){\vector(-1,2){29}}
 \put(295,60){\small smooth curve}
 \put(295,45){\small in $\Delta^{\ast 3}$}
 \end{picture}\]\vspace*{12pt}\caption{}\label{fig1}\end{figure}

 With the picture 
 \[ \begin{picture}(150,80)
\put(0,0){\includegraphics{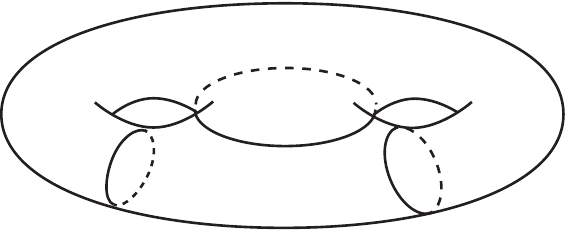}}
\put(33,-10){$\delta_1$}
\put(90,49){$\delta_3$}
\put(120,-10){$\delta_2$}
 \end{picture}\vspace*{22pt}\]
    each of the coordinate planes outside the axes is a family of nodal curves where one of the vanishing cycles $\delta_i$ has shrunk to a point.  Along each of the coordinate axes two of the three cycles have shrunk to a second node, and at the origin we have the dollar bill curve.
 
 We complete the $\delta_i$ to a symplectic basis by adding cycles $\ga_i$.\footnote{Here $\ga_3 $ is not drawn~in.}
   \[ \begin{picture}(150,86)
\put(0,10){\includegraphics{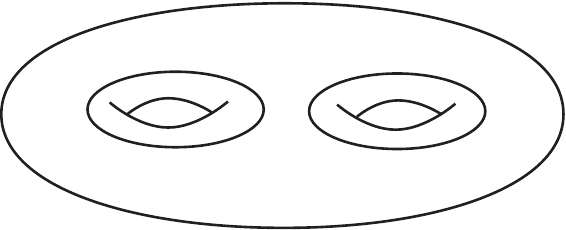}}
\put(17,35){$\ga_1$}
\put(144,35){$\ga_2$} \end{picture}
 \]
 The corresponding monodromies around the coordinate axes are Picard-Lefschetz transformations with logarithms 
 \[
 N_i(\mu)=(\mu,\delta_i)\delta_i\]
and with matrices
 \[
 N_1 = \left(\begin{array}{c|c}
 0 & \begin{smallmatrix} 1&0\\ 0&1\end{smallmatrix}\\
 \hline
 0&0\end{array}\right) , \quad N_2 = \left(\begin{array}{c|c}
 0 & \begin{smallmatrix} 0&0\\ 0&1\end{smallmatrix}\\
 \hline
 0&0\end{array}\right) , \quad  N_3 = \left(\begin{array}{c|c}
 0 & \begin{smallmatrix} 1&1\\ 1&1\end{smallmatrix}\\
 \hline
 0&0\end{array}\right).  \]
 Setting as usual $\ell (t)=\log t/2\pi i$, the normalized period matrix is
 \[
 \Om(t) = \begin{pmatrix} \ell(t_1)+\ell(t_3)&\ell(t_3)\\
 \ell(t_3)&\ell(t_2)+\ell(t_3)\end{pmatrix} + \begin{pmatrix} \hbox{holomorphic}\\ \hbox{term}\end{pmatrix}.\]
 The corresponding nilpotent orbit is obtained by taking the value at the $t_i=0$ of the holomorphic term, and by rescaling this term may be eliminated.
 
 Setting $L(t)=(-\log |t|)/4\pi^2$ and  $PM(t)=(i/2)\ol\part\part\log L(t)$, the metric in the canonically framed Hodge \vb\ is the $2\times 2$ Hermitian matrix
 \[
 H(t)=\begin{pmatrix}
 L(t_1)+L(t_3)&L(t_3)\\
 L(t_3)&L(t_2)+L(t_3)\end{pmatrix};\]
 in the Hodge line bundle the metric is 
 \begin{align*}
 h(t)&= L(t_1)L(t_2)+(L(t_1)+L(t_2)) L(t_3)\\
 &= L(t_1)L(t_2)+L(t_1t_2)L(t_3).\end{align*}
 Setting  
 \begin{align*}
     \om&= \part\ol\part \log h(t)=\part\ol\part \bp{\log (L(t_1)L(t_2)+L(t_1t_2)L(t_3))}\\
 \om_3&= \part\ol\part \log L(t_1t_2)\end{align*}
 we will show   that
 \begin{equation}\lab{5c16}
 \om\big|_{t_3=0} \hensp{\em is defined and is equal to}\om_3.\end{equation}
 \end{Exam}
 
 \begin{proof}
 Setting $\psi=\part h/h$ and $\eta=\part\ol\part h/h$ we have
 \[
 \om=-\psi\wedge\ol\psi + \eta.\]
 Now
 \[
 \psi=\frac{\part\lrp{L(t_1) L(t_2)+L(t_2t_2)L(t_3)}}{L(t_1)L(t_2)+L(t_1t_2)L(t_3)}.\]
 Setting $dt_3=0$ the dominant term of what is left is the left-hand term in
 \[
 \frac{\part L(t_1t_2)}{\frac{L(t_1)L(t_2)}{L(t_3)} +L(t_1t_2)} \longrightarrow \frac{\part L(t_1t_1)}{L(t_1t_2)},\]
 and the arrow means that the limit as $t_3\to 0$ exists and is equal to the term on the right. 
 
 For $\eta$, letting $\equiv$ denote modulo $dt_3$ and $\ol{dt_3}$ and taking the limit as above
 \[
 \eta\equiv \frac{\part\ol\part L(t_1t_2)}{\frac{L(t_1)L(t_2)}{L(t_3)} + L(t_1t_2)} \longrightarrow
 \frac{\part\ol\part L(t_1t_2)}{L(t_1t_2)},\]
 which gives the result.
 \end{proof}
 
 We next observe that
 \begin{equation}\lab{5c17}
 \om_3\big|_{t_2=0} \hensp{\em is defined and is equal to zero.}\end{equation}
 \begin{proof}
 The computation is similar to but simpler than that in the proof of \pref{5c16}.\end{proof}
 
 \ssni{Interpretations}  The curves pictured in Figure 1 map to  an open set $\Delta^3 \subset \ol\cM_2$.  The PHS's of the smooth curves in $\Delta^{\ast 3}$ vary with three parameters.  Those on the codimension 1 strata such as $\Delta^{\ast 2}\times \Delta$ vary in moduli with two parameters.  Their normalizations are
\[ \begin{picture}(300,80)
\put(0,10){\includegraphics{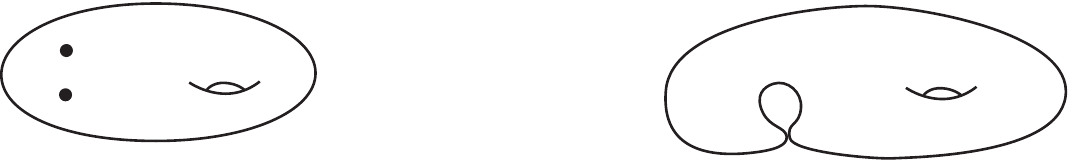}}
\put(50,0){$E$}
\put(30,27){$q$}
\put(30,45){$p$}
\put(110,35){\vector(1,0){60}}
\end{picture}\]
and their LMHS's vary with two parameters with
\[
\bcs
\Gr_1(\hbox{LMHS})\cong H^1(E)\\
\Gr_0(\hbox{LMHS})\cong \Q\ecs\]
and where the extension data in the LMHS is locally  given by $\AJ_E(p-q)$.  Thus $\Gr$(LMHS) varies with one parameter and for the approximating nilpotent orbit is constant along the curves $t_1t_2=c$.  This local fibre of the map $\ol\cM_2\to \ol M$ is part of the closed fibre parametrized by $E$.

Along the codimension 2 strata such as $\Delta^\ast\times \Delta^2$ the curves vary in moduli with 1-parameter.  Their normalizations are
\[ \begin{picture}(300,80)
\put(0,10){\includegraphics{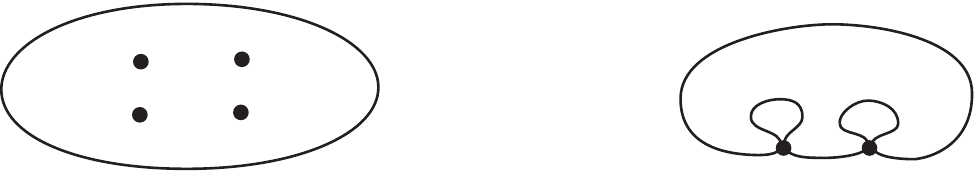}}
 \put(28,25){$q$}
\put(28,40){$p$}
\put(75,40){$p'$}
\put(75,25){$q'$}
\put(125,35){\vector(1,0){55}}
\end{picture}\]
 and the moduli parameter is locally the cross-ratio of $\{p,q;p',q'\}$.  The LMHS's are purely Hodge-Tate and thus $\Gr$(LMHS) has no continuous parameters.  In summary
 \begin{itemize}
 \item $\Phi_e$ is locally 1-1 on $\Delta^{\ast 3}$;
 \item $\Phi_{e,\ast}$ has rank 1 on $\Delta^{\ast 2}\times \Delta$;
 \item $\Phi_e$ is locally constant on $\Delta^\ast \times \Delta^2$.\end{itemize}
 As $c\to 0$ the fibres of $\Phi_e$ on the $\Delta^{\ast 2}\times \Delta$ tend to the coordinate axis $\Delta^\ast \times \Delta^2$ along which $\Phi_e$ is locally constant.
 
 \subsection*{Proof of Theorem I.A.14 in the $\dim B=2$ case}
 We are given $\Phi:B\to \Ga\bsl D$ where $B=\ol B\bsl Z$ with $Z=\sum Z_i$ a normal crossing divisor whose components are irreducible, smooth curves around which the monodromy $T_i$ is unipotent and non-trivial.  We may also assume that $\Phi_\ast$ is generically 1-1.  The mapping $\Phi$ is then proper and the images $M=\Phi(B)\subset\Ga\bsl D$ is a closed complex surface.  We want to show that $M$ has a canonical completion to a complex analytic surface $\ol M$ over which the \hlb\ over $M$ extends to an ample line bundle $\La_e\to\ol M$.  We choose our labelling so that the LMHS is locally constant along the $Z_i$ for $i=1,\dots,k$ and is not locally constant along the remaining $Z_j$.  Then the curves $Z_i$ for $i=1,\dots,m$ are the ones to be contracted to give $\ol M$.  It is classical that the criterion to be able to contract the $Z_i$ to normal singularities is that the intersection matrix
 \begin{equation}\lab{5c18}
 \| Z_i\cdot Z_j\|_{1\leq i,j\leqq k}<0\end{equation}
 be negative definite.
 
 Since the monodromy of LMHS along $Z_i$ is finite, the canonically extended \hlb
 \[
 \La_{e,i} \to Z_i\]
 is of finite order, and moreover the canonically extended Chern $\om_e$ may be restricted to $Z_i$ and
 \[
 \om_i = \om_{e,i}\big|_{Z_i}=0\]
 where $\om_i$ is the Chern form of the \hlb\ $\La_i = \La_e \big|_{Z_i}$.   Since $\om^2_e >0$ our desired result \pref{5c18} follows from the Hodge index theorem.  Since $\La_e \big|_{Z_i}$ is of finite order a power of it is trivial and therefore for some $m$ the line bundle $\La^m_e\to\ol B$ descends to a line bundle $\La^m_{e,\ol M}\to \ol M$.  Although $\ol M$ is singular and $\om_e$ only descends to a singular differential form on $\ol M$, the classical Kodaira theorem may be extended to this case to give that $\La^M_{e,\ol M}\to \ol M$ is ample (cf.\ Section VI in \cite{GGLR17} for the details).

 \section{Applications, further results  and some open questions}
\subsection{The  Satake-Baily-Borel completion of period mappings} \hfill 

We will discuss the proof of Theorem \ref{1a7} as stated in the introduction.  There are three steps in the argument:
\begin{itemize}
\item[(a)] construction of the completion $\ol M$ of $M$;
\item[(b)] analysis of the extended Chern form $\om_e$ on $\ol B$; and
\item[(c)] extension of the classical Kodaira theorem to the case of the singular variety $\ol M$ where the Chern form  up on $\ol B$ has mild singularities.
\end{itemize}
Step (b) was discussed in Section V above and using it step (c) can be done by extending  essentially standard arguments.  For this we refer to Section 6 in \cite{GGLR17} and  here shall  take up step (a). What follows is not a formal proof; the intent is to illustrate some of the key ideas behind the argument given in  \cite{GGLR17}.  There are four parts to the argument:
\begin{itemize}
\item[(a1)] localize the period mapping to
\[
\Phi:\Delta^{\ast k}\times \Delta^\ell \to\Ga_{\loc}\bsl D\]
where $\Ga_\loc= \{T_1,\dots,T_k\}$ is the local monodromy group and determine the structure of local image
\[
\Phi(\Delta^k\times \Delta^\ell)\subset \ol M;\]
\item[(a2)] using the first step and the proper mapping theorem, show that the fibres $\cF$ of the set-theoretic map
\[
\Phi_e :\ol B\to\ol M\]
are compact analytic subvarieties of $\ol B$, and from this infer that $\ol M$ has the structure of a compact Hausdorff topological space and $\Phi_e$ is a proper mapping;

\item[(a3)] define the sheaf $\cO_{\ol M}$ whose sections over an open set $\cU\subset \ol M$ are the continuous functions $f$ such that $f\circ \Phi_e$ is holomorphic in $\Phi^{-1}_e(\cU)$, and show that $\cO_{\ol M}$ endows $\ol M$ with the structure of an analytic variety; and
\item[(a4)] show that certain identifications of connected components of the fibres that result from the global action of monodromy give a finite equivalence relation (this uses \cite{CDK95}).

\end{itemize}
In this paper we shall mainly discuss the first step isolating the essential point of how the relative weight filtration property enters in the analysis of how the local period mapping in (a1) extends across the boundary strata of $\Delta^{\ast k}\times \Delta^\ell$ in $\Delta^k \times \Delta^\ell$.

Step (a2) is based on the following (cf. \cite{Som78}).  Let
\[
\Phi:B\to\Ga\bsl D\]
be a period mapping where $B=\ol B\bsl Z$ with $Z=\sum_i Z_i$ a normal crossing divisor with unipotent monodromy $T_i$ around $Z_i$.  The $\Phi$ may be extended across all the $Z^\ast_i:= Z_i\bsl \lrp{\cap_{j\ne i} Z_i\cap Z_j}$ around which $T_i = \Id$ and the resulting mapping
\[
\Phi_e:B_e\to \Ga\bsl D\]
is proper.\footnote{This result uses that an integral element of the IPR the holomorphic sectional curvatures are $\leqq -c$ for some $c>0$ and the Ahlfors-Schwarz lemma.  For the former we refer to Theorem \ref{6b1} below.  The latter is a by now standard deep fact in complex function theory.} Thus $\Phi_e(B_e)\subset \Ga\bsl D$ is a closed analytic variety.  This result will be used for $B$, and also for the period mappings
\[
\Phi_I:Z^\ast_I\to\Ga_I\bsl D_I\]
given by taking the associated gradeds to the \lmhs s along the strata $Z^\ast_I:=Z_I\bsl \lrp{\cap_{J\supset I}Z_J\cap Z_I}$. 

Step (a3), which is work still in progress, is based on analysis of the global structure of the fibres $\cF$, specifically positivity properties of the co-normal sheaf $I_\cF/ I^2_\cF$ (which is the co-normal bundle $N^\ast_{\cF/\ol B}$ in case $\cF$ is smooth).

Turning to the discussion of step (a1) we shall only consider the key special case when $\ell=0$ (there are no parameters) and $\Phi$ is a nilpotent orbit \pref{5b3}.  Then $\Phi$ is given by taking the orbit of the image of the homomorphism of  complex Lie groups
\[
\rho:\C^{\ast k}\to G_\C\]
whose differential is
\[
\rho_\ast(t_i\part/\part t_i)=N_i.\]

Referring to the discussion following the statement of Theorem \ref{1a7} in the introduction and recalled below we will prove the following
\begin{Prop}\lab{6a1} There exists a mapping
\[
\mu:\Delta^{\ast k}\to\C^N\]
whose   fibres are exactly the fibres of the set-theoretic mapping $\Phi_e$ localized to $\Delta^{\ast k}$ and with monodromy group $\Ga_{\loc}$ (cf.\ {\rm VI.A.3} below).\end{Prop}

\noindent The mapping $\mu$ in the proposition will be given by monomials, and for this reason it will be called a \emph{monomial mapping}.

The mapping $\Phi_e$ localized to $\Delta^{\ast k}$ arises from a period mapping
\begin{equation}\lab{6a2}
\Phi:\Delta^{\ast k}\to \Ga_{\loc}\bsl D\end{equation}
given by a nilpotent orbit.  The corresponding variation of \hs\ over $\Delta^{\ast k}$ induces variations of polarized \lmhs s along the open boundary strata $\Delta^\ast_I=\{(t_1,\dots,t_k)\in \Delta^k:t_i=0$ for $i\in I$ and $t_j\ne 0$ for $j\in I^c\}$.  Passing to the primitive parts of the associated graded gives period mappings
\[
\Phi_I:\Delta^\ast_I\to\Ga_I\bsl D_I,\]
 where $\Ga_I$ is generated by the $T_i$ for $i\in I$.
The set-theoretic mapping $\Phi_e$ is given by 
\begin{equation}\lab{6a3}
\Phi_e:\Delta^k \to \Ga_{\loc}\bsl D\amalg \lrp{\coprod_I \Ga_{I,\loc}\bsl D_I}\end{equation}
where $\Phi_e\big|_{\Delta^\ast_I}$ and $ \Phi_{I,\loc}$ is the monodromy group given by the VHS over $\Delta^\ast_I$.  

The proof of Proposition \ref{6a1} will be given in several steps, as follows.

{\setbox0\hbox{(1)}\leftmargini=\wd0 \advance\leftmargini\labelsep
 \begin{itemize}
\item[1.] determine the connected components of the nilpotent orbit
\begin{equation}\lab{6a4}
\exp\lrp{ \sum_j \ell(t_j)N_j} \cdot F: \Delta^{\ast k}\to \Ga_{\loc}\bsl D;\end{equation}
\item[2.] extend that argument to determining the connected components of the mappings $\Phi_I$, which will also be given by nilpotent orbits
\begin{equation}\lab{6a5}
\Phi_I(t_I)=\exp\lrp{\sum_{j\in I^c} \ell(t_j)N_j} \cdot F_I:\Delta^\ast_I \to  \Ga_I\bsl D_I;\end{equation}
\item[3.] show that for $I\subset J$ the closure in $\ol\Delta^\ast_I\cap \Delta^\ast_J$ of a fibre of $\Phi_I$ is equal to a connected component of a fibre of $\Phi_J$; and
\item[4.] show that the limit of a sequence of fibres of $\Phi_I$ is contained in a fibre of $\Phi_J$.
\end{itemize}} 

As for removing the ``connected component" qualifiers go, part of this deals with the identifications induced by  global monodromy that was mentioned above.

The basic idea in the construction appears already in step 1; the most interesting part of the argument  is step 3 where the relative weight filtration property provides the key.

\ssn{Step 1:} We consider the question: \emph{What are the conditions that a monomial
\begin{equation} \lab{6a6}
t^B= t^{b_1}_1\cdots t^{b_k}_k,\qquad b_i\in \Z\end{equation}
be constant on the fibres of \pref{6a2}?}  For this we let
\[
R=\lrc{A=(a_1,\dots,a_k):\sum_i a_iN_i=0}\subset \R^k\]
be the set of relations on the $N_i\in \cG$.\footnote{Here we use the \cite{GGLR17} notations $\cG=\End_Q(V)$  and  $\cG_\R=\End_Q(V_\R)$.}  We note that $R$ is defined over $\Q$.
\begin{Prop}\lab{6a7}
The conditions that the monomial \pref{6a6} be constant on the fibres of \pref{6a2} are
\[
A\cdot B=\sum_j a_jb_j=0,\qquad A\in R.\]
\end{Prop}
\begin{proof}
The vector field induced by $N=\sum a_j N_j\in\cG_\R$ is nowhere vanishing on $D$.\footnote{This is because $D\cong G_\R/H$ where $H$ is a compact subgroup of $G_\R$. The Lie algebra $\hh$ then contains no non-zero nilpotent elements in $\End_Q(V_\R$).}  Thus on the one hand
\[
\rho_\ast \lrp{\sum_j a_j t_j\part/\part t_j}\hensp{is tangent to a fiber of $\Phi$}\iff \sum_j a_jN_j=0.\]
On the other hand, the condition that the monomial \pref{6a6} be constant on the orbits of the vector field $\sum_j a_j t_j\part/\part t_j$ on $\Delta^{\ast k}$ is
\[
\lrp{\sum_j a_j t_j \part/\part t_j} t^B = (A\cdot B) t^B=0.\qedhere\]
\end{proof}

This simple computation contains one of the key ideas in the construction of the monomial mapping 
\[
\mu:\Delta^k \to \C^N.\]

We next consider the question:
{\setbox0\hbox{()}\leftmargini=\wd0 \advance\leftmargini\labelsep
 \begin{quote}
\em Are there enough monomials \pref{6a6} satisfying $A\cdot B=0$ for all $A\in R$ and where $b_j\in\Z^{\geqq 0}$ to separate the connected components of the fibres of \pref{6a4}?\end{quote}}  \noindent 
This is the existence result that is needed to give local charts that will define the fibres of \pref{6a4} up to connected components.  It is a consequence of the following

\begin{Prop}\lab{6a8}
The subspace $R^\bot$ is spanned by vectors $B$ where all $b_i\in\Q^{\geqq 0}$.\end{Prop}
This is  a consequence of a result in linear programming, known as \emph{Farkas' alternative theorem}.  We refer to Section 3 in \cite{GGLR17} for details and a reference.

\ssn{Step 2:} We consider the period mappings \pref{6a5}
\[
\Phi_I:Z^\ast_I\to\Ga_{\loc,I}\bsl D_I\]
given by the variation of  polarized \lmhs s on the open smooth strata $Z^\ast_I$.  Again restricting to the case of a nilpotent orbit  \pref{6a5}
 we may ask for the analogue of the question
in Step 1 for this nilpotent orbit.

The key observations here are that here \emph{both the weight and the Hodge filtrations enter}, and \emph{since $\Phi_I$ maps to the associated graded relative to the weight filtration $W(N_I)$ any operation that decreases $W(N_I)$ has no effect.}  Recalling from Section V.B our notations
\begin{itemize}
\item $N_I= \sum_{i\in I}N_i$;
\item $Y_I=$ grading element for $N_I$ and $\{N_I,Y_I,N^+_I\}$ is the resulting $\rsl_2$;
\item for $j\in I^c$ we have $N_j=N_{j,0}+N_{j,-1}+\cdots$ where $N_{j,-m}$ is the $-m$ weight space for $Y_I$;\end{itemize}
it follows that  the nilpotent orbit \pref{6a5} is the same as the nilpotent orbit using $N_{j,0}$ in place of $N_j$, and in place of  
Proposition \ref{6a7}
 we have 

\begin{Prop}\lab{6a9} The connected components of the fibres of \pref{6a5} are the level sets of monomials $t^B$ where $b_j\in \Z^{\geqq 0}$ and
\[
\sum_{j\in I^c} b_j N_j\in W_{-1}(N_I)\cG.\]
Moreover, recalling that the Chern form of the \hlb\ is given by $\om_I=\om_e\big|_{Z^\ast_I}$ these connected components are exactly the connected integral varieties of the exterior differential system 
\[
\om_I=0.\]
\end{Prop}

We recall our notation $\Delta^\ast_I = \{t\in \Delta^{\ast k}:t_i=0$ for $i\in I$ and $t_j\ne 0$ for $j\in I^c\}$ and denote by
\begin{equation}\lab{6a10}
\mu_I:\Delta^\ast_I\to \C^{N_I}\end{equation}
the monomial map constructed in the same way as  the monomial map constructed from Propositions \ref{6a7} and \ref{6a8} using the $t^B$ for a generating set of vectors $B\in R^\bot$ with the $b_i\in \Z^{\geqq 0}$.
The  to be constructed compact analytic variety  is $\ol M$ is set theoretically the disjoint union
\begin{equation}\lab{6a11}
\ol M=M\amalg \lrp{\coprod_I M_I}\end{equation}
where $M_I$ is a finite quotient of   the union of the  images
\[
\Phi_I(Z^\ast_I)\subset \Ga_{\loc,I}\bsl D_I.\]
To complete the proof of the construction of $\ol M$ as compact analytic variety two issues need to be addressed:
\begin{enumerate}[{\rm (i)}]
\item set-theoretically, the inverse image of $\Phi_I(\Delta^\ast_I)\subset M_I$ is a finite cover of $\mu_I(\Delta^\ast_I)$ and we need to describe analytic functions that will separate the sheets of this covering,\footnote{Here there is both a local issue dealing with the fibres of $\Phi_I:\Delta^\ast_I\to\Ga_I\bsl D_I$, and a global issue arising  from the possibility that two connected components of the fibre of $\Phi_I:\Delta^\ast_I \to\Ga_I\bsl D_I$ may be subsets of a single fibre of $\Phi_I$ on $Z^\ast_I$ due to the global action of monodromy.}
 and
\item we need to show that the analytic varieties $M_I$ fit together to give the structure of an analytic variety on $\ol M$.
\end{enumerate}
\ssn{Step 3:} For the second  of the two issues above we observe that the restriction
\[
t^B \big|_{\Delta^\ast_I}=0\hensp{if} b_i>0\hensp{for some} i\in I.\]
To establish (ii) we have  the
\begin{Prop}\lab{6a12}
For $I \subsetneqq J$ so that $\Delta^\ast_J\subset \ol{\Delta}^\ast_I$, the closure of a level set of $\mu_I$ is contained in a level set of $\mu_J$.  Moreover the limit in $\Delta^\ast_J$ of level sets in $\Delta^\ast_I$ is contained in a level set of $\mu _J$.\end{Prop}
\begin{proof} This will be a consequence of the relative weight filtration property (RWFP) \pref{5b7} that we now recall in a form adapted to the proof of \pref{6a12}. 

Given $A,B\subset \{1,\dots,k\}$ with $A\cap B=\emptyset$, we denote by
\[
N_B = N_{B,0}+N_{B,-1} +N_{B,-2}+\cdots\] 
the $Y_A$-eigenspace decomposition of $N_B$ relative to the $\rsl_2\;\{N^+_A, Y_A,N_A\}$.  Then the nilpotent operator
\[
N_{B,0}\big|_{\Gr^{W(N_A)}_m V}: \Gr^{W(N_A)}_m V\to \Gr^{W(N_A)}_m V\]
induces a weight filtration on $\Gr^{W(N_A)}_\bullet V$.  Another weight filtration on this vector space is defined by 
\[
\frac{\wbul(N_A+N_B)\cap W_m(N_A)}{\wbul(N_A+N_B)\cap W_{m-1}(N_A)}.\]
The RWFP is that these two filtrations coincide; i.e.,
\begin{equation}\lab{6a13}
\frac{W_{m+m'} (N_A+N_B)\cap W_m(N_A)}{W_{m+m'}(N_A+N_B)\cap W_{m-1}(N_A)} = W_{m'} \lrp{ N_{B,0} \big|_{\Gr^{W(N_A)}_m V}}.\end{equation}

Returning to the proof of \pref{6a12}, what must be proved is that for $I \subsetneqq J$
\begin{equation}\lab{6a14}
\sum_{j\in I^c} a_j N_j \in W_{-1} (N_I)\implies \sum_{j\in J^c} a_j N_j\in W_{-1}(N_J).\end{equation}
On the RHS we have used that $N_j \in W_{-2}(N_J)$ for $j\in J$, so the sum is really over $j\in J^c$.
What \pref{6a14} translates into is that if $t^A$ is a monomial that is constant on the fibres of $\Phi_I$, then the restriction $t^A\big|_{\Delta^\ast_J}$ is constant on the level sets of $\Phi_J$.

We let $X=\sum_{j\in I^c} a_j N_j\in W_{-1} (N_I)\cG$, where the ``$\in$'' is because
\[
X\in \cZ(N_I)\implies X\hensp{has only negative weights in the $N_I$-string decomposition of}\cG.\]
Write
\[
X=X_0+X_{-1}+X_{-2}+\cdots\]
in terms of the eigenspace decomposition of $Y_J$.  Since $N_{J-I}$ is in the $-2$ eigenspace of $Y_I$, we have
\[
[N_{J-I},X_0] = 0.\]
Now decompose $X_0$ into $Y_I$ eigenspace components
\[
X=X_{0,-m}+ X_{0,-(m+1)}+\cdots ,\qquad m\geqq 1.\]
Then
\[
[N_{J-I,0},X_{0,-m}] = 0\implies X_{0,-m}\in \Gr^{W(N_I)}_{-m}\cG\hensp{lies in} \cZ \lrp{N_{J-I,0} \big|_{\Gr^{W(N_I)}_{-m}}\cG}\]
and consequently 
\[
X_{0,-m}\in W_0\lrp{N_{J-I,0}\big|_{\Gr^{W(N_I)}_{-m}\cG}}.\]
Applying   \pref{6a13} with $A=I$, $B=J-I$ gives \pref{6a14}, which proves \ref{6a12}.
\end{proof}

\ssn{Step 4:} At this point we have constructed a monomial mapping
\[
\mu:\C^k \to \C^N\]
whose fibres are unions of the fibres of the nilpotent orbit \pref{5a3}.  The final local step is to refine the construction to have
\begin{equation}\lab{new6a15} \begin{split}
\xymatrix{\C^k\ar[rr]^\eta\ar[dr]_\mu&&\C^k\ar[dl]^{\tilde \mu}&\\
&\C^N&} \end{split}
\end{equation}
where $\tilde \mu$ is a monomial mapping with connected fibres.  The basic idea already occurs when $k=N=1$ and $\mu(t)=t^m$; then ``$t^{1/m}$'' separates the fibres of $\mu$.

In general we suppose that
\[\mu(t)=(t^{I_1},\dots,t^{I_N})\]
where 
\[
t^{I_j}=t^{i_{j_1}}_1\cdots t^{i_{j_k}}_k.\]
Define a map $\Z^k\to\Z^N$ by 
\[
e_j \to(i_{1_j},\dots,i_{N_j}).\]
Then identifying $\Z^N$ with $\Hom(\Z^N,\Z)$, up to a finite group  the image $\La\subset \Z^N$ is defined by $\La^\bot\subset \Z^N$.  Setting $\wt \La=(\La^\bot)^\bot$, for the finite abelian group $\wt \La/\La$ we have
\begin{equation}\lab{new6a16}
\wt\La/\La\cong \oplus \Z/d_i\Z.\end{equation}
The mapping $\eta$ in \pref{new6a15} will be a $|\wt \La/\La|$-to-1 monomial map.  To construct it, at the level of exponents of monomials we will have
\begin{equation}\lab{new6a17} \begin{split}
\xymatrix{\Z^k \ar[rr]^B\ar[dr]_A&&\Z^k\ar[dl]^{\wt A}\\
&\Z^N&}
\end{split}\end{equation}
where $\rim(A)=\La$, $\rim(\wt A)=\wt \La$ and $\Z^k/\rim(B)\cong \wt \La/\La$.  Such a map always exists and may be constructed using \pref{new6a16}.  We then use \pref{new6a17} to define \pref{new6a15} where $\tilde\mu$ is a monomial map with connected fibres. \hfill\qed\medbreak

For the issue arising from the global action of monodromy we refer to Section 3 in \cite{GGLR17}.

\begin{Exam}\lab{n6a18} This is a continuation of Example \ref{5c15} above.  Since $N_1,N_2,N_3$ are linearly independent, the mapping $\Delta^3\to\Ga_\loc\bsl D$ given by the corresponding nilpotent orbit is 1-1.\footnote{In this case $D=\cH_2$.}  

The interesting situation is on the face $\{0\}\times \Delta^{\ast 2}$ given by $t_1=0$.  There we observe from the computations in Example \ref{5c15} that the induced maps
\[\ol N_2, \ol N_3:\Gr^{W(N_1)}_1 V\to \Gr^{W(N_1)}_1 V\]
are equal.  This gives the relation $\ol N_2-\ol N_3=0$ which leads to the monomial $t_2t_3$ that is constant on the fibres of the period mapping
\[
\Phi_1:\{0\}\times \Delta^{\ast 2}\to \Ga_{1,\loc}\bsl D_1.\footnotemark\]
\footnotetext{In this case $D_1=\cH$.}%
On the codimension 2 strata the associated graded to the LMHS's are Hodge-Tate, so the corresponding period mappings are constant.\footnote{We note that the limit of the parabolas $t_2t_3=c$ as $c\to 0$ will be the coordinate axes, confirming the limit part of the statement in \ref{6a12} in this case.}
\end{Exam}

We conclude this section by discussing the following    question:
\[
\hbox{\em What are the Zariski tangent spaces to $\ol M$?}\]
More precisely,
\begin{equation}\lab{new6a18}
\hbox{\em What is the kernel of the mapping $T_b \ol B\to T_{\Phi_e(b)} \ol M$?}\end{equation}

Recalling the notation $\om_e$ for the Chern form of the canonically extended \hlb\ $\La_e\to\ol B$, we want to define in each tangent space to $\ol B$ the meaning of the equations
\begin{equation}\lab{6a16}
\om_e(\xi)=0,\qquad \xi\in T_b\ol B.\end{equation}
The issue is that $\om_e$ is not smooth,  continuous, or even bounded.  If $b\in B$, then \pref{6a16} has the usual meaning $\Phi_\ast(\xi)=0$.  If $b\in Z^\ast_I$ and $\xi\in T_{b}Z^\ast_I$ is tangent to $Z^\ast_I$, then since $\om_e\big|_{Z^\ast_I} = \om_I$ \pref{6a16} means that $\Phi_{I,\ast}(\xi)=0$.

Thus the interesting case is when $b\in Z^\ast_I$ and $\xi$ is a normal vector to $Z^\ast_I$  in $\ol B$.  This amounts to the situation of a 1-parameter VHS
\[
\Phi:\Delta^\ast \to 
\Ga_T\bsl D,\qquad \Ga_T=T^\Z \]
and \pref{6a16} becomes the condition
\begin{equation}\lab{6a17}
\Phi_{e,\ast}(\part/\part t)\big|_{t=0}=0.\end{equation}
In this case we have the
\begin{Prop}\lab{6a18}
If $T\ne \Id$, then $\Phi_{e,\ast}(\part/\part t)\big|_{t=0}\ne 0$.\end{Prop}

\begin{proof}
If $T=\exp N$ is unipotent, then the methods used in Section V.A above and Section 3 in \cite{GGLR17} give for the Chern form on $\Delta^\ast$
\[
\om_e \geqq C \frac{dt\wedge d\bar t}{|t|^2 (-\log|t|)^2},\qquad C>0.\]
In general $T$ is quasi-unipotent and after a base change $t'=t^m$ we will have
\[\xymatrix@R=1.2pc@C=1.2pc{
\Delta^{'\ast} \ar[d]_\pi \ar[r]^(.4){\Phi'}& \Ga_{T'}\bsl D\ar[d]\\
 \Delta^\ast\ar[r]^(.4)\Phi&\Ga_T\bsl D}\]
where $T'$ is unipotent.
Then from
\[
\pi^\ast \lrp{\frac{dt\wedge d\bar t}{|t|^2 (-\log|t|)^2}} =C' \frac{dt'\wedge d\bar t'}{|t'|^2(-\log|t'|)^2},\quad C'>0\]
we may infer the proposition.\footnote{It is interesting to note that the pullback of smooth forms under a branched covering map vanish  along the branch locus.  For the Poincar\'e metric this is not the case, illustrating again the general principle that in Hodge theory singularities increase positivity.}
\end{proof}

From \ref{6a18} we may draw the
\begin{Concls} \lab{6a19}
{\rm (i)} If the $N_j$, $j\in I^c$, are linearly independent modulo $W_{-1}(N_1)$, then $\om_e >0$ in the normal spaces to $Z^\ast_I$.

{\rm (ii)} For nilpotent orbits, $\om_e(\xi)=0\iff \mu_\ast(\xi)=0$ where $\mu:\Delta^{\ast k}\to \C^N$ is the monomial map.
\end{Concls}

We comment that  (ii) holds in the general situation in Section 3 of \cite{GGLR17} where local quasi-charts are constructed for arbitrary VHS over $\Delta^{\ast k}\times \Delta^\ell$.

 Finally we   use the results of Section V above and Section 3 in \cite{GGLR17} to summarize the  properties of the \hbox{EDS \pref{6a16}:}
\begin{equation}\lab{6a20}
\bmp{5}{\em {\setbox0\hbox{(1)}\leftmargini=\wd0 \advance\leftmargini\labelsep
 \begin{itemize}
\item[(a)] \pref{6a16} defines a coherent, integrable sub-sheaf $\cI\subset \cO_{\ol B}(T\ol B)$;
\item[(b)] the maximal leaves of $\cI$ are \emph{closed}, complex analytic subvarieties of~$\ol B$;
\item[(c)] as a set, $\ol M$ is the quotient of the fibration of $\ol B$ given by the leaves of $\cI$.\end{itemize}}    }
\end{equation}
Singular integrable foliations and their quotients have been introduced and studied in \cite{Dem12}.

\subsection{Norm positivity and the cotangent bundle to the image of a period mapping}
\hfill

(i) \emph{Statement of results.}  Let $\Phi:B\to\Ga\bsl D$ be a period mapping with image a quasi-projective variety $M\subset  \Ga\bsl D$.  The $G_\R$-invariant metric on $D$ constructed from the Cartan-Killing form on $\cG_\R$ induces a K\"ahler metric on the Zariski open set $M^o$ of smooth points of $M$.  We denote by $R(\eta,\xi)$ and $R(\xi)$ the holomorphic bi-sectional and holomorphic sectional curvatures respectively.

\begin{Thm} \lab{6b1}
There exists a constant $c>0$ such that \begin{enumerate}[{\rm (i)}]
\item $R(\xi)\leqq -c$ for all $\xi\in TM^o$;
\item $R(\eta,\xi)\leqq 0$ for all $\eta,\xi\in TM^o \times_{M^o}TM^o$;
\item For any $b\in M_o$ there exists a $\xi\in T_b M^o$ such that $R(\eta,\xi)\leqq -c/2$ for all $\eta\in T_b M^o$.
\end{enumerate}\end{Thm}

Observe that using \pref{2f3} from \cite{BKT14}   (iii) follows from (i) and (ii).   As a corollary to (ii) we have 
\begin{enumerate}[(iv)]
 \item $R(\eta,\xi)\leqq -c/2$ is an open set in $T M^o \times_{M^o} TM^o$.\end{enumerate}
We note that (iii) implies that this open set projects onto each factor in $M^o\times M^o$.

As applications of the proof of Theorem \ref{6b1}  
 and consideration of the singularity issues that arise we have the following   extensions of some of the results of Zuo \cite{Zuo} and others (cf.\ Chapter 13 in \cite{CM-SP}):
\begin{align}\lab{6b2}
&M\hensp{\em is of log-general type,}\\
\lab{6b3}
& \Sym^m\Om^1_M(\log) \hensp{\em is big for} m\geqq m_0.\end{align}
The result in   \pref{6b2}  means that for any desingularization $\wt{\ol M}$ of $\ol M$ with $\wt M$ lying over $M$ and $\wt Z=\wt{\ol M}\bsl \wt M$, the Kodaira dimension 
\[
\kappa\lrp{K_{\wt{\ol M}} (\wt Z)} = \dim M.\]
The result in \pref{6b3} means that 
\[
\Sym^m \Om^1_{\oltm}(\log \wt Z)\hensp{\em is big for} m\geqq m_0.\]
The proof will show that we may choose $m_0$ to depend only on the \hn s for the original VHS.

The proof will also show that
\[
M\hensp{\em is of stratified-log-general type,} \leqno({\rm VI.B.2})_{\rm S}\]
\[
\Sym^m \Om^1_{\ol M}(\log) \hensp{\em is stratified-big for} m\geqq m_0.\leqno{({\rm VI.B.3})}_{\rm S}\]
Here stratified-log-general type means that there is a canonical stratification $\{M^\ast_I\}$ of $\ol M$ such that each stratum $M^\ast_I$ is of log-general type.  There is the analogous definition for stratified big.

Without loss of generality, using the notations   above we may take $\oltm=\ol B$, $\wt M=B$ and $\wt Z=Z$; we shall assume this to be the case.

\begin{rem}
The results of Zuo, Brunebarbe and others are essentially that $K_{\ol M}(\log)$ and $\Om^1_M(\log)$ are  weakly positive in the sense of Viehweg.   This is implied by \pref{6b3}.

The proof of Theorem \ref{6b1} will be done first in the case
\begin{equation}\lab{6b4}
B=\ol B\hensp{\em and $\Phi_\ast$ is everywhere injective.}\end{equation}
It is here that the  main ideas and calculations occur.  

The singularities that arise are of the types 
\begin{equation} \lab{6b5}
\begin{cases}
{\rm (a)}& \hbox{ where $\Phi_\ast$ fails to be injective (e.g., on $M_{\sing}$),}\\
{\rm (b)}& \hbox{ on $Z=\ol B\bsl B$ where the VHS has singularities,}\\
{\rm (c)}&\hbox{ the combination of (a) and (b).}\end{cases}\end{equation}
As will be seen below, there will be a coherent sheaf I with
\[
\Phi_\ast(TB)\subset I\subset \Phi^\ast(T(\Ga\bsl D)).\footnotemark\]
\footnotetext{Here we are identifying a coherent sub-sheaf of a vector bundle with the corresponding family of linear subspaces in the fibres of the vector bundle.  The coherent sheaf $I$ will be a subsheaf of the pull-back $\Phi^\ast T(D\bsl \Ga)_h$ of the horizontal tangent spaces to $\Ga\bsl D$.  The critical step in the calculation will be that it is integrable as a subsheaf $\Phi^\ast T(\Ga\bsl D)_h$.}%
Denoting by $I^o$ the open set where $I$ is locally free, there is an induced metric and corresponding curvature form for $I^o$, and with the properties (i), (ii) in the theorem for $I^o$  Theorem \ref{6b1} will follow from the curvature decreasing property of holomorphic sub-bundles, which gives 
\[
R(\eta,\xi)=\Theta_{TM^o} (\eta,\xi)\leqq \Theta_{I^o}(\eta,\xi).\]
  \end{rem}

As for the singularities, if we show that
\begin{align}
\lab{6b6}
\kappa\lrp{\det I^o(\log)}&= \dim B\\
\lab{6b7}
\Sym^m I^o(\log)&\hensp{\em is big}\end{align}
then \pref{6b2} and \pref{6b3} will follow from the general result: If over a projective variety $Y$ we have \lb s $L,L'$ and a morphism $L\to L'$ that is an inclusion over an open set, then
\begin{equation}\lab{6b8}
L\to Y\hensp{\em big}\implies L'\to Y\hensp{\em is big.}
\end{equation}
We will explain  how \pref{6b6} and \pref{6b7} will follow from \pref{6b8} for suitable choices of $Y, L$ and $L'$.

\medbreak (ii) \emph{Basic calculation.}  It is conventient to use Simpson's  system of Higgs bundles framework (cf.\  \cite{Sim92} and Chapter 13 in \cite{CM-SP}) whereby a VHS is given by a system of holomorphic \vb s $E^p$, and maps
\[
E^{p+1} \xri{\theta^{p+1}} E^p\otimes \Om^1_B\xri{\theta^p} E^{p-1} \otimes \wedge^2 \Om^2_X\]
that satisfy
\begin{equation}\lab{6b9}
\theta^p \wedge\theta^{p+1}=0.\end{equation}
Thus there is induced
\[
E^{p+1}\xri{\theta^{p+1}} E^p\otimes \Om^1_B\xri{\theta^p} E^{p-1} \otimes \Sym^2 \Om^1_B,\]
and the data $(\opplus_p E^p \otimes \Sym^{k-p} \Om^1_B,\opplus_p \theta^p)$ for any $k$  with $k\geqq p$ is related to the notion of an infinitesimal variation of \hs\ (IVHS) (cf.\ 5.5 ff.\ in \cite{CM-SP}).

In our situation the \vb s $E^p$ will have Hermitian metrics with Chern connections $D^p$.  The metrics define adjoints
\[
\theta^{p^\ast}:E^p \to E^{p+1} \otimes\ol{\Om^1_B},\]
and in the cases we shall consider if we take the direct sum over $p$ we obtain
\[
(E,\nabla=\theta^\ast +D+\theta),\hensp{ \em with \pref{6b9} equivalent to} \nabla^2=0.\]
The properties uniquely characterizing the Chern connection together  with $\nabla^2=0$ give for the curvature matrix of $E^p$ the expression
\begin{equation}\lab{6b10}
\Theta_{E^p} = \theta^{p+1} \wedge \theta^{p+1^\ast} +\theta^{p^\ast}\wedge \theta^p,\end{equation}
which is a difference of non-negative terms each of which has the norm positivity property \pref{3a3} (cf.\ \cite{Zuo} and  Chapter 13 in \cite{CM-SP}).

For a PVHS $(V,Q,\nabla,F)$ we now  set
\[
E^p = \Gr^p \Hom_Q(V,V),\qquad -n\leqq p\leqq n\]
where $\Gr^p$ is relative to the filtration induced by $F$ on $\Hom_Q(V,V)$.  At each point $b$  of $B$ there is a weight zero PHS induced on $\Hom_Q(V,V)=\cG$ and
\[
E^p_b = \cG^{p,-p} \]
with the   bracket
\[
[\enspace,\enspace]: E^p\otimes E^q\to E^{p+q}.\]
Thinking of $\theta$ as an element in $\cG\otimes \Om^1_B$, the integrability condition \ref{6b9} translates into
\begin{equation}\lab{6b11}
 [\theta,\theta]=0.\end{equation}
   We shall use the notation
 \[
 \Gr^p = \Gr^p\Hom_Q(V,V)\]
 rather than $E^p$ for this example.
 
 The differential of $\Phi$ gives a map
 \[
 \Phi_\ast:TB\to \Gr^{-1}.\]

\begin{defin}
$I\subset \Gr^{-1}$ is the coherent subsheaf generated by the sections of $\Gr^{-1}$ that are locally in the image of $\Phi_\ast$ over the Zariski open set where $\Phi_\ast$ is injective.  
\end{defin}

For $\xi$ a section of $I$ we denote by $\ad_\xi$ the corresponding section of $\Gr^{-1}$.  The integrability condition \pref{6b11} then translates into the first part of the
\begin{Prop}\lab{6b12}
$I$ is a sheaf of abelian Lie sub-algebras of $\opplus_p \Gr^p$.  For $\eta,\xi$ sections of $I$
\[
\Theta_{\Gr^{-1}} (\eta,\xi) = -\| \ad^\ast_\xi(\eta)\|^2.\]
\end{Prop}

\begin{proof}  For $\eta,\xi\in \Gr^{-1}$ the curvature formula \pref{6b10} is
\[
\Theta_{\Gr^{-1}}(\eta,\xi)=\|\ad_\xi(\eta)\|^2 -\| \ad^\ast_\xi(\eta)\|^2.\]
The result then follows from $\ad_\xi(\eta)=[\xi,\eta]=0$ for $\eta,\xi\in I$.\end{proof}

On the open set where $I^o$ is a \vb\ with metric induced from that on $\Gr^{-1}$ we have 
\[
\Theta_{I^o} (\eta,\xi)\leqq \Theta_{\Gr^{-1}}(\eta,\xi)\leqq 0.\]
The first term is the holomorphic bi-sectional curvature for the indued metric on $\Phi(B)$.

To complete the proof of Theorem \ref{6b10} we need to show the existence of $c>0$ such that for all $\xi$ of unit length
\begin{equation}
\lab{6b13}
\|\ad^\ast_\xi(\xi)\|\geqq c.\end{equation}
The linear algebra situation is this: At a point of $B$ we have
\[
V=\opplus_{p+q=n}V^{p,q}\]
and $\xi$ is given by maps
\[
A_p: V^{p,q}\to V^{p-1,q+1},\qquad \Big\lfloor\frac{n+1}{2}\Big\rfloor\leqq p\leqq n.\]

In general a linear map
\[
A:E\to F\]
between unitary vector spaces has  {\em principal values} $\la_i$ defined by 
\[
A e_i = \la_i f_i,\qquad \la_i\hbox{ real and non-zero}\]
where $e_i$ is a unitary basis for $(\ker A)^\bot$ and $f_i$ is a unitary basis for $\rim A$.
 The square norm is 
 \[
 \|A\|^2 = \Tr A^\ast A=\sum_i \la^2_i.\]
 
 We denote by $\la_{p,i}$ the principal values of $A_p$.  The $\la_{p,i}$ depend on $\xi$, and the square norm of $\xi$ as a vector in $T_p B\subset T_{\Phi(p)}(\Ga\bsl D)$ is
 \[
 \| \xi\|^2 = \sum_p \sum_i \la^2_{p,i}.\]

In the above  we now replace $V$ by $\Hom_Q(V,V)$ and use linear algebra to determine the principal values of $\ad^\ast_\xi$.  These will be quadratic in the $\la_{p,i}$'s, and then
 \[\| \ad_{\xi^\ast}(\xi)\|^2\]
 will be quartic in the $\la_{p,i}$.  A calculation   gives
 \begin{equation}\lab{6b14}
 \|\ad_{\xi^\ast}(\xi)\|^2 = \sum_p \lrp{
 \frac{\sum_i a_p \la^4_{p,i}}{\lrp{\sum_i \la^2_{p,i}}^2}}\end{equation}
 where the $a_p$ are non-negative integers that are positive if $A_p\ne 0$, and from this by an elementary algebra argument we may infer the existence of the $c>0$ in Theorem \ref{6b1}. \hfill\qed\medbreak

At this point we have proved the theorem.  The basic idea is very simple:

{\setbox0\hbox{(1)}\leftmargini=\wd0 \advance\leftmargini\labelsep
 \begin{quote}
\em For a VHS the curvature \pref{6b10} of the \hb s is a difference of non-negative terms, each of which is of norm positivity type where the ``$A$''  in Definition \ref{3a3} is a Kodaira-Spencer map or its adjoint.  For the $\Hom_Q(V,V)$ variation of \hs,   $A(\xi)(\eta)=[\xi,\eta]=0$ by integrability.  Consequently the curvature form has a sign, and a linear algebra calculation gives the strict negativity $\Theta_I(\xi,\xi)\leqq -c\|\xi\|^4$ for some $c>0$.\footnote{This first proof of the result that appeared in the literature was Lie-theoretic where the metric on $\cG$ was given by the Cartan-Killing form.  As will be illustrated below the above direct algebra argument is perhaps more amenable to the computation in examples.}\end{quote}
}

(iii) \emph{Singularities.}  The singularity issues were identified in \pref{6b5}, and  we shall state a result that addresses them.  The proof of this result follows from the results in \cite{CKS86} as extended in \cite{GGLR17}, \cite{Kol87} and the arguments in \cite{Zuo}.\footnote{These arguments have been amplified at a number of places in the literature; cf.\ \cite{VZ03} and \cite{Pau16}.}

Using the notations introduced in (ii) above, a key observation is that the differential
\[
\Phi_\ast:TB\to \Gr^{-1}\]
extends to 
\[
\Phi_\ast:T\ol B\lra{-Z}\to \Gr^{-1}_e\]
where $T\ol B\lra{-Z} = \Om^1_{\ol B}(\log Z)^\ast$ and $\Gr^{-1}_e$ is the canonical extension to $\ol B$ of $\Gr^{-1}\to B$.  This is just a reformulation of the general result (cf.\ \cite{CM-SP}) that for all $p$, $\theta^p :E^p\to E^p\otimes \Om^1_{\ol B}$ extends to
\begin{equation}\lab{6b15} \theta^p_e :E^p_e \to E^{p-1}_e\otimes \Om^1_{\ol B}(\log Z).\end{equation}
 As noted above, the image $\Phi_\ast TB\subset \Gr^{-1}$ generates a coherent subsheaf $I\subset \Gr^{-1}$ and from \pref{6b15} we may infer that $I$ extends to a coherent subsheaf $I_e\subset \Gr^{-1}_e$.  As  in \cite{Zuo} we now blow up $\ol B$ to obtain a vector sub-bundle of the pullback of $\Gr^{-1}$ and note that  $I_e \subset \Gr^{-1}_e$ \emph{will be  an integrable sub-bundle.}

The metric on $\Gr^{-1}$ induces a metric in $I$ and we use the notations
{\setbox0\hbox{()}\leftmargini=\wd0 \advance\leftmargini\labelsep
 \begin{itemize} 
\item $\vp=$ Chern form of $\det I^{o\ast}$;
\item $\om=$ Chern form of $\cO_{\P I^{o\ast}}(1)$.\end{itemize}}  \noindent 

\begin{Thm}\lab{6b16}
Both $\vp$ and $\om$ extend to closed, $(1,1)$ currents $\vp_e$ and $\om_e$ on $\ol B$ and $\P I^\ast_e$  that respectively represent $c_1(\det I^\ast_e)$ and $c_1(\cO_{\P I^\ast_e})(1)$.  They have mild singularities and satisfy
{\setbox0\hbox{()}\leftmargini=\wd0 \advance\leftmargini\labelsep
 \begin{itemize} \item $\vp_e\geqq 0$ and $\vp_e>0$ on an open set;
\item $\om_e\geqq 0$ and $\om_e>0$ on an open set.\end{itemize}}   \end{Thm}
With one extra step this result follows from singularity considerations similar to those in Section IV above.  The extra step is that
{\setbox0\hbox{(1)}\leftmargini=\wd0 \advance\leftmargini\labelsep
 \begin{quote}
\em $I_e$ is not a \hb, but rather it is the kernel of the map $\theta^{-1}:\Gr^{-1}_e \to \Gr^{-2}_e\otimes \Om^1_e \otimes \Om^1_{\ol B}(\log Z)$.\end{quote}}  \noindent 
As was noted in \cite{Zuo}, either directly or using  (5.20) in \cite{Kol87} we may infer the stated properties of $\vp_e$ and $\om_e$.\hfill\qed

\begin{rem} It is almost certainly not the case that any sub-bundle $G\subset \Gr^{-1}_e$ will have Chern forms with mild singularities.  The bundle $I_e$ is special in that it is the kernel of the map $\Gr^{-1}_e\to \Gr^{-2}_e\otimes \Om^1_{\ol B}(\log Z)$.  Although we  have not computed the $2^{\rm nd}$ fundamental form of $I_e \subset \Gr^{-1}_e$, for reasons to be discussed below it is reasonable to expect it to also have good properties.\end{rem}

The issue of the curvature form of the induced metric on the image $M=\Phi(B)\subset \Ga\bsl D$ seems likely to be interesting.  Since the metric on the smooth points $M^o\subset M$ is the K\"ahler metric given by the Chern form of the augmented Hodge \lb, the curvature matrix of $TM^o$ is computed from a positive (1,1) form that is  itself the curvature of a singular metric.  In the 1-parameter case the dominant term in $\om$ is the Poincar\'e metric $PM=d t\otimes d\bar t / |t|^2 (-\log|t|)^2$, and the curvature of the $PM$ is a positive constant times $-PM$.  One may again suspect that the contributions of the lower order terms in $\om$ are less singular than $PM$.  This issue may well be relevant to Question \ref{1a10n}.

\medbreak (iv) \emph{Examples.}  On the smooth points  of $M^o$ of  the image of a period mapping the holomorphic bi-sectional curvature satisfies 
\begin{equation}\lab{6b17} R(\eta,\xi)\leqq 0,\end{equation}
and for $\eta,\xi$ in an open set  in $TM^o\times_{M^o} TM^o$ it is strictly negative.  This raises the interesting question of the degree of flatness of $T^\ast M^o$.  In the classical case when $D$ is a Hermitian symmetric domain and $B=\Ga\bsl D$ is compact this question has been studied by Mok  \cite{Mok87} and others.  In case $B$ is a Shimura variety the related question of the degree of flatness of the extended \hb\ $F_e$ over a toriodal compactification of $\Ga\bsl D$ is one of current interest (cf.\ \cite{Bru16b}, \cite{Bru16} and the references cited there).  This issue will be further discussed in Section VI.E.

Here we shall discuss the equation 
\[
\Theta_{I^0}(\eta,\xi) =0 \]
over the smooth locus $M^o$ of $M$. 
In view of \pref{6b10} this equation  is equivalent to
\[
\ad_{\xi^\ast}(\eta)=0,\qquad \eta\in I .\]
To compute  the dimension of the solution space  to this equation,
we use the duality
\[
\ker(\ad^\ast_\xi)=(\rim (\ad_\xi))^\bot\]
to have
\begin{equation}\lab{6b18}
\dim\ker(\ad^\ast_\xi)= \dim\lrp{\mathrm{coker} \lrp{\rim \{\ad_\xi:\Gr^0 \to \Gr^{-1}\}}}.\end{equation}

Since $I$ depends on the particular VHS, at least as a first step it is easier to study the equation
\begin{equation}\lab{6b19}
\Ad_{\xi^\ast}(\eta)= 0,\qquad \eta \in \Gr^{-1}.\end{equation}
Because the curvature form decreases on the sub-bundle $I\subset \Gr^{-1}$, over $M^o$  we have
\[
\hbox{\pref{6b18}$\implies$\pref{6b19}}\]
but   in general  not conversely.

\begin{example} \lab{exam1}
For weight $n=1$ with $h^{1,0}=g$, with a suitable choice of coordinates the tangent vector $\xi$ is given by  $g\times g$ symmetric matrix $A$, and on $\Gr^{-1}$ we have
\begin{equation}\lab{6b20}
\dim \ker(\ad^\ast_\xi) =  \bpm{g-\hbox{rank }A+1}\\{2 }\epm\end{equation}
\end{example}

\nproof At  a point   we may choose a basis for that $Q=\bspm 0& Ig\\ -Ig&0\espm $ and
\begin{align*}
F^1&\hbox{ is given by } \bpm \Om\\ I_g\epm,\quad \rim\Om>0\\
\xi \in \Gr^{-1}&\hbox{ is given by } \bpm 0&A\\ 0&0\epm,\qquad A= {}^t A\\
\eta\in\Gr^0 &\hbox{ is given by } \bpm C&0\\ 0&-{}^t C\epm.\end{align*}
Then
\[
[\xi,\eta]= \bpm 0&AC+ {}^tCA\\
0&0\epm.\]
Diagonalizing $A$ and using \pref{6b18} we obtain \pref{6b20}.

\begin{example}\lab{exam2}
For weight $n=2$, $\xi$ is given by
\[
A=h^{2,0} \times h^{1,1}\hbox{ matrix}.\]
We will show that on $\Gr^{-1}$
\begin{equation}\lab{6b21}
\dim\ker (\ad^\ast_\xi) = (h^{2,0}\hbox{-rank }A)(h^{1,1}\hbox{-rank } A).\end{equation}
\end{example}

\begin{proof}
We may choose bases so that $Q=\diag (I_{h^{2,0}},-I_{h^{1,1}}, I_{h^{2,0}})$ and
\begin{align*}
F^2&\hbox{ is given by } \bpm \Om \\ 0 \\ i\Om\epm, \quad \Om\hbox{ non-singular},
\\
\xi&\hbox{ is given by } \bpm 0&A&0\\
0&0&^t A\\
0&0&0\epm,\\
\eta&\hbox{ is given by } \bpm C&0&0\\
0&D&0\\
0&0&-{}^t C\epm.\end{align*}
Then
\[
[\xi,\eta] = \bpm 0&AC-DA&0\\
0&0&{}^t A D+ {}^t( AC)\\
0&0&0\epm.\]
Choosing bases so that $A=\bspm I&0\\ 0&0\espm$, $C=\bspm C_{11}&C_{12}\\ C_{21}&C_{22}\espm$ and $D=\bspm D_{11}&D_{12}\\ -{}^t D_{12}&D_{22}\espm$, we have
\[
AC-DA=\bpm C_{11}&-D_{11}&D_{12}\\
0&-{}^t D_{12}&0\epm.
\]
Setting $\rk(E)=\rank E$ for a matrix $E$, this gives 
\begin{align*} & \hspace*{.75in} \hbox{rk } A \quad h^{2,0}\hbox{-rk }A\\
& \begin{matrix}
\hfill \rk A\\
\hfill  h^{1,1}\hbox{-rk } A\end{matrix} \bpm 
 \enspace\ast &\hspace*{.45in}\ast\quad\;\\
 \enspace\ast &\hspace*{.45in}0\quad\;\epm   \end{align*}
where the $\ast$'s are arbitrary.\end{proof}

As in the $n=1$ case  we note that
\begin{equation}\lab{6b22}
A \hensp{of maximal rank} \iff \ker (\ad^\ast_\xi)=0 .\end{equation}

\begin{example}\lab{exam3}
Associated to a several parameter nilpotent orbit
\[
\exp \lrp{\sum_i \ell(t_i)N_i}\cdot F\]
is a nilpotent cone $\sig=\{N_\la=\sum \la_i N_i, \la_i>0\}$ and the weight filtration $W(N)$ is independent of $N\in \sig$.  As discussed in Section 2 of \cite{GGLR17}, without loss of generality in what follows here  we may assume that the LMHS associate to $N\in \sig$ is $\R$-split. Thus there is a single $Y\in \Gr^0 \Hom_Q(V,V)$ such that for any $N\in \sig$ 
\[
[Y,N] = -2N,\]
and using the Hard Leftschetz Property $N^k:\Gr^{W(N)}_{n+k}(V)\simto \Gr^{W(N)}_{n-k}(V)$ we may uniquely complete $Y,N$ to an $\rsl_2\, \{N,Y,N^+\}$.  Let $\cG_\sig \subset \End(\Gr^{W(N)}_\bullet V)$ be the Lie algebra generated by the $N_i$ and $Y$.  The properties of this important Lie algebra will be discussed elsewhere; here we only note that $\cG_\sig$ {is semi-simple} and that the nilpotent orbit gives a period mapping
\[
\Delta^{\ast k}\xri{\Phi_\sig} \Ga_{\loc}\bsl D_\sig\]
where $D_\sig = G_{\sig,\R}/H_\sig$ is a Mumford-Tate sub-domain of $D$.  Of interest are the holomorphic bi-sectional  curvatures of $\Phi_\sig(\Delta^{\ast k})$.  We shall not completely answer this, but shall give a proof of the
\begin{Prop}\lab{6b23} $\Theta_I(\eta,N)=0$ for all $N\in \sig$, if and only if, $\eta\in \cZ(\cG_\sig)$.\end{Prop}

\begin{proof}
We denote by $\cG_\C=\opplus_p \cG^{p,-p}$ the Hodge decomposition on the associated graded to the \lmhs\ defined by $\sig$.  The Hodge metric is given on $\cG_\C$ by the Cartan-Killing form, and its restriction to $\cG^{-1,-1}$ is non-degenerate.\footnote{The decomposition of $\cG_\C$ into the primitive sub-spaces and their images under powers of $N$ depends  on the particular $N$.  The Hodge metric on $\cG^{-1,-1}$ is only definite on the subspaces arising from the primitive decomposition for such an  $N$.} 
The decomposition of $\cG_\C$ into $N$-strings for the $\rsl_2$ given by $\{N,Y,N^+\}$ is orthogonal with respect to the Hodge metric, from which we may infer that the adjoint $\ad_{N^\ast}$ acts separately on each $N$-string.  The picture is something like 
\[
\xymatrix{ \eta \circ \ar[r]^\eta & \circ \ar @/^/[l]^{N^\ast} \ar[r]^\eta&\circ\ar[r]^\eta \ar @/^/[l]^{N^\ast}& 
\circ\ar @/^/[l]^{N^\ast}.}\]
Because $N$ is an isomorphism the same is true of $N^\ast$; consequently
\[
\Theta_I (\eta,N)=0\iff \eta\hensp{belongs to an $N$-string of length 1,}\]
and this implies that $[\eta,Y] = [\eta,N^+]=0$.  By varying $N$ over $\sig$ we may conclude the proposition.
\end{proof}

\end{example}

\begin{example}\lab{exam4}
One of the earliest examples of the positivity of the \hlb\  arose in the work of Arakelev (\cite{Ara}).  For 1-parameter families it gives an \emph{upper} bound on the degree of the \hlb\ in terms of the degree of the logarithmic canonical bundle of the parameter spaces.\footnote{Using the above notations, the logarithmic canonical bundle of the parameter space is $K_{\ol B}(Z)$.}  This result has been extended in a number of directions; we refer to \cite{CM-SP}, Section 13.4 for further general discussion and references to the literature.

One such extension is due to \cite{Zuo}, \cite{VZ03} and \cite{VZ}.  This proof of that result   centers around the above  observation  that the curvature of \hb s has a sign on the kernels of Kodaira-Spencer mappings.  There is a new ingredient in the argument that will be useful in other contexts and we shall now explain this.  As above there are singularity issues that arise where the differential of $\Phi$ fails to be injective.  These may be treated in a similar manner to what was done above, and for simplicity of exposition and to get at the essential  new point we shall assume that $\Phi_\ast$ is everywhere injective and that the relevant Kodaira-Spencer mappings have constant rank.

The basic Arakelev-type inequality then exists at the curvature level.  For a variation of Hodge structure $(V,Q,\nabla,F)$ over $B$ with a completion to $\ol B$ with $Z= \ol B\bsl B$ a reduced normal crossing divisor, the inequality  is
\begin{equation}\lab{6b24}
\lrp{\begin{matrix} \hbox{curvature of}\\
\det \Gr^p V\end{matrix}} \leqq C_p\lrp{\begin{matrix} \hbox{curvature of}\\
\det \Om^1_{\ol B}(\log Z)\end{matrix}} \end{equation}
where $C_p$ is a positive constant that depends on the ranks of the Kodaira-Spencer mappings.  Here we will continue using the notations
\begin{equation}\lab{6b25}
\begin{cases}
\Gr^pV=F^pV/F^{p+1} V,\\
 \Gr^p V\xri{\theta}\Gr^{p-1} V\otimes \Om^1_{\ol B} (\log Z).\end{cases}\end{equation}
The second of these was denoted by $\theta^p$ above; we shall drop the ``$p$" here but note that below we shall use $\theta^\ell$ to denote the $\ell^{\rm th}$ iterate of $\theta$.
\end{example}

\begin{proof}[Proof of \pref{6b24}]
Using the integrability condition \pref{6b9} the iterates of \pref{6b25} give
\[
\Gr^p V\xri{\theta^\ell} \Gr^{p-\ell} \otimes \Sym^\ell \Om^1_{\ol B}(\log Z)\]
We use the natural inclusion $\Sym^\ell \Om^1_{\ol B}(\log Z)\subset \ottimes^\ell \Om^1_{\ol B}(\log Z)$ and consider this map as giving
\begin{equation}\lab{6b26}
\Gr^p V\xri{\theta^\ell} \Gr^{p-\ell} V\otimes \lrp{\ottimes^\ell \Om^1_{\ol B}(\log Z)}.\end{equation}
There is a filtration
\[
\ker\theta\subset\ker(\theta^2) \subseteq \dots\subseteq \ker \theta^{p+1} = \Gr^p V\]
and   $\Gr^p V$ has graded quotients
\[
\ker\theta, \frac{\ker\theta^2}{\ker\theta},\dots, \frac{\Gr^p V}{\ker\theta^p}.\]
The crucial observation (and what motivates the above use of $\ottimes^\ell$ rather than $\Sym^\ell$) is
\begin{equation}
\lab{6b27}\begin{split}
&\frac{\ker\theta^\ell}{\ker\theta^{\ell+1}}\hookrightarrow \Gr^{p-\ell+1}V\otimes \lrp{\ottimes^\ell \Om^1_{\ol B}(\log Z)} \\
&\hbox{\em lies in }K^{p-\ell+1}\otimes \lrp{\ottimes^{\ell-1} \Om^1_{\ol B}(\log Z)} \hensp{\em where}\\
&K^{p-\ell+1}=\ker \lrc{ \Gr^{p-\ell+1} V\xri{\theta}\Gr^{p-\ell}\otimes \Om^1_{\ol B}(\log Z)}.\end{split}
\end{equation}
From this we infer that 
\begin{enumerate}[(i)]
\item $K^p,K^{p-1},\dots,K^o$ all have negative semi-definite curvature forms;
\item $\frac{\ker \theta^\ell}{\ker\theta^{\ell-1}}\hookrightarrow K^{p-\ell+1} \otimes \lrp{\ottimes^{\ell-1} \Om^1_{\ol B}(\log Z)}$\end{enumerate}
which gives
\begin{enumerate}[(i)]
\item[(iii)] $\det\lrp{\frac{\ker\theta^\ell}{\ker\theta^{\ell-1}}} \hookrightarrow \wedge^{d_{p,\ell}} \lrp{K^{p-\ell+1}\otimes \lrp{\ottimes^{\ell-1}\Om^1_{\ol B}(\log Z)}}$. \end{enumerate}
Using
\begin{enumerate}[(i)] \item[(iv)] $\det \Gr^p V\cong \ottimes^{p+1}_{\ell=1} \det\lrp{\frac{ \ker \theta^\ell}{\ker\theta^{\ell-1}}}$ 
\end{enumerate}
and combining (iv), (iii) and (ii) at the level of curvatures gives \pref{6b24}.
\end{proof}

\ssni{Note} In \cite{GGK08} there are results that in the 1-parameter case express the ``error term" in the Arakelov inequality by quantities involving the ranks of the Kodaira-Spencer maps and structure of the monodromy at the singular points.

\subsection{The Iitaka conjecture } \hfill

 One of the main steps in the general classification theory of algebraic varieties was provided by a proof of the Iitaka conjecture.  An important special case of this conjecture is the
 \begin{Thm}\lab{6.1}
 Let $f:X\to Y$ be a morphism between smooth projective varieties and assume that
 \begin{enumerate}[{\rm (i)}]
 \item $\Var f=\dim Y$ (i.e., the Kodaira-Spencer maps are generically  1-1);
 \item the general fibre $X_y = f^{-1}(y)$ is of general type.\end{enumerate}
 Then the Kodaira-Iitaka dimensions satisfy
 \begin{equation}\lab{6.2}
 \kappa(X)\geqq \kappa(X_y)+\kappa(Y).\end{equation}\end{Thm}
 
 As noted in the introduction, this result was proved with one assumption (later seen to not be necessary) by 
Viehweg (\cite{Vie83a}, \cite{Vie83b}), 
    and in general by Koll\'ar \cite{Kol87}.  The role of positivity of the \hvb\ had earlier been identified  in \cite{Fuj78},  \cite{Kaw81}, \cite{Kaw83}, \cite{Kaw85} and Ueno \cite{Uen75}, \cite{Uen77}.  Over the years there has been a number  of interesting results concerning the positivity of the \hvb\ and, in the geometric case, the positivity of the direct images of the higher pluricanonical series; cf.\ \cite{PauTak14} and also  \cite{Pau16} and \cite{Sch15} for  recent  results and a survey of some of what is known together with  further references.
    
    To establish \pref{6.2} one needs to find global  sections of $\om^{m}_X$.  From
    \[
    h^0(\om^{m}_X)=h^0(f_\ast \om^{m}_X)\]
    and
    \[
    f_\ast \om^{m}_X=f_\ast \om^{m}_{X/Y}\otimes \om^m_Y,\]
    the issue is to find sections of $f_\ast \om^{m}_{X/Y}$.  If for example $f_\ast \om^{m}_{X/Y}$ is generically globally generated, then we have at least approximately $h^0(\om^{m}_{X_y})\cdot h^0(\om^{m}_Y)$ sections of $f_\ast \om^{m}_X$ which leads to the result. 
    
    To find sections of $f_\ast \om^{m}_{X/Y}$, if  for example  $\Sym^m f_\ast \om_{X/Y}$ has  sections, then since the multiplication mapping
    \begin{equation}\lab{6.3}
    \Sym^m f_\ast \om_{X/Y}\to f_\ast \om^{m}_{X/Y}\end{equation}
    is injective on decomposable tensors, we get  sections of the image.
    In general the issue is to find sections of $\Sym^{m''} f_\ast \om^{m'}_{X/Y}$ and use an analogue of \pref{6.3}.
    
     The arguments in \cite{Vie83b}, \cite{Kol87} have two main aspects:
    \begin{enumerate}
    \item the use of Hodge theory;
    \item algebro-geometric arguments using Viehweg's notion of weak positivity.
    \end{enumerate}
    In the works cited above, Hodge theory is used to show that under the assumptions in \pref{6.1} for $m\gg 0$
  we have
    \[ 
    \kappa(\det f_\ast \om^{m}_{X/Y}) = \dim Y.\]
    From this one wants to infer that $\Sym^{m''} f_\ast \om^{m'}_{X/Y}$ has  sections for $m',m'' \gg 0$.  This is where (ii) comes in.
    
    The objectives of this section  are twofold.  One is to show that in case local Torelli holds for $f:X\to Y$ step (ii) may be directly circumvented by using the special form
  $
    \Theta_F = -{}^t \ol A\wedge A$
    of the curvature in the \hvb\ where $A$ has algebro-geometric meaning that leads to  positivity properties.\footnote{As noted above, $A$ is the end piece of the differential of the period mapping.}  The other   is to discuss Viehweg's branched covering construction, which provides a mechanism to apply  the positivity properties of the \hvb\  to the pluri-canonical series.   
     
   One issue has been that the assumptions (i), (ii) plus Viehweg's branched covering method give
   \begin{equation}\lab{2}
   \kappa\lrp{\det\lrp{f_\ast \om^{m}_{X/Y}}}= \dim Y.\end{equation}
   From this one wants to show that
   \begin{equation}\lab{3}
   \lrc{\begin{matrix} \lrp{\det f_\ast \om^{m}_{X/Y}}^{m'} \hbox{ has} \\ \hbox{lots of sections}\end{matrix}} \implies
   \lrc{\begin{matrix} \Sym^{m''} \lrp{f_\ast \om^{m}_{X/Y}} \hbox{ has}\\
   \hbox{lots of sections}\end{matrix}}.\end{equation}
   The Viehweg-Koll\'ar method proves \pref{2} using the positivity properties of the Hodge \emph{line} bundle, and from this goes on to infer \pref{3} by an algebro-geometric argument involving Vieweg's  concept of weak positivity for a coherent sheaf.\footnote{The definition for a vector bundle was recalled above.}  The positivity of the Hodge \emph{vector} bundle does not enter directly.

  We will    give a four-step sketch of the proof of  Theorem \ref{6.1}, one that avoids the use of weak positivity.   The first three steps   follow   from the discussions above.  The fourth step  uses a variant of Viehweg's argument to derive Hodge theoretic information from the 
   pluri-canonical series together with \pref{3b7} and \pref{3b8}.
      
   \ssni{Step one {\rm (already noted above)}} Suppose that the fibres of $X \xri{f} Y$ are smooth and that $\Phi_{\ast,n}$ is injective at a general point.  Then by Theorem \ref{3b1}   \[ 
   \kappa\lrp{\Sym^{h^{n,0}}(f_\ast \om_{X/Y})} = \dim \P_{f_\ast \om_{X/Y}}(1)> \dim Y.\]    
   
   \ssni{Step two} This is the same as step one but where we allow singular fibres.  What is needed to handle these  follows from the discussion in Section \ref{4b}.

   \ssni{Reformulation of step two} We set $S^k = \Sym^k$ and
   \begin{itemize}
   \item $F_e = f_\ast \om_{X/Y}$;
   \item $\om _k= $ curvature form of $\cO_{\P S^k F_e}(1)$ 
   \end{itemize}
   where $\om_k$ is the  (1,1) form computed using the induced Hodge metric in $S^k F$.  Then $\om_k>0$ in the vertical tangent space to a  general fibre of $\P S^k F_e\to Y$.  Using the injectivity of the Kodaira-Spencer map
   \[
 T_y Y \to \Hom(F^n_y, F^{n-1}_y/F^n_y)\]
   at a general point $y\in Y$, from Theorem \ref{3b1} it follows that for  all $k\geqq h^{n,0}$
   \begin{equation}\lab{6}
   \bmp{2.5}{\em $\om_k>0$ in the horizontal space at a general point in $(\P S^kF_e)_y$.}\end{equation}
   This gives
   \[
   \om > 0 \hensp{in an open set in} \P S^k F_e,\]
 which implies that $\cO_{\P S^k F_e}(1)$ is big. This in turn implies the same for $S^k F_e$.
   
   \ssni{Step three}   Since $S^k f_\ast \om_{X/Y}$ is big, $S^\ell (S^k f_\ast \om_{X/Y})$ is generically globally generated for $\ell\gg 0$.  It follows that the direct summand $S^{k \ell}f_\ast \om_{X/Y}$ is generically globally generated for $\ell \gg 0$.  Since
   \[
   S^{k\ell}f_\ast \om_{X/Y}\to f_\ast \om^{k \ell}_{X/Y}\]
   is injective on decomposable tensors, we obtain at least approximately
   \[
   h^0 \lrp{\om^{ k\ell}_{X_y}}   h^0 \lrp{\om^{ k\ell}_Y}\]
   sections of $\om^{k\ell}_X$.
   
   \begin{Rem}\lab{r6b7}  In the geometric case an alternative geometric argument that $S^k f_\ast \om_{X/Y}$ is  big may be given as follows:
   
   First, for a family $W\xri{g} Y$ with smooth general fibre $W_y = g^{-1}(y)$, the condition \pref{3b3} for bigness of $g_\ast \om_{W/Y}$ is:
 {\setbox0\hbox{(ii)}\leftmargini=\wd0 \advance\leftmargini\labelsep
   \begin{enumerate}
   \item for general $y\in Y$  the Kodaira-Spencer map $\rho_y :T_y Y\to H^1(T {W_y})$ should be injective;
   \item for general $\psi\in (g_\ast \om_{W/Y})_y = H^0 (\Om^n_{W_y})$ the map
   \[
   H^1(T {W_y}) \xri{\psi} H^1(\Om^{n-1}_{W_y})\]
   should be injective on the image $\rho_y (T_yY) \subset H^1(TW_{y})$.
   \end{enumerate}}  
   
  Next, for  a famliy $X\to Y$ with generically injective Kodaira-Spencer mappings, we set
 \[
 W= X\overbrace{ \times_Y \times\cdots \times_Y}^k X.\]
 Then $H^0(\Om^n_W)$ contains $\ottimes^k H^0(\Om^n_X)$ as a direct summand and the same argument as in the proof of Theorem \pref{3b7} gives that $\ottimes^k f_\ast \om_{X/Y}$ is big.
 
 Finally, we may apply a similar argument to a desingularization of the quotient of $W\to Y$ by the action of the symmetric group.  Then the general fibre is
 \[
 X^{(k)}_y  = \hbox{ desingularization of }\Sym^k X_y\]
 with
 \[
 S^k H^0 (\Om^n_{X_y}) = \hbox{ direct summand of } H^0(\Om^n_{X_y^{(k)}})\]
 and again the argument in the proof of Theorem \ref{3b7} will apply.\end{Rem}
   
   \ssni{Step four} The idea is to apply the reformulation of step two with $f_\ast \om^{m}_{X/Y}$ replacing $f_\ast \om_{X/Y}$.  This will be discussed  below in which  the pluricanonical series of a smooth variety will be seen to have Hodge theoretic interpretations.\footnote{An alternative approach to the Iitaka conjecture which replaces the use of Hodge theory by vanishing theorems and also uses the cyclic covering trick has been used by Koll\'ar (\emph{Ann. of Math.} {\bf 123} (1986), 11--42).  The relation between vanishing theorems and Hodge theory is classical dating to Kodaira-Spencer and Akizuki-Nakano in the 1950's.  As noted by a number of people, Hodge theory in some form and the curvature properties of the Hodge bundles seem   to generally be lurking behind the positivity of direct images of pluricanonical sheaves.}  
   
   An alternate more direct  approach would be to have a metric 
 in
   \[
   \cO_{\P S^k f_\ast \om^{m}_{X/Y}}(1)\]
   whose Chern form $\om$ is positive in the   space at a general point of $\P S^k f_\ast \om^{m}_{X/Y}$.  For this    approach  to work one would  need  to have a metric in
   \[
   f_\ast \om^{m}_{X/Y}\]
   that has   a property analogous to that   obtained using the Kodaira-Spencer map in the case $m=1$ considered above.   Here one possibility might be to use the  relative Bergman kernel metrics that have appeared from the recent work of a number of people; cf.\ \cite{PauTak14} and the references cited therein.  This possibility will be discussed further at the end of Section VI.C.

 Before giving the detailed discussion we will give some general comments.
 
 (a) A guiding heuristic principle is
 \begin{equation}\lab{6c1}
 \bmp{5}{\em For families $f:X\to Y$ of varieties of general type the 
 \[
 f_\ast \om^k_{X/Y}\]
 become more positive as $k$ increases.}\end{equation}
 The reasons are
 \begin{itemize}
 \item $f_\ast \om_{X/Y}\geqq 0$ by Hodge theory, and $S^k f_\ast \om_{X/Y}$ becomes more positive with $k$ by the argument in the proof of Theorem \ref{3b7};
 \item the map
 \[
 S^\ell f_\ast \om^k_{X/Y}\to f_\ast \om^{k\cdot\ell}_{X/Y}\]
 is non-trivial since it is injective on decomposable tensors (similar to the   proof of Clifford's theorem);
 \item passing to the quotient increases positivity (curvatures increase on quotient bundles).\end{itemize}
 
 (b) In what follows we will without comment make simplifying modifications by replacing $f:X\to Y$ by $f':X'\to Y'$ where all the varieties are smooth and where
 \begin{itemize}
 \item $Y'\to Y$ is an isomorphism outside a codimension 2 subvariety of $Y$;
 \item $X'$ is a desingularization of $X\times_{Y'}Y$ and $f'$ is flat;
 \item the fibres $X'_{y'}=f^{'-1}(y')$ are smooth outside a normal crossing divisor in $Y'$ around which the local monodromies are unipotent (\cite{Vie83b}, page 577).
 \end{itemize}
 
 (c) We now recall the  two constructions from the introduction that associate to $f_\ast \om^m_{X/Y}$ a variation of Hodge structure; these constructions will be done first for a fixed smooth $W$ and then for a family $f:X\to Y$ where $W$ is a typical general fibre $X_y$.  The objectives are to illustrate how pluricanonical series can give rise to \hs s.
 
 \subsection*{(i) The case of fixed $W$, \hs s associated to $H^0(K^m_W)$}\hfill
 
 We summarize and establish notation for the standard construction of a cyclic covering
 \[
 \wt W_\psi \to W\]
 associated to $\psi\in H^0(K^m_W)$ whose divisor $(\psi)\in |mK_X|$ is smooth.  The construction is
 \begin{equation}
 \lab{6c2}
 \xymatrix@C=.5pt@R=1.5pc{
 \wt W_\psi\ar[d]_\pi&\subset &\ar[d]K_W&=\hbox{ total space of the line bundle } \\
 W&=&W}\end{equation}
 where
 \[
 \wt W_\psi= \{(w,\eta):\eta\in K_{W,w},\eta^m = \psi(w)\}.\]
 Then the direct image 
 \begin{equation}\lab{6c3}
 \psi_\ast K_{\wt W_\psi}\cong \opplus^{m-1}_{i=0} K^{m-i}_W \end{equation}
 which gives
 \begin{equation}\lab{6c4}
 H^0(K_{\wt W_\psi}) \cong \opplus^{m-1}_{i=0} H^0 (K^{m-i}_W).\end{equation}
 In this way pluri-differentials on $W$ become ordinary differentials on $\wt W_\psi$; this is the initial step in the relation between the pluricanonical series and Hodge theory.
 
 We observe that  the cyclic group $\G_m$ acts on $\wt W_\psi\to W$ with the  action by $\zeta =e^{2\pi i/m}$ on $H^0(K_{\wt W_\psi})$ and by $\zeta^{i+1}$ on $H^0(K^{m-i}_X)$ in \pref{6c4}.
 
 We also observe the diagram
 \begin{equation} \lab{6c5}
  \begin{split}
 \xymatrix{\wt W_\psi\ar[d]\ar[r]^\sim&\wt W_{\la\psi}\ar[d]\\
 W\ar@{=}[r]&W} \end{split}
 \end{equation}
 induced by scaling the action of $\la\in \C^\ast$ on $K_W\to W$ in \pref{6c2}.  \emph{The isomorphism $\simto$ in \pref{6c5} depends on the choice of $\la^{1/m}$.}
 
 We denote by $H^0(K^m_W)^0\subset H^0(K^m_W)$ the $\psi$'s with smooth divisor $(\psi)$, and we set
 \[
 \P^0 = \P H^0(K^m_W)^0 \subset \P = H^0 (K^m_W).\]
 Then $H^0(\cO_\P(1))\cong H^0(K^m_W)^\ast$, and the identity  gives a canonical section
 \begin{equation}
 \lab{6c6}
 \Psi\in H^0(W\times \P,K^m_W\boxtimes \cO_\P(1)).\end{equation}
 \begin{defin}
 $(\Psi)$ is the \emph{universal divisor of the $(\psi)$'s for $\psi\in H^0(K^m_W)$}.\end{defin}
 
 To construct a universal family of cyclic coverings $\wt W_\psi\to W$ it is necessary to choose an auxiliary  cyclic covering 
 \[
 \wt \P\xri{q}\P,\]
 with $q^\ast \cO_\P(1)=L^m$ for an ample \lb\ $L\to \wt \P$.  From 
 \[
 W\times \wt \P\xri{\id \times g} W\times \P \]
we obtain
 \[
 H^0(W\times \wt \P,K^m_W\boxtimes L^m) \cong H^0(W\times\wt\P,(K_W\boxtimes L)^m)\]
 and using the pullback to $W\times \wt \P$ of $\Psi$ in \pref{6c6}  there is  a cyclic covering
 \[
 \wt W_\Psi:= \wt{W\times \P}\xri{h}W\times \P\]
 branched over the universal divisor $(\Psi)$ and  depending on the choice of $\wt \P\to \P$.
 In this way the choice of isomorphisms in \pref{6c5} necessitated by choosing an $m^{\rm th}$ root of $\la$ may be made uniform.  We observe that
  \begin{equation}\lab{6c7}
 \bcs
 \wt \Psi\in H^0(K_{\wt W_\Psi/\wt \P}) \\
 \wt \Psi^m = p^\ast \Psi.\ecs\end{equation}
 
 Setting $\wt  W^0_\Psi = h^{-1} (W\times \P^0)$ the total space and the fibres of 
 \[
 \wt W^0_\Psi\to \wt \P^0\]
 are irreducible and smooth.  This gives a period mapping, or equivalently a VHS,
 \begin{equation}\lab{6c8}
 \Phi:\wt \P^0\to\Ga\bsl D.\end{equation}
 \begin{Prop}\lab{6c9}
 The \hvb
 \[
F_\Psi\cong \opplus^{m-1}_{i=0} H^0(K^{m-i}_W)\otimes L^{i+1}.\]
\end{Prop}
 
 Arguments as in sections (1.4)--(1.7) of \cite{Vie83b} show that for $m\gg 0$ local Torelli holds for the period mapping \pref{6c8}.  In fact, we have the 
 \begin{Prop}\lab{6c10} For $m\gg 0$ the part 
 \[
 \Phi_{\ast,n}:T\wt \P^0 \to \Hom \lrp{H^0(K^m_W)\otimes L, F^{n-1}_\Psi/F^n_\Psi}\]
 of the  end piece of $\Phi_{\ast,n}$ is injective.\end{Prop}
 
 This is proved by showing that $\Phi_\ast$ is injective on certain of the eigenspaces for the $\G$-equivariant action in the picture, and that the $\zeta$-eigenspace is among those that are included.
 
 \ssni{Discussion of singularities}   
For $\psi\in H^0 (K^m_W)$ the Finsler-type norm
 \begin{align*}
 \|\psi\|&= \int_W (\psi\wedge\ol\psi)^{1/m}\\
 &=\int_{\wt W_\psi}\wt \psi\wedge\ol{\wt \psi}\end{align*}
 is equal to the square of the  Hodge length of $\wt \psi=\psi^{1/m} \in H^0(K_{\wt W_\psi})$.  Even when the divisor $(\psi)$ acquires singularities so that $\wt W_\psi$ becomes singular, the Hodge length of the canonical section $\wt \psi$ will remain finite.  However, although $\|\psi\|$ is continuous in $\psi$, it is not smooth as its derivatives detect singularities of the degenerating Hodge structures.\footnote{Norms of this type appear in \cite{NarSim68} and have been used extensively in the recent literature (cf.\ \cite{Pau16} for a summary and survey).}  In terms of \lmhs s, $\wt \Psi$ lies in the lowest possible weight   part.
  
 The second method of associating Hodge theoretic data to the pluricanonical series $H^0(K^m_W)$ is the following: Let $\wt W_\Psi$ be a desingularization of a completion of $\wt W^0_\Psi$ and
 \[
 \wt W_\Psi\xri{\wt \pi} \P\]
 the resulting fibration.  If $\dim \wt W_\Psi=\tilde n$, then $H^{\tilde n}(\wt W_\Psi)$ has a \phs.  The general Hodge theory of  maps such as $\wt \pi$   is contained in the \emph{decomposition theorem} (\cite{deC-Mig}).  In the case at hand a special feature arises in that from Proposition \ref{6c9} we may infer that
 \begin{equation}\lab{6c11} \bmp{5}{\em $H^{\tilde n,0}(\wt W_\Psi)$ contains as a direct summand
 \[
 \opplus^{m-1}_{i=0} H^0(K^{m-i}_W) \otimes H^0(\cO_\P(n_i)).\]}
  \end{equation}
 Here the $n_i>0$ are determined by the $\nu_i>0$
 in
 \[
 q_\ast L^{i+1}\cong \oplus \cO_\P(\nu_i).\]
 
\subsection*{(ii)  The case of a family $f:X\to Y$; Hodge structures associated to $f_\ast \om^m_{X/Y}$}\hfill

 The following is a sketch of the proof of Theorem \ref{6.1}.  It is intended  to point out some of the geometric aspects of the arguments in \cite{Vie83b}, in particular the way in which the Hodge-theoretic interpretations enter into those arguments, referring to that paper for the details.
 
   Since the publication of \cite{Vie83b} several important general results concerning the pluricanonical series have been established (cf.\ \cite{Dem12a}), and   we shall assume the following:
 \begin{itemize}
 \item for general $y\in Y$, $X_y$ is smooth and there is an $\ell$ such that for all $m=k\ell$, $k\geqq 1$, the linear system $|mK_{X_y}|$ is ample; whenever an $m$ appears below it will be of this form; and
 \item the assumptions to have local Torelli in the form used in \cite{Vie83b} are satisfied for $f_\ast \om^m_{X/Y}$.
 \end{itemize}
 
 The \emph{basic diagram} is
 \begin{equation}\lab{6c12}\begin{split}
 \xymatrix{\wtcx\ar[d]_{\tilde f} \ar[r]& \cX_2\ar[d]^{\tilde h}\ar[r]& \cX_1\ar[r]\ar[d]^h& X\ar[d]^f\\
 \wt \P\ar[r]^g& \wt \P\ar[r]^q& \P\ar[r]^p&Y}\end{split}\end{equation}
 where
 \begin{itemize}
 \item $\P=\P f_\ast \om^m_{X/Y}$, so that $p_\ast \cO_\P(1) = f_\ast \om^{m\enspace^\ast }_{X/Y}$ is the \emph{dual} of $f_\ast \om^m_{X/Y}$;
 \item $\cX_1= X\times_Y \P$;
 \item $\wt \P\xri{q} \P$ is a cyclic branched covering where there is an ample line bundle $L\to \wt \P$ with
 \[
 q^\ast \cO_\P(1)=L^m;\]
 \item $\cX_2 = \cX_1 \times_\P \wt \P$; and
 \item $\wt \cX\to \cX_2$ is the cyclic covering obtained by globalizing the construction of $\wt W_\Psi\to W\times\wt \P$ given by the completion of $\wt W^0_\Psi$ in \pref{6c7} above.\end{itemize}
 
 The players in the basic diagram are
 \begin{itemize}
 \item $f^{-1}(y)=X_y$;
 \item $p^{-1}(y) = \P H^0(\om^m_{X_y})$, whose points are $[\psi]$ where $\psi\in H^0(\om^m_{X_y})$;
 \item $q^{-1}([\psi])=\{[\psi_i]\}$ where $[\psi_i]\xri{q}[\psi]$ under the cyclic covering;
 \item $\Psi$ is the tautological section of $\om^m_{\cX_1/\P}\otimes h^\ast \cO_\P(1)$;
 \item $\wt \Psi$ is the tautological section of
 \[
 \om^m_{\cX_2/\wt \P}\otimes \tilde h^\ast L^m = \lrp{\om_{\cX_2/\wt \P} \otimes \tilde h^{\ast}L}^m;\]
 \item $\wt \cX\to \cX_2$ is the $m$-sheeted cyclic covering obtained by extracting an $m^{\rm th}$ root of $\wt \Psi$.
 \end{itemize}
 We will denote the fibre over $y$ of $\wt \cX\to Y$ by
 \begin{equation}\lab{6c13}
 \wt \cX_y = \bigcup_{[\psi]\in\P H^0\lrp{\om^m_{X_y}}} \wt X_{y,\psi}\end{equation}
 where $X_y$ corresponds to $W$ and $\wt X_{y,\psi}$ to $\wt W_\psi$ above.
 
 There are two families of varieties constructed from the basic diagram.  Denoting the composition $q\circ g\circ \tilde f$ by $G$ for the first we have
 \begin{equation}\lab{6c14}
 G:\wt \cX\to \P\end{equation}
 whose fibre over $(y,[\psi])\in \P$ is $\wt X_{y,\psi}$.  The second is
 \begin{equation}\lab{6c15}
 F:\wt \cX\to Y\end{equation}
 whose fibre over $y\in Y$ is the variety \pref{6c13}.
 
 There are two important observations concerning these families:
 \begin{align}\lab{6c16}&\bmp{5}{{\em Generic local Torelli holds for both families;} and}\\
 \lab{6c17}&\bmp{5}{\em The basic diagram \pref{6c12} is commutative.}\end{align}
 The fibres of $G_\ast\om_{\wt \cX/\P}$ are given by
 \[
 \bp{G_\ast\om_{\wt \cX/\P}}_{(y,[\psi])}=H^0\bp{\om_{\wt X_{y,\psi}}}.\]
 From \pref{6c16} we have
 \begin{equation}\lab{6c19}
G_\ast \om_{\wtcx/\P}\geqq 0\hensp{\em and} \det G_\ast \om_{\wt \cX/\P}>0 \hensp{\em on an open set.}\end{equation}
 From 
 \pref{6c19}  we have
 \begin{equation}\lab{6c20}\bmp{5}{\em For $k\gg 0$, both $S^k G_\ast \om_{\wtcx/\P} $ and $S^k F_\ast\om_{\wtcx/Y}$ are big.}\end{equation}
 
 To complete the proof of Theorem \ref{6.1} the argument one first might try to make is this:
 \begin{quote} \em 
 $H^0(\om^m_{X_y})$ is a direct factor of $H^0(\om_{\wt X_{y,\psi}})$, and from the commutativity of the basic diagram \pref{6c12} it follows that $f_\ast \om^m_{X/Y}$ is a direct factor of the \hvb\ associated to the family $\wtcx \to Y$.  By the local Torelli property \pref{6c16} with the implication \pref{6c20}, it follows that $S^k f_\ast \om^m_{X/Y}$ is big.\end{quote}
 However, this argument is not   correct; the issue is more subtle.  The problem is that under the mapping
\vspace*{-3pt} \[
 \wtcx \xri{H} \cX_1\vspace*{-3pt}\]
 in \pref{6c12} we do \emph{not} have
\vspace*{-3pt} \[
 H_\ast \om_{\wtcx/\P}\cong \opplus^{m-1}_{i=0} \om^{m-i}_{\cX_1/\P},\vspace*{-3pt}\]
 but rather 
\vspace*{-3pt} \[
 H_\ast \om_{\wtcx/\P}\cong \opplus^{m-1}_{i=0} \om^{m-i}_{\cX_1/\P} \otimes h^\ast \cO_\P (i+1)\]
 where the $\cO_\P(i+1)$'s reflect the global twisting of the identification \pref{6c3}.  
 This would not be an issue  if $\cO_\P(1)$ were  positive.  This positivity    trivially holds along the fibres of $\P\to Y$, but since $p_\ast \cO_\P(1)=f_\ast \om^{m\enspace^\ast }_{X/Y}$  any positivity of $f_\ast \om^m_{X/Y}$ becomes negativity of $\cO_\P(1)$ in directions normal to fibres of $\P\to Y$.  An additional step is required.

 The key observation is that (cf.\ \cite{Vie83b}, page 587)
 \begin{quote}
 For $a>0$, $\om^{ma+a+1}_{\cX_1/\P}\otimes h^\ast \cO_\P(a)$ is a direct summand of $G_\ast \om^{a+1}_{\wtcx/\P}$.\end{quote}
 It is the additional factor of $a+1$ in $\om^{ma+a+1}_{\cX_1/\P}$ that offsets the negativity of $\cO_{\P}(a)$ in the normal direction of the fibres of $\P\to Y$.
 
 Thus what needs to be shown is
\vspace*{-3pt} \begin{equation}
 \lab{6c21}
 S^k G_\ast \om^{a+1}_{\wtcx/\P}\hensp{\em is big for} k\gg 0.\vspace*{-3pt}\end{equation}
 This follows from \pref{6c21} if we  have
\vspace*{-3pt} \[
 S^{a+1} G_\ast \om_{\wtcx/\P}\to G_\ast \om^{a+1}_{\wtcx/\P}\to 0,\vspace*{-3pt}\]
 and this may be accomplished generically in $Y$ by choosing $m\gg 0$.

 \begin{Rem}\lab{6c22}
 We conclude with a comment and a question.  The comment is that perhaps the most direct way to prove the Iitaka Conjecture \ref{6.2} would be to  use the curvature properties of the Finsler-type metric in $f_\ast \om^m_{X/Y}$.  Specifically, referring to \cite{Ber09}, \cite{BerPau12} and \cite{PauTak14} for details we set 
 \[
 \P^\ast  = \P(f_\ast \om^m_{X/Y})^\ast.\]
 Then $\cO_{\P^\ast}(1)\xri{\hat \pi} Y$ has fibres
 \begin{equation}\lab{6c23}
 \cO_{\P^\ast}(1)_{(y,[\eta])} = H^0(\om^m_{X_y})/\eta^\bot,\qquad \eta \in H^0(\om^m_{X_y})^\ast\end{equation}
 and
 \[
 \hat\pi_\ast \cO_{\P^\ast}(1) = f_\ast\om^m_{X/Y}.\]
 Following \cite{Kaw82} and \cite{PauTak14}, we define a metric in $\cO_{\P^\ast}(1)$ by taking the infinimum of the $\|\psi\|$'s where $\psi\in H^0(\om^m_{X_y})$ projects to a fixed vector in the quotient \pref{6c23}.  Leaving aside  the question of singularities, for this metric the form $\om_m$ satisfies
 \[
 \om_m\geqq 0.\]
  We do not expect to have $\om_m>0$ in an open set in $\P^\ast$: this is already not the case when $m=1$.  Instead we carry out a similar construction replacing $f_\ast \om^m_{X/Y}$ by $S^{m'}f_\ast \om^m_{X/Y}$ and denote by $\om_{m',m}$ the Chern form of the corresponding $\cO(1)$-bundle.
 
 \ssn{Question:} \emph{Assuming $\Var f=\dim Y$, for $m'\gg 0$  do we have $\om_{m',m}>0$ in an open set?}
  \end{Rem}
 
 This is true when $m=1$, and if it   holds for general $m$ and the issue of singularities can be handled one would have a direct ``curvature" proof of \ref{6.2}.  In fact, heuristic reasoning suggests the following
 \begin{Conj}\lab{6c24} Under the assumptions in Theorem \ref{6.1}, let $m\geqq 1$ be such that for general $y\in Y$ the bundle $K^m_{X_y}$ is globally generated.  Then for $P_m(X_y)=h^0(K^m_{X_y})$ 
 \[
 \Sym^{m'} f_\ast \om^m_{X/Y}  \hensp{is big for} m'\geqq P_m(X_y).\]\end{Conj}

 \subsection{The Hodge vector bundle  may detect   extension data}\label{S:extdata}

\begin{Prop}
\lab{b.19}
On a subvariety $Y\subset Z^\ast_I$ along which the period mapping $\Phi_I$ is locally constant, the extended Hodge bundle $F_{e,I}\to Y$ is flat.  It may however have non-trivial monodromy.\end{Prop}

\begin{proof}  $F_{e,I}$ is filtered by $W_\bullet(N_I)\cap F_{e,I}$. For $\V_I$ the local system corresponding to $\Phi_I:Z^\ast_I\to \Ga_I\bsl D_I$, the Gauss-Manin connection $\nabla$ acting on  $\cO_{\Z^\ast_I}(\V_I)$ preserves $W_\bullet(N_I)$, and on the associated graded to $W_\bullet(N_I)\V_I$ it preserves the associated graded to $W_\bullet(N_I)\cap F_{e,I}.$  It follows that $\nabla$ preserves $F_{e,I}\subset  \cO_Y(\V)$.
\end{proof}

\begin{Exam}
On an algebraic surface $S$ suppose we have a smooth curve $\wt C$ of genus $g(\wt C)\ge 2$, and that through each pair of distinct points $p,q\in\wt C$ there is a unique rational curve meeting $\wt C$ at two points.
\[
\begin{picture}(200,90)
\put(30,10){\includegraphics[scale=.5]{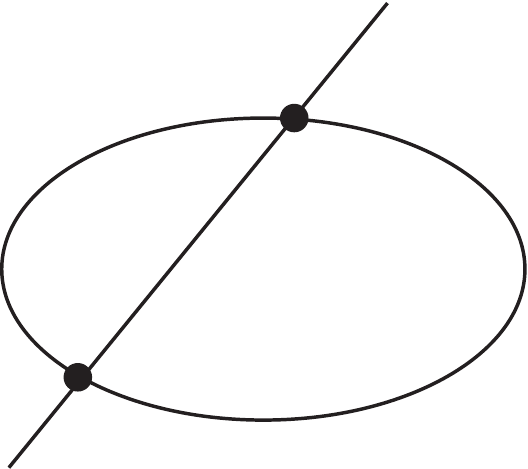}}
\put(114,30){$\wt C$}
\put(74,80){$\P^1$}
\put(42,13){$p$}
\put(72,52){$q$}
\end{picture}\]
Suppose moreover that along the diagonal $D\subset \wt C\times \wt C$ the $\P^1$ becomes simply tangent.  Then $\wt C+\P^1$ is a nodal curve, and for $m\gg 0$ the pluricanonical map given by $|m\om_{\wt C+\P^1}|$ contracts the $\P^1$ and we obtain an irreducible stable curve $C_{p,q}$ with arithmetic genus $p_a(C_{p,q})=3$.
\[
\begin{picture}(120,50)
\put(0,0){\includegraphics[scale=.7] {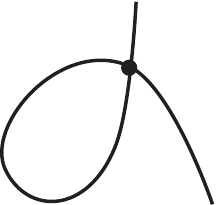}}
\put(36,25){$p=q$}
\end{picture}\]
As $\P^1$ becomes tangent  we obtain a cusp.

The extension data for the MHS is given  (cf.\  \cite{Car80}) by
\[
\AJ_{{\wt C}}(p-q)\in J(\wt C).\]
When we turn around $p=q$ we interchange $p,q$; after base change $t\to t^2$ the extension class is well defined locally.  Globally, it is another story.

The vector space $H^0(\om_{C_{p,q}})$, which is the fibre of the canonically extended Hodge bundle, has the 2-dimensional fixed subspace $H^0(\Om^1_{\wt C})$ and variable 1-dimensional quotient represented by a differential of the third kind $\vp_{p,q}\in H^0(\Om^1_{\wt C}(p+q))$ having non-zero residues at $p,q$. 
As $p\to q$ the differential $\vp_{p,q}$ tends to a differential of the second kind $\vp_{2p} \in H^0(\Om_{\wt C}(2p))$ with non-zero polar part at $p$.  For the global monodromy over $Y=\wt C\times \wt C\bsl D$, the extension class is given for $\psi\in H^0(\Om^1_{\wt C})$ by
\[
\psi \to \int^q_p \psi \hensp{(modulo periods).}\]
The action of $\pi_1(Y)$ on $H_1(\wt C,\{p,q\})$ may then   be shown to give an infinite subgroup of $H_1(\wt C,\Z)$, which implies the  assertion.

\begin{claim}
The Hodge vector bundle may detect continuous extension data.
\end{claim}

\noindent Here continuous extension data means the extension classes that arise in the induced filtration of $F_e\to Y$ where there is a variation of \lmhs\ over $Y$ with $F_e$ the \hvb.  
  \end{Exam}

 \begin{Exam} We let $\wt C$ be a smooth curve with $g(\wt C)\ge 2$ and $p\in\wt C$ a fixed point.  We construct a family $C_q$, $p\in \wt C$, of stable curves as follows.
 
 \begin{itemize}
 \item For $q\ne p$ we identify $p,q$
 \[
 \begin{picture}(350,90)
\put(0,20){\includegraphics[scale=.90]{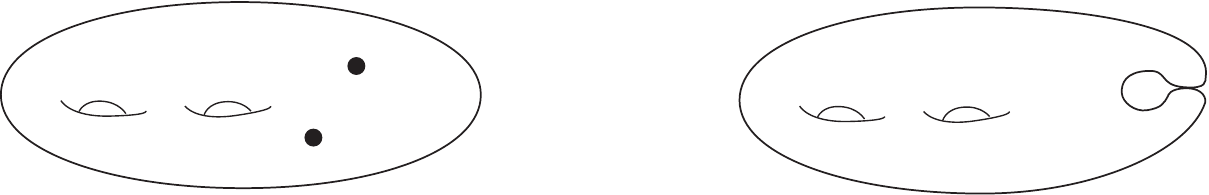}}
\put(88,30){$ p$}
\put(80,52){$q$}
\put(60,0){$\wt C$}
\put(144,40){\vector(1,0){30}}
\put(250,0){$C_q$}
\end{picture}
\]
\item For $p=q$, we obtain a curve
\begin{equation}\lab{b.21} 
\begin{split}
 \begin{picture}(150,90)
\put(0,20){\includegraphics[scale=.90]{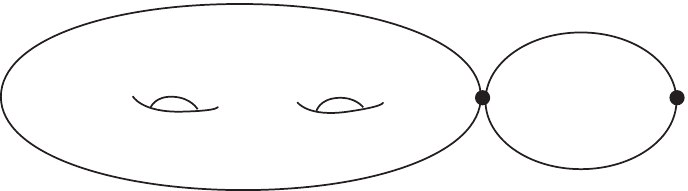}}
\put(150,10){$E$}
\put(121,30){$p$}
\put(182,42){$r$}
\put(60,5){$\wt C$}\end{picture} \end{split}
\end{equation}
 \end{itemize}
 In this way we obtain a VMHS parametrized by $\wt C$ and with trivial monodromy.  

 For the filtration on the canonically extended Hodge bundle there is a fixed part
\[
W_1(N)\cap F^1_e\cong H^0(\Om^1_{\wt C}),\]
and a variable part whose quotient is
\[
W_2(N)\cap F^1_e / W_1(N)\cap F^1_e=\bcs 
\C \vp_{p,q}& \hensp{for} q\ne p\\
\C \vp_{E,r}&\hensp{for} q=p.\ecs\]
The notation means that $\vp_{E,r}$ is the differential of the third kind on  the normalization of $E$ and with residues at  $\pm 1$ at the two points over the node.  Since there is no monodromy we may normalize the $\vp_{p,q}$ and $\vp_{E,r}$.   

Combining the above  we have over $\wt C$ an exact sequence
\[
0\to H^0(\Om^1_{\wt C})\otimes \cO_{\wt C}\to F_e\to\cO_{\wt C}\to 0\]
with non-trivial extension class
\[
e=\hbox{``Identity"} \in H^1(\cO_{\wt C})\otimes H^0(\Om^1_{\wt C}).\]
 \end{Exam}

    \subsection{The exterior differential system defined by a Chern form}
    In this section we will discuss the exterior differential system
    \begin{equation}\lab{6e1}
    \om=0\end{equation}
    defined by the Chern form of the \lb\ $\cope(1)$ where $E\to X$ is a Hermitian \vb\ whose curvature has the norm positivity property \pref{3a1}. Without assuming the norm positivity property, this type of EDS has been previously studied in \cite{BeKa77} and \cite{Som59} and also appeared in \cite{Kol87}.
    
    Here our motivation is the following question:
    \begin{equation}\lab{6e2}\bmp{5}{\em Under what conditions can one say that the Kodaira-Iitaka dimension of $E\to X$ is equal to its numerical dimension?}\end{equation}
    \begin{Prop}\lab{6e3}
    The exterior differential system \pref{6e1} defines a foliation of $\P E$ by complex analytic subvarieties $W\subset \P E$ with the properties
    \begin{enumerate}
    \item $W$ meets the fibres of $\P E\xri{\pi} X$ transversely; thus $W\to\pi(W)$ is an \'etal\'e map;
    \item the restriction $E\big|_{\pi(W)}$ is flat.\end{enumerate}\end{Prop}
    
    \begin{proof}
    Since $\om >0$ on the fibres of $\P E\to X$, the vectors $\xi\in T_{(x,[e])}\P E$ that satisfy $\om(\xi)=0$ project isomorphically to $T X$.  The image of these vectors is the subspace (here identifying $\xi$ with $\pi_\ast(\xi)$)
    \begin{equation} \lab{6e4}
    \{\xi\in T_xX:A(e\otimes \xi)=0\}.\end{equation}
    This is the same as the subspace of $T_xX$ defined by
    \[
    \Theta_E(e\otimes \xi)=0,\]
    which implies that $E\big|_{\pi(w)}$ is flat.
    \end{proof}
    
    \begin{rem}
    Given any holomorphic bundle map
    \begin{equation}\lab{6e5}
    A:T X\otimes E\to G,\end{equation}
    if we have a  metric in $E\to X$   we may use it  to identify $E\cong E^\ast$ and then define the horizontal sub-bundle $H\subset T\cope(1)$.  It follows that  \pref{6e4} defines a $C^\infty$ distribution (with jumping fibre dimensions) in $T\cope(1)$, and when the map \pref{6e5} arises from the curvature of the metric connection as in \pref{3a3} this distribution is integrable and the maximal leaves of the corresponding foliation of $\P E$ by complex analytic subvarieties are described by Proposition \ref{6e3}.
       \end{rem}
       
    The restrictions $E\big|_{\pi (W)}$ being flat, the monodromy is discrete.  Heuristic arguments   suggest that the maximal leaves $W\subset \P E$ are   \emph{closed} analytic subvarieties.  
   \begin{Conj} \lab{6e6} Finite monodromy provides the necessary and sufficient condition to have the equality
   \[
   \kappa(E) =n(E)\]
   of Kodaira-Iitaka and numerical dimensions of a holomorphic \vb\ having a Hermitian metric whose curvature satisfies the norm positivity condition.
   \end{Conj}
   
   The idea is that the quotient $\P E/\sim$, where $\sim$ is the equivalence relation given by the connected components of the foliation defined by \pref{6e1}, exists as a complex analytic variety of dimension equal to $n(E)$, and there is a meromorphic mapping
   \[
   \P E\dashrightarrow \P E/\sim\]
   together with an ample line bundle on $\P E/\sim$ that pulls back to $\cope(1)$.  The rather simple guiding model here is the dual of the  universal sub-bundle  over the Grassmannian that was discussed above.  In fact, the conjecture holds if $E\to X$ is globally generated with metrics induced from the corresponding mapping to a Grassmannian.
 
 We note    that the foliation defined by the null space of the holomorphic bi-sectional curvature on quotients of bounded symmetric domains has been studied in \cite{Mok87}.  In this case the leaves are generally not closed.

 Finally we point out the very interesting papers \cite{CD17a} and \cite{CD17b}.  In these papers the authors construct examples of   smooth fibrations
 \[
 f:X\to B\]
 of a surface over a curve such that for $E=f_\ast \om_{X/B}$ one has
 \[
 E=A\oplus Q\]
 where $A$ is an ample vector bundle and $Q$ is a flat $\cU(m,\C)$-bundle with infinite monodromy group.\footnote{We note that $Q\subset f_\ast\om_{X/B}\subset R^1_f\C_X$ is \emph{not} flat relative to the Gauss-Manin convention on $R^1_f\C_X$.}
 In this case the leaves of the EDS \pref{6e1} may be described as follows: For each $b\in B$ we have
 \[
 \P Q^\ast_b \subset \P E^\ast_b\]
 and using the flat connection on $Q^\ast$ the parallel translate of any point in $\P Q^\ast_b$ defines an integral curve of  the EDS.
\providecommand{\bysame}{\leavevmode\hbox to3em{\hrulefill}\thinspace}
\providecommand{\noopsort}[1]{}
\providecommand{\mr}[1]{\href{http://www.ams.org/mathscinet-getitem?mr=#1}{MR~#1}}
\providecommand{\zbl}[1]{\href{http://www.zentralblatt-math.org/zmath/en/search/?q=an:#1}{Zbl~#1}}
\providecommand{\jfm}[1]{\href{http://www.emis.de/cgi-bin/JFM-item?#1}{JFM~#1}}
\providecommand{\arxiv}[1]{\href{http://www.arxiv.org/abs/#1}{arXiv~#1}}
\providecommand{\MR}{\relax\ifhmode\unskip\space\fi MR }
\providecommand{\MRhref}[2]{%
  \href{http://www.ams.org/mathscinet-getitem?mr=#1}{#2}
}
\providecommand{\href}[2]{#2}


\begin{thebibliography}{GGLR17}

\bibitem[AK00]{AK00}
\bgroup\scshape{}D.~Abramovich\egroup{} and \bgroup\scshape{}K.~Karu\egroup{},
  Weak semistable reduction in characteristic 0,  \emph{Invent. Math.}
  \textbf{139} no.~2 (2000), 241--273. \mr{1738451}.  \zbl{0958.14006}.
  Available at {https://doi.org/10.1007/s002229900024}.

\bibitem[Ara71]{Ara} A. Arakelev, Formulas of algebraic curves with fixed degeneracies, \emph{Izv. Akad. Nauk. SSSR Soc. Math.} {\bf 35} (1971), 1277--1302.

\bibitem[BK77]{BeKa77}
\bgroup\scshape{}E.~Bedford\egroup{} and \bgroup\scshape{}M.~Kalka\egroup{},
  Foliations and complex {M}onge-{A}mp\`ere equations,  \emph{Comm. Pure Appl.
  Math.} \textbf{30} no.~5 (1977), 543--571. \mr{0481107}.  \zbl{0351.35063}.
  Available at {https://doi.org/10.1002/cpa.3160300503}.

\bibitem[Ber09]{Ber09}
\bgroup\scshape{}B.~Berndtsson\egroup{}, Curvature of vector bundles associated
  to holomorphic fibrations,  \emph{Ann. of Math. (2)} \textbf{169} no.~2
  (2009), 531--560. \mr{2480611}.  \zbl{1195.32012}.  Available at
  {https://doi.org/10.4007/annals.2009.169.531}.

\bibitem[BP{\u a}12]{BerPau12}
\bgroup\scshape{}B.~Berndtsson\egroup{} and
  \bgroup\scshape{}M.~P{\u a}un\egroup{}, Quantitative extensions of
  pluricanonical forms and closed positive currents,  \emph{Nagoya Math. J.}
  \textbf{205} (2012), 25--65. \mr{2891164}.  \zbl{1248.32012}.  Available at
  {http://projecteuclid.org/euclid.nmj/1330611001}.

\bibitem[Bru16a]{Bru16b}
\bgroup\scshape{}Y.~Brunebarbe\egroup{}, A strong hyperbolicity property of
  locally symmetric varieties, 2016. \arxiv{1606.03972}.

\bibitem[Bru16b]{Bru16}
 \bysame, Symmetric differentials and variations
  of {H}odge structures, 2016, {\em J. reine angew. Math.}, published on-line.
  {https://doi.org/10.1515/crelle-2015-0109}.

\bibitem[BKT13]{BKT14}
\bgroup\scshape{}Y.~Brunebarbe\egroup{}, \bgroup\scshape{}B.~Klingler\egroup{},
  and \bgroup\scshape{}B.~Totaro\egroup{}, Symmetric differentials and the
  fundamental group,  \emph{Duke Math. J.} \textbf{162} no.~14 (2013),
  2797--2813. \mr{3127814}.  \zbl{1296.32003}.  Available at
  {https://doi.org/10.1215/00127094-2381442}. 

\bibitem[CMSP17]{CM-SP}
\bgroup\scshape{}J.~Carlson\egroup{},
  \bgroup\scshape{}S.~M\"uller-Stach\egroup{}, and
  \bgroup\scshape{}C.~Peters\egroup{}, \emph{Period mappings and period
  domains}, \emph{Cambridge Studies in Advanced Mathematics} \textbf{85}, 2nd edition,
  Cambridge University Press, Cambridge, 2017. \mr{2012297}.  \zbl{1030.14004}.

\bibitem[Car80]{Car80}
\bgroup\scshape{}J.~A. Carlson\egroup{}, Extensions of mixed {H}odge
  structures,  in \emph{Journ\'ees de {G}\'eometrie {A}lg\'ebrique d'{A}ngers,
  {J}uillet 1979/{A}lgebraic {G}eometry, {A}ngers, 1979}, Sijthoff \&\
  Noordhoff, Alphen aan den Rijn---Germantown, Md., 1980, pp.~107--127.
  \mr{0605338}.  \zbl{0471.14003}.  

\bibitem[dCM09]{deC-Mig}
\bgroup\scshape{}M.~A.~A. de~Cataldo\egroup{} and
  \bgroup\scshape{}L.~Migliorini\egroup{}, The decomposition theorem, perverse
  sheaves and the topology of algebraic maps,  \emph{Bull. Amer. Math. Soc.
  (N.S.)} \textbf{46} no.~4 (2009), 535--633. \mr{2525735}.  \zbl{1181.14001}.
  Available at {https://doi.org/10.1090/S0273-0979-09-01260-9}.

\bibitem[CD17a]{CD17a} F. Catanese and M. Dettweiler, Answer to a question by Fujita on variation of Hodge structures, in 
\emph{Higher Dimensional Algebraic Geometry---In Honour of Professor Yujiro Kawamata's Sixtieth Birthday}, 73--102, \emph{Adv. Stud. Pure Math.} {\bf 74}, Math. Soc. Japan, Tokyo, 2017.

\bibitem[CD17b]{CD17b} \bysame,

\bibitem[CDK95]{CDK95}
\bgroup\scshape{}E.~Cattani\egroup{}, \bgroup\scshape{}P.~Deligne\egroup{}, and
  \bgroup\scshape{}A.~Kaplan\egroup{}, On the locus of {H}odge classes,
  \emph{J. Amer. Math. Soc.} \textbf{8} no.~2 (1995), 483--506. \mr{1273413}.
  \zbl{0851.14004}.  Available at {https://doi.org/10.2307/2152824}.

\bibitem[CKS86]{CKS86}
\bgroup\scshape{}E.~Cattani\egroup{}, \bgroup\scshape{}A.~Kaplan\egroup{}, and
  \bgroup\scshape{}W.~Schmid\egroup{}, Degeneration of {H}odge structures,
  \emph{Ann. of Math. (2)} \textbf{123} no.~3 (1986), 457--535. \mr{0840721}.
  \zbl{0617.14005}.  Available at {https://doi.org/10.2307/1971333}.

\bibitem[Del70]{Del}
\bgroup\scshape{}P.~Deligne\egroup{}, \emph{\'Equations diff\'erentielles \`a
  points singuliers r\'eguliers}, \emph{Lecture Notes in Mathematics, Vol.
  163}, Springer-Verlag, Berlin-New York, 1970. \mr{0417174}.
  \zbl{0244.14004}.


\bibitem[Dem12a]{Dem12a}
\bgroup\scshape{}J.-P. Demailly\egroup{}, \emph{Analytic methods in algebraic
  geometry}, \emph{Surveys of Modern Mathematics} \textbf{1}, International
  Press, Somerville, MA; Higher Education Press, Beijing, 2012. \mr{2978333}.
  \zbl{1271.14001}.

\bibitem[Dem12b]{Dem12}
\bgroup\bysame\egroup{}, Hyperbolic algebraic varieties and
  holomorphic differential equations,  \emph{Acta Math. Vietnam.} \textbf{37}
  no.~4 (2012), 441--512. \mr{3058660}.  \zbl{1264.32022}.
  
  \bibitem[Dem16]{Dem16}
 \bysame, \emph{Variational approach for
  complex {M}ondge-{A}mp\`ere equations and geometric applications [after
  {B}erman, {B}erndtsson, {B}oucksom, {E}yssidieux, {G}uedj, {J}onsson,
  {Z}eriahi,\dots]}, \textbf{2015--2016} Sem. Bourbaki no. 1112, 2016.

\bibitem[Den18]{Den18} Ya Deng, \emph{Pseudo Kobayashi hyperbolicity of base spaces of families of minimal projective manifolds with maximal variation}, arXiv:1809.005891v1 [math.AG] 16Sep2018.

\bibitem[Don15]{Don2}
\bgroup\scshape{}S.~Donaldson\egroup{}, Stability of algebraic varieties and
  {K}\"ahler geometry, to appear in \emph{Proc. Amer. Math. Soc. Sumer School
  in Algebraic Geometry} (Salt Lake City 2015).

\bibitem[Don16]{Don1}
 \bysame, K\"ahler-{E}instein metrics and
  algebraic geometry,  in \emph{Current developments in mathematics 2015}, Int.
  Press, Somerville, MA, 2016, pp.~1--25. \mr{3642542}.  \zbl{06751629}.


  
\bibitem[FPR15a]{FPR15a}
\bgroup\scshape{}M.~Franciosi\egroup{}, \bgroup\scshape{}R.~Pardini\egroup{},
  and \bgroup\scshape{}S.~Rollenske\egroup{}, Computing invariants of
  semi-log-canonical surfaces,  \emph{Math. Z.} \textbf{280} no.~3-4 (2015),
  1107--1123. \mr{3369370}.  \zbl{1329.14076}.  Available at
  {https://doi.org/10.1007/s00209-015-1469-9}.

\bibitem[FPR15b]{FPR15b}
 \bysame, Log-canonical pairs and
  {G}orenstein stable surfaces with {$K_X^2=1$},  \emph{Compos. Math.}
  \textbf{151} no.~8 (2015), 1529--1542. \mr{3383166}.  \zbl{1331.14037}.
  Available at {https://doi.org/10.1112/S0010437X14008045}.

\bibitem[FPR17]{FPR15c}
 \bysame, Gorenstein stable surfaces with
  {$K^2_X\!=\!1$} and {$p_g\!>\!0$},  \emph{Math.~Nachr.} \textbf{290} no.~5-6 (2017),
  794--814. \mr{3636379}.  \zbl{06717887}.  Available at
 https://doi.org/10.1002/ mana.201600090.

\bibitem[Fuj78]{Fuj78}
\bgroup\scshape{}T.~Fujita\egroup{}, On {K}\"ahler fiber spaces over curves,
  \emph{J. Math. Soc. Japan} \textbf{30} no.~4 (1978), 779--794. \mr{0513085}.
  \zbl{0393.14006}.  Available at {https://doi.org/10.2969/jmsj/03040779}.
  
  \bibitem[GGK08]{GGK08} M. Green, P. Griffiths and M. Kerr, Some enumerative global properties of variations of Hodge structure, \emph{Moscow Math.} {\bf 9} (2009), 469--530.

\bibitem[GGLR17]{GGLR17}
\bgroup\scshape{}M.~Green\egroup{}, \bgroup\scshape{}P.~Griffiths\egroup{},
  \bgroup\scshape{}R.~Laza\egroup{}, and \bgroup\scshape{}C.~Robles\egroup{},
  Completion of period mappings and ampleness of the {H}odge bundle, 2017.
  \arxiv{1708.09523v1}.

\bibitem[Gri69]{Gri69}
\bgroup\scshape{}P.~A. Griffiths\egroup{}, Hermitian differential geometry,
  {C}hern classes, and positive vector bundles,  in \emph{Global {A}nalysis
  ({P}apers in {H}onor of {K}. {K}odaira)}, Univ. Tokyo Press, Tokyo, 1969,
  pp.~185--251. \mr{0258070}.

\bibitem[Gri70]{Gri72}
 \bysame, Periods of integrals on algebraic
  manifolds. {III}. {S}ome global differential-geometric properties of the
  period mapping,  \emph{Inst. Hautes \'Etudes Sci. Publ. Math.} no.~38 (1970),
  125--180. \mr{0282990}.  \zbl{0212.53503}.  

\bibitem[Kaw81]{Kaw81}
\bgroup\scshape{}Y.~Kawamata\egroup{}, Characterization of abelian varieties,
  \emph{Compositio Math.} \textbf{43} no.~2 (1981), 253--276. \mr{0622451}.
  \zbl{0471.14022}. 

\bibitem[Kaw82]{Kaw82}
 \bysame, Kodaira dimension of algebraic fiber
  spaces over curves,  \emph{Invent. Math.} \textbf{66} no.~1 (1982), 57--71.
  \mr{0652646}.  \zbl{0461.14004}.  Available at
  {https://doi.org/10.1007/BF01404756}.

\bibitem[Kaw83]{Kaw83}
 \bysame, Kodaira dimension of certain algebraic
  fiber spaces,  \emph{J. Fac. Sci. Univ. Tokyo Sect. IA Math.} \textbf{30}
  no.~1 (1983), 1--24. \mr{0700593}.  \zbl{0516.14026}.

\bibitem[Kaw85]{Kaw85}
 \bysame, Minimal models and the {K}odaira
  dimension of algebraic fiber spaces,  \emph{J. Reine Angew. Math.}
  \textbf{363} (1985), 1--46. \mr{0814013}.  \zbl{0589.14014}.  Available at
  {https://doi.org/10.1515/crll.1985.363.1}.

\bibitem[Kaw02]{Kaw02}
 \bysame, On algebraic fiber spaces,  in
  \emph{Contemporary trends in algebraic geometry and algebraic topology
  ({T}ianjin, 2000)}, \emph{Nankai Tracts Math.} \textbf{5}, World Sci. Publ.,
  River Edge, NJ, 2002, pp.~135--154. \mr{1945358}.  \zbl{1080.14507}.
  Available at  https://doi.org/10.1142/9789812777416$_{}$0006.

\bibitem[Kol87]{Kol87}
\bgroup\scshape{}J.~Koll\'ar\egroup{}, Subadditivity of the {K}odaira
  dimension: fibers of general type,  in \emph{Algebraic geometry, {S}endai,
  1985}, \emph{Adv. Stud. Pure Math.} \textbf{10}, North-Holland, Amsterdam,
  1987, pp.~361--398. \mr{0946244}.  \zbl{0659.14024}.

\bibitem[Kol13]{Kol13}
 \bysame, Moduli of varieties of general type,  in
  \emph{Handbook of moduli. {V}ol. {II}}, \emph{Adv. Lect. Math. (ALM)}
  \textbf{25}, Int. Press, Somerville, MA, 2013, pp.~131--157. \mr{3184176}.
  \zbl{1322.14006}.
  
  \bibitem[Ko-Mo]{Ko-Mo} J. Koll\'ar and S. Mori, Birational geometry of algebraic varieties, \emph{Cambridge Tracts in Math.} {\bf  134}, Cambridge Univ. Press, 1998.

\bibitem[Laz04]{Laz04}
\bgroup\scshape{}R.~Lazarsfeld\egroup{}, \emph{Positivity in algebraic
  geometry. {II}}, \emph{Ergebnisse der Mathematik und ihrer Grenzgebiete. 3.
  Folge. A Series of Modern Surveys in Mathematics [Results in Mathematics and
  Related Areas. 3rd Series. A Series of Modern Surveys in Mathematics]}
  \textbf{49}, Springer-Verlag, Berlin, 2004, Positivity for vector bundles,
  and multiplier ideals. \mr{2095472}.  \zbl{1093.14500}.  Available at
  {https://doi.org/10.1007/978-3-642-18808-4}.

\bibitem[Mok87]{Mok87}
\bgroup\scshape{}N.~Mok\egroup{}, Uniqueness theorems of {H}ermitian metrics of
  seminegative curvature on quotients of bounded symmetric domains,  \emph{Ann.
  of Math. (2)} \textbf{125} no.~1 (1987), 105--152. \mr{0873379}.
  \zbl{0616.53040}.  Available at {https://doi.org/10.2307/1971290}.
  
  \bibitem[Mok12]{Mok12} \bysame, Projective-algebraisity of minimal compactifications of complex hyperbolic space forms of finite volume, \emph{Perspectives in Analysis, Geometry and Topology}, Birkh\"auser/Springer, NY (2012), 331--354.

\bibitem[MT07]{MorTak07}
\bgroup\scshape{}C.~Mourougane\egroup{} and
  \bgroup\scshape{}S.~Takayama\egroup{}, Hodge metrics and positivity of direct
  images,  \emph{J. Reine Angew. Math.} \textbf{606} (2007), 167--178.
  \mr{2337646}.  \zbl{1128.14030}.  Available at
  {https://doi.org/10.1515/CRELLE.2007.039}.

\bibitem[MT08]{MorTak08}
 \bysame, Hodge metrics and the curvature of
  higher direct images,  \emph{Ann. Sci. \'Ec. Norm. Sup\'er. (4)} \textbf{41}
  no.~6 (2008), 905--924. \mr{2504108}.  \zbl{1167.14027}.  Available at
  {https://doi.org/10.24033/asens.2084}.

\bibitem[NS68]{NarSim68}
\bgroup\scshape{}M.~S. Narasimhan\egroup{} and \bgroup\scshape{}R.~R.
  Simha\egroup{}, Manifolds with ample canonical class,  \emph{Invent. Math.}
  \textbf{5} (1968), 120--128. \mr{0236960}.  \zbl{0159.37902}.  Available at
  {https://doi.org/10.1007/BF01425543}.

\bibitem[PS08]{PeSt08}
\bgroup\scshape{}C.~A.~M. Peters\egroup{} and \bgroup\scshape{}J.~H.~M.
  Steenbrink\egroup{}, \emph{Mixed {H}odge structures}, \emph{Ergebnisse der
  Mathematik und ihrer Grenzgebiete. 3. Folge. A Series of Modern Surveys in
  Mathematics [Results in Mathematics and Related Areas. 3rd Series. A Series
  of Modern Surveys in Mathematics]} \textbf{52}, Springer-Verlag, Berlin,
  2008. \mr{2393625}.  \zbl{1138.14002}.

\bibitem[P\u{a}16]{Pau16}
\bgroup\scshape{}M.~P\u{a}un\egroup{}, Singular {H}ermitian metrics and
  positivity of direct images of pluricanonical bundles, 2016.
  \arxiv{1606.00174}.

\bibitem[PT14]{PauTak14}
\bgroup\scshape{}M.~P\u{a}un\egroup{} and
  \bgroup\scshape{}S.~Takayama\egroup{}, Positivity of twisted relative
  pluricanonical bundles and their direct images, 2014. \arxiv{1409.5505v1}.

\bibitem[Sch15]{Sch15}
\bgroup\scshape{}C.~Schnell\egroup{}, Weak positivity via mixed {H}odge
  modules,  in \emph{Hodge theory and classical algebraic geometry},
  \emph{Contemp. Math.} \textbf{647}, Amer. Math. Soc., Providence, RI, 2015,
  pp.~129--137. \mr{3445002}.  \zbl{1360.14027}.  Available at
  {https://doi.org/10.1090/conm/647/12957}.
  
  \bibitem[Sim92]{Sim92}
???  Simpson, Higgs bundles and local systems, \emph{Publ.\ Math.\ IHES} {\bf 75} (1992), 5--95.

\bibitem[Som59]{Som59}
\bgroup\scshape{}F.~Sommer\egroup{}, Komplex-analytische {B}l\"atterung reeler
  {H}yperfl\"achen im {$C\sp{n}$},  \emph{Math. Ann.} \textbf{137} (1959),
  392--411. \mr{0108821}.  \zbl{0092.29903}.  Available at
  {https://doi.org/10.1007/BF01360840}.

\bibitem[Som78]{Som78}
\bgroup\scshape{}A.~J. Sommese\egroup{}, On the rationality of the period
  mapping,  \emph{Ann. Scuola Norm. Sup. Pisa Cl. Sci. (4)} \textbf{5} no.~4
  (1978), 683--717. \mr{0519890}.  \zbl{0396.32004}.  

\bibitem[Uen74]{Uen75}
\bgroup\scshape{}K.~Ueno\egroup{}, Introduction to classification theory of
  algebraic varieties and compact complex spaces,  in \emph{Classification of
  algebraic varieties and compact complex manifolds}, \emph{Lecture Notes in
  Math.} \textbf{412}, Springer, Berlin, 1974, pp.~288--332. \mr{0361174}.

\bibitem[Uen78]{Uen77}
 \bysame, Classification of algebraic varieties. {II}.
  {A}lgebraic threefolds of parabolic type,  in \emph{Proceedings of the
  {I}nternational {S}ymposium on {A}lgebraic {G}eometry ({K}yoto {U}niv.,
  {K}yoto, 1977)}, Kinokuniya Book Store, Tokyo, 1978, pp.~693--708.
  \mr{0578882}.

\bibitem[Vie83a]{Vie83a}
\bgroup\scshape{}E.~Viehweg\egroup{}, Weak positivity and the additivity of the
  {K}odaira dimension for certain fibre spaces,  in \emph{Algebraic varieties
  and analytic varieties ({T}okyo, 1981)}, \emph{Adv. Stud. Pure Math.}
  \textbf{1}, North-Holland, Amsterdam, 1983, pp.~329--353. \mr{0715656}.
  \zbl{0513.14019}.

\bibitem[Vie83b]{Vie83b}
 \bysame, Weak positivity and the additivity of the
  {K}odaira dimension. {II}. {T}he local {T}orelli map,  in
  \emph{Classification of algebraic and analytic manifolds ({K}atata, 1982)},
  \emph{Progr. Math.} \textbf{39}, Birkh\"auser Boston, Boston, MA, 1983,
  pp.~567--589. \mr{0728619}.  \zbl{0543.14006}.
  
  \bibitem[VZ03]{VZ03} E.~Viehweg and K. Zuo, On the Brady hyperbolicity of moduli spaces for canonically polarized manifolds, \emph{Duke Math. J.} {\bf 118} (2003), 103--150.
  
  \bibitem[VZ06]{VZ}  \bysame, Numerical bounds for semi-stable families of curves or of certain higher dimensional manifolds, \emph{J. Diff. Geom.} {\bf 15} (2006), 291--352.

\bibitem[Zuo00]{Zuo}
\bgroup\scshape{}K.~Zuo\egroup{}, On the negativity of kernels of
  {K}odaira-{S}pencer maps on {H}odge bundles and applications,  \emph{Asian J.
  Math.} \textbf{4} no.~1 (2000), 279--301, Kodaira's issue. \mr{1803724}.
  \zbl{0983.32020}.  Available at
  {https://doi.org/10.4310/AJM.2000.v4.n1.a17}.

\end{thebibliography}
     \end{document}